\documentclass[12pt,reqno]{amsart}
\usepackage{amsmath,amsthm,amsfonts,amssymb,times}
\usepackage{verbatim}
\usepackage{url}
\setbox0=\hbox{$+$}
\newdimen\plusheight
\plusheight=\ht0
\def\+{\;\lower\plusheight\hbox{$+$}\;}

\setbox0=\hbox{$-$}
\newdimen\minusheight
\minusheight=\ht0
\def\-{\;\lower\minusheight\hbox{$-$}\;}

\setbox0=\hbox{$\cdots$}
\newdimen\cdotsheight
\cdotsheight=\plusheight
\def\cds{\lower\cdotsheight\hbox{$\cdots$}}

\makeatletter
\def\leqalignno#1{\displ@y \tabskip\z@ plus\@ne fil
  \halign to\displaywidth{\hfil$\@lign\displaystyle{##}$\tabskip\z@skip
    &$\@lign\displaystyle{{}##}$\hfil\tabskip\z@ plus\@ne fil
    &\kern-\displaywidth\rlap{$\@lign\hbox{\rm##}$}\tabskip\displaywidth\crcr
    #1\crcr}}
\makeatother

\newcommand{\eb}{\begin{equation}}
\newcommand{\ee}{\end{equation}}

\newcommand{\df}{\dfrac}

 \renewcommand{\a}{\alpha}
\renewcommand{\b}{\beta}

\renewcommand{\o}{{\omega}}

\newcommand{\x}{\xi}

\renewcommand{\(}{\left\(}
\renewcommand{\)}{\right\)}
\renewcommand{\[}{\left\[}
\renewcommand{\]}{\right\]}

\renewcommand{\pmod}[1]{\,(\textup{mod}\,#1)}
\numberwithin{equation}{section}
 \theoremstyle{plain}

\newtheorem{theorem}{Theorem}[section]
\numberwithin{equation}{section}
\theoremstyle{plain}
\newtheorem{lemma}[theorem]{Lemma}
\newtheorem{corollary}[theorem]{Corollary}

\newtheorem{example}[theorem]{Example}

\allowdisplaybreaks

\begin{document}

\title[Evaluations and Reciprocity Theorems for Finite Trigonometric Sums]
{Exact Evaluations and Reciprocity Theorems for Finite Trigonometric Sums}
\author{Bruce C.~Berndt, Sun Kim, Alexandru Zaharescu}
\address{Department of Mathematics, University of Illinois, 1409 West Green
Street, Urbana, IL 61801, USA} \email{berndt@illinois.edu}
\address{Department of Mathematics, and Institute of Pure and Applied Mathematics, Jeonbuk National
University, 567 Baekje-daero, Jeonju-si, Jeollabuk-do 54896, Republic of Korea}
\email{sunkim@jbnu.ac.kr}
\address{Department of Mathematics, University of Illinois, 1409 West Green
Street, Urbana, IL 61801, USA; Institute of Mathematics of the Romanian
Academy, P.O.~Box 1-764, Bucharest RO-70700, Romania}
\email{zaharesc@illinois.edu}

\begin{abstract} We evaluate in closed form several classes of finite trigonometric sums. Two general methods are used.  The first is new and involves sums of roots of unity.  The second uses contour integration and extends a previous method used by two of the authors.  Reciprocity theorems for certain trigonometric sums are also established.
\end{abstract}


\maketitle

\tableofcontents

\section{Introduction}
Sums involving the
exponential function or trigonometric functions have been used extensively in
mathematics. In particular, they appear in many problems in number theory. In most instances, the
relevant sums cannot be evaluated in closed form, and in such cases one is content to provide either upper or
 lower bounds for those sums. Here, we focus on certain classes of trigonometric  sums with the aim of obtaining exact evaluations

Another important goal of the paper is to provide
reciprocity theorems for certain trigonometric sums. Besides their
intrinsic elegance, these
reciprocity theorems can also be interpreted as a decisive step toward
the exact evaluation of the corresponding sums.  This is similar to how the law of quadratic reciprocity can be applied
repeatedly in order to compute the Legendre symbol in any given concrete case.

  We continue the studies of trigonometric sum evaluations that were made in earlier papers \cite{bbyz}, \cite{yeap}, \cite{mathann} by two of the present authors and their collaborators.  Since then, several other authors, for example, C.~M.~ da Fonseca,  M.~L.~Glasser, and V.~ Kowalenko \cite{fgk}, \cite{fgk1}, X. ~Wang and D.~Y. Zheng \cite{wangzheng}, and Y.~He \cite{he}, have made excellent contributions to the study and evaluations of trigonometric sums.  Later in this paper, we give more details, especially in regard to the trigonometric sums studied here.

We introduce a new method for evaluating trigonometric sums.  The method features logarithmic differentiation, partial fractions, and the  evaluations of several necessary sums of roots of unity.  Although our primary motivation was to use these identities to evaluate  trigonometric sums, particularly, involving $\sin$'s and $\cos$'s, we think that the theorems involving roots of unity are interesting in themselves and potentially useful in other investigations.  Several examples are given to illustrate our method.  It should be emphasized that some of these evaluations were previously established in the literature by other means.  Nonetheless, our methods are new and perhaps more elementary.

Our elementary approach rests upon four elementary logarithmic differentiations. It  can be extended with the use of $n$ successive logarithmic differentiations, for any natural number $n$.  However, the calculations become increasingly laborious, and a computer algebra system, such as \emph{Mathematica}, would be advisable.

One of our primary goals is to evaluate sums, in which the trigonometric functions have two distinct periods.  To the best of our knowledge, there are no previous explicit evaluations of trigonometric  sums with two distinct periods in the literature.  However, Richard McIntosh \cite[p.~202]{mcintosh} expressed one such sum in terms of generalized Dedekind sums; see \eqref{mc6} below.  Although his representation is not an `evaluation' in the sense we use the term in this paper, his result was of enormous importance in motivating us to do the research in our paper.

Reciprocity theorems for sums with two distinct periods are found in the literature. In particular, one of the reciprocity theorems involving cotangents is equivalent to the reciprocity theorem for Dedekind sums \cite{rademacher}, \cite{rademachergrosswald}.  See also a paper of S.~Fukuhara \cite{fukuhara}, in which he obtains reciprocity theorems for cotangent sums connected with certain generalized Dedekind sums.  We establish some apparently new reciprocity theorems for other cotangent sums.  We indicate how to prove reciprocity theorems for
 $2n$ sums with $2n$ distinct periods, although we work out the details for only $n=1,2$.

\section{Elementary Preliminary Lemmas}

\begin{lemma}\label{1zi} Let $k$ be an odd positive integer and  $z_n=e^{2\pi in/k}.$
If $k\equiv 1 \pmod{4},$ then

 $\displaystyle\sum_{n=1}^k\frac{1}{z_n\pm i}=\pm\frac{k}{i\pm 1},$

$\displaystyle\sum_{n=1}^k\frac{1}{(z_n\pm i)^2}=\frac{k(k-1)i}{i\pm 1}\mp \frac{k^2i}{2},$

 $\displaystyle\sum_{n=1}^k\frac{1}{(z_n\pm i)^3}=
\frac{\mp k(k-1)(k-2)-k^3i}{2(i\pm 1)}+\frac{3k^2(k-1)}{4},$

 $\displaystyle\sum_{n=1}^k\frac{1}{(z_n\pm i)^4}=
-\frac{k(k-1)(k-2)(k-3)i}{6(i\pm 1)}\pm\frac{k^3(k-1)}{i\pm 1}\pm \frac{k^2(k-1)(7k-11)i}{12}-\frac{k^4}{4}.$

\medskip

If $k\equiv 3 \pmod{4},$ then we have

 $\displaystyle\sum_{n=1}^k\frac{1}{z_n\pm i}=\pm\frac{k}{i\mp 1},$

$\displaystyle\sum_{n=1}^k\frac{1}{(z_n\pm i)^2}=\frac{k(k-1)i}{i\mp 1}\pm \frac{k^2i}{2},$

 $\displaystyle\sum_{n=1}^k\frac{1}{(z_n\pm i)^3}=
\frac{\mp k(k-1)(k-2)+k^3i}{2(i\mp 1)}-\frac{3k^2(k-1)}{4},$

 $\displaystyle\sum_{n=1}^k\frac{1}{(z_n\pm i)^4}=
-\frac{k(k-1)(k-2)(k-3)i}{6(i\mp 1)}\mp\frac{k^3(k-1)}{i\mp 1}\mp \frac{k^2(k-1)(7k-11)i}{12}-\frac{k^4}{4}.$
\end{lemma}

\begin{proof}
Defining $f(x):=x^k-1=(x-z_1)\cdots(x-z_k),$  we obtain the identities,
\begin{equation}\label{f}
  \begin{split}
\sum_{n=1}^k\frac{1}{x-z_n}&=\frac{f'}{f}=\frac{k x^{k-1}}{x^k-1}, \\
\sum_{n=1}^k\frac{1}{(x-z_n)^2}&=-\Big(\frac{f'}{f}\Big)'=
-\frac{k(k-1)x^{k-2}}{x^k-1}+\frac{k^2x^{2k-2}}{(x^k-1)^2}, \\
\sum_{n=1}^k\frac{1}{(x-z_n)^3}&=\frac12\Big(\frac{f'}{f}\Big)''
=\frac{k(k-1)(k-2)x^{k-3}}{2(x^k-1)}-\frac{3k^2(k-1)x^{2k-3}}{2(x^k-1)^2}+\frac{k^3x^{3k-3}}{(x^k-1)^3},\\
\sum_{n=1}^k\frac{1}{(x-z_n)^4}&=-\frac16\Big(\frac{f'}{f}\Big)'''
=-\frac{k(k-1)(k-2)(k-3)x^{k-4}}{6(x^k-1)}+\frac{k^2(k-1)(7k-11)x^{2k-4}}{6(x^k-1)^2}\\
&\qquad \qquad \qquad -\frac{2k^3(k-1)x^{3k-4}}{(x^k-1)^3}+\frac{k^4x^{4k-4}}{(x^k-1)^4}.
\end{split}
\end{equation}
Setting $x=i$ and  $x=-i$ in these identities, we derive the identities of Lemma \ref{1zi}.
\end{proof}

\begin{lemma}\label{1z1}
Let $k$ be an odd positive integer and  $z_n=e^{2\pi in/k}.$ Then,

\medskip

$\displaystyle\sum_{n=1}^{k-1}\frac{1}{z_n+1}=\frac{k-1}{2}, \qquad\qquad\qquad\qquad \sum_{n=1}^{k-1}\frac{1}{(z_n+1)^2}=-\frac{(k-1)^2}{4},$

$\displaystyle \sum_{n=1}^{k-1}\frac{1}{(z_n+1)^3}=-\frac{(k-1)(3k-1)}{8}, \qquad \sum_{n=1}^{k-1}\frac{1}{(z_n+1)^4}=\frac{(k-1)(k^3+k^2-21k+3)}{48}.$
\end{lemma}

\begin{proof}
Setting $x=-1$ in the identities in \eqref{f} , we can easily derive
\begin{align*}
\sum_{n=1}^{k-1}\frac{1}{z_n+1}&=\sum_{n=1}^{k}\frac{1}{z_n+1}-\frac12=\frac{k-1}{2}, \\
\sum_{n=1}^{k-1}\frac{1}{(z_n+1)^2}&=\sum_{n=1}^{k}\frac{1}{(z_n+1)^2}-\frac14
=-\frac{k(k-1)}{2}+\frac{k^2}{4}-\frac14=-\frac{(k-1)^2}{4}, \\
\sum_{n=1}^{k-1}\frac{1}{(z_n+1)^3}&=\sum_{n=1}^{k}\frac{1}{(z_n+1)^3}-\frac18\\
&=\frac{k(k-1)(k-2)}{4}-\frac{3k^2(k-1)}{8}+\frac{k^3}{8}-\frac18=-\frac{(k-1)(3k-1)}{8},\\
\sum_{n=1}^{k-1}\frac{1}{(z_n+1)^4}&=\sum_{n=1}^{k}\frac{1}{(z_n+1)^4}-\frac{1}{16} \\
&=-\frac{k(k-1)(k-2)(k-3)}{12}+\frac{k^2(k-1)(7k-11)}{24}-\frac{k^3(k-1)}{4}+\frac{k^4}{16}-\frac{1}{16}\\
&=\frac{(k-1)(k^3+k^2-21k+3)}{48}.
\end{align*}
This completes our proof.
\end{proof}

\begin{lemma}\label{1z-1}
Let $k$ be a positive integer and  $z_n=e^{2\pi in/k}.$ Then,

$\displaystyle\sum_{n=1}^{k-1}\frac{1}{z_n-1}=-\frac{k-1}{2}, \qquad\qquad\qquad \sum_{n=1}^{k-1}\frac{1}{(z_n-1)^2}=-\frac{(k-1)(k-5)}{12},$

$\displaystyle\sum_{n=1}^{k-1}\frac{1}{(z_n-1)^3}=\frac{(k-1)(k-3)}{8}, \qquad \sum_{n=1}^{k-1}\frac{1}{(z_n-1)^4}=\frac{(k-1)(k^3+k^2-109k+251)}{720}.$
\end{lemma}

\begin{proof}
Defining
$$g(x):=\dfrac{x^k-1}{x-1}=(x-z_1)\cdots(x-z_{k-1})=\sum_{j=0}^{k-1}x^{j},$$
we find that
$$g(1)=k, \qquad g'(1)=\frac{k(k-1)}{2}, \qquad  g''(1)=\frac{k(k-1)(k-2)}{3},$$
$$g'''(1)=\frac{k(k-1)(k-2)(k-3)}{4}, \qquad g^{(4)}(1)=\frac{k(k-1)(k-2)(k-3)(k-4)}{5}.$$
Therefore, we see that
\begin{align*}
\sum_{n=1}^{k-1}\frac{1}{z_n-1}&=-\left[\frac{g'}{g}\right]_{x=1}=-\frac{k-1}{2}, \\
\sum_{n=1}^{k-1}\frac{1}{(z_n-1)^2}&=-\left[\Big(\frac{g'}{g}\Big)'\right]_{x=1}
=\left[\frac{(g')^2-gg''}{g^2}\right]_{x=1}=-\frac{(k-1)(k-5)}{12}, \\
\sum_{n=1}^{k-1}\frac{1}{(z_n-1)^3}&=-\frac12\left[\Big(\frac{g'}{g}\Big)''\right]_{x=1}=
-\left[\frac{g'''}{2g}-\frac{3g'g''}{2g^2}+\frac{(g')^3}{g^3}\right]_{x=1}
=\frac{(k-1)(k-3)}{8}, \\
\sum_{n=1}^{k-1}\frac{1}{(z_n-1)^4}&=-\frac16\left[\Big(\frac{g'}{g}\Big)'''\right]_{x=1}=
\left[-\frac{g^{(4)}}{6g}+\frac{3(g'')^2+4g'g'''}{6g^2}-\frac{2(g')^2g''}{g^3}+
\frac{(g')^4}{g^4}\right]_{x=1} \\
&=\frac{(k-1)(k^3+k^2-109k+251)}{720}.
\end{align*}
This completes the proof.
\end{proof}

\section{Explicit Evaluations of Several Trigonometric Sums}

\begin{theorem}\label{cc}
Let $k$ denote an odd positive integer. Then
\begin{align*}
\sum_{n=1}^k\frac{\cos^2(\pi n/k)}{\cos^4(2\pi n/k)}=
\begin{cases}
\dfrac{k}{12}(3+4k+3k^2+2k^3), & \text{if} \quad k\equiv 1 \pmod{4}, \\[.15in]
\dfrac{k}{12}(-3+4k-3k^2+2k^3), & \text{if} \quad k\equiv 3 \pmod{4}.
\end{cases}
\end{align*}
\end{theorem}

\begin{proof}
We let
\begin{align*}
S=\sum_{n=1}^k\frac{\cos^2(\pi n/k)}{\cos^4(2\pi n/k)}, \quad  z_n=e^{2\pi i n/k}, \quad \text{and} \quad R(x)=4\frac{x^3(x+1)^2}{(x^2+1)^4}.
\end{align*}
Note that
\begin{align}\label{zcos}
z_n+\frac{1}{z_n}=2\cos(2\pi n/k)=4\cos^2(\pi n/k)-2.
\end{align}
Thus,
\begin{align*}
R(z_n)=4\frac{z_n+2+1/z_n}{(z_n+1/z_n)^4}=\frac{\cos^2(\pi n/k)}{\cos^4(2\pi n/k)}, \quad \text{and}
\quad S=\sum_{n=1}^kR(z_n).
\end{align*}
We can check that $R(x)$  has the following partial fraction decomposition:
\begin{align*}
R(x)&=\frac{i}{4(x+i)}-\frac{1-3i}{4(x+i)^2}+\frac{1+2i}{2(x+i)^3}+\frac{1}{2(x+i)^4} \notag\\
&\quad -\frac{i}{4(x-i)}-\frac{1+3i}{4(x-i)^2}+\frac{1-2i}{2(x-i)^3}+\frac{1}{2(x-i)^4}.
\end{align*}
Therefore, we have
\begin{align*}
S&=\sum_{n=1}^kR(z_n)\\
&=\frac{i}{4}\sum_{n=1}^k\left(\frac{1}{z_n+i}-\frac{1}{z_n-i}\right)-\frac14\sum_{n=1}^k\left(\frac{1}{(z_n+i)^2}+\frac{1}{(z_n-i)^2}\right)\\
&\quad +\frac{3i}{4}\sum_{n=1}^k\left(\frac{1}{(z_n+i)^2}-\frac{1}{(z_n-i)^2}\right)+\frac12
\sum_{n=1}^k\left(\frac{1}{(z_n+i)^3}+\frac{1}{(z_n-i)^3}\right)\\&\quad +i\sum_{n=1}^k\left(\frac{1}{(z_n+i)^3}-\frac{1}{(z_n-i)^3}\right)
+\frac12\sum_{n=1}^k\left(\frac{1}{(z_n+i)^4}+\frac{1}{(z_n-i)^4}\right).
\end{align*}

By Lemma \ref{1zi}, for $k\equiv 1 \pmod{4}$,
\begin{align*}
S=&\frac{ik}{4}\left(\frac{1}{i+1}+\frac{1}{i-1}\right)
-\frac{k(k-1)i}{4}\left(\frac{1}{i+1}+\frac{1}{i-1}\right)\\
&-\frac{3k(k-1)}{4}\left(\frac{1}{i+1}-\frac{1}{i-1}\right)+\frac{3k^2}{4}-\frac{k(k-1)(k-2)}{4}\left(\frac{1}{i+1}-\frac{1}{i-1}\right)\\
&-\frac{k^3i}{4}\left(\frac{1}{i+1}+\frac{1}{i-1}\right)+\frac{3k^2(k-1)}{4}-\frac{k(k-1)(k-2)i}{2}\left(\frac{1}{i+1}+\frac{1}{i-1}\right)\\
&+\frac{k^3}{2}\left(\frac{1}{i+1}-\frac{1}{i-1}\right)-\frac{k(k-1)(k-2)(k-3)i}{12}\left(\frac{1}{i+1}+\frac{1}{i-1}\right)\\
&+\frac{k^3(k-1)}{2}\left(\frac{1}{i+1}-\frac{1}{i-1}\right)-\frac{k^4}{4}.
\end{align*}
Using the identities
\begin{align}\label{i}
\frac{1}{i+1}+\frac{1}{i-1}=-i,  & \quad  \frac{1}{i+1}-\frac{1}{i-1}=1,
\end{align}
we arrive at
\begin{align*}
S&=\frac{k}{4}-\frac{k(k-1)}{4}-\frac{3k(k-1)}{4}+\frac{3k^2}{4}-\frac{k(k-1)(k-2)}{4}-\frac{k^3}{4}+\frac{3k^2(k-1)}{4}\\
&\quad  -\frac{k(k-1)(k-2)}{2}+\frac{k^3}{2}-\frac{k(k-1)(k-2)(k-3)}{12}+\frac{k^3(k-1)}{2}-\frac{k^4}{4}\\
&=\frac{k}{4}+\frac{k^2}{3}+\frac{k^3}{4}+\frac{k^4}{6}.
\end{align*}

Similarly, Lemma \ref{1zi} implies that, for $k\equiv 3 \pmod{4},$
\begin{align*}
S=&\frac{ik}{4}\left(\frac{1}{i-1}+\frac{1}{i+1}\right)
-\frac{k(k-1)i}{4}\left(\frac{1}{i-1}+\frac{1}{i+1}\right)\\
&-\frac{3k(k-1)}{4}\left(\frac{1}{i-1}-\frac{1}{i+1}\right)-\frac{3k^2}{4}\\
&-\frac{k(k-1)(k-2)}{4}\left(\frac{1}{i-1}-\frac{1}{i+1}\right)+\frac{k^3i}{4}\left(\frac{1}{i-1}+\frac{1}{i+1}\right)
-\frac{3k^2(k-1)}{4}\\
&-\frac{k(k-1)(k-2)i}{2}\left(\frac{1}{i-1}+\frac{1}{i+1}\right)
-\frac{k^3}{2}\left(\frac{1}{i-1}-\frac{1}{i+1}\right)\\
&-\frac{k(k-1)(k-2)(k-3)i}{12}\left(\frac{1}{i-1}+\frac{1}{i+1}\right)-\frac{k^3(k-1)}{2}\left(\frac{1}{i-1}-\frac{1}{i+1}\right)
-\frac{k^4}{4}.
\end{align*}
By  \eqref{i}, we obtain
\begin{align*}
S&=\frac{k}{4}-\frac{k(k-1)}{4}+\frac{3k(k-1)}{4}-\frac{3k^2}{4}+\frac{k(k-1)(k-2)}{4}+\frac{k^3}{4}-\frac{3k^2(k-1)}{4}\\
&\quad  -\frac{k(k-1)(k-2)}{2}+\frac{k^3}{2}-\frac{k(k-1)(k-2)(k-3)}{12}+\frac{k^3(k-1)}{2}-\frac{k^4}{4}\\
&=-\frac{k}{4}+\frac{k^2}{3}-\frac{k^3}{4}+\frac{k^4}{6}.
\end{align*}
This completes our proof.
\end{proof}

\begin{example} Let $k=5$.  An easy calculation shows that
$$\left(\df{\cos\frac{\pi}{5}}{2\cos^2\frac{2\pi}{5}}\right)^2
+\left(\df{\cos\frac{2\pi}{5}}{2\cos^2\frac{\pi}{5}}\right)^2=18.
$$
\end{example}

\begin{example}
Let $k=7$.  Then, after a moderate computation, we find that
$$\left(\df{\cos\frac{\pi}{7}}{2\cos^2\frac{2\pi}{7}}\right)^2
+\left(\df{\cos\frac{2\pi}{7}}{2\cos^2\frac{3\pi}{7}}\right)^2
+\left(\df{\cos\frac{3\pi}{7}}{2\cos^2\frac{\pi}{7}}\right)^2=41.
$$
\end{example}
This past example was also established by \"{O}rs Reb\'{a}k \cite{rebakseptic} in the course of solving a longstanding open problem in Ramanujan's lost notebook \cite{lnb}.

\begin{theorem}\label{1cos4}
If $k$ is an odd positive integer, then
\begin{align*}
\sum_{n=1}^k\frac{1}{\cos^4(2\pi n/k)}=\sum_{n=1}^{k}\frac{1}{\cos^4(\pi n/k)}=\frac{k^2(k^2+2)}{3}.
\end{align*}
\end{theorem}

\begin{proof}
Since $k$ is odd, we can easily check that
$$\sum_{n=1}^k\frac{1}{\cos^4(2\pi n/k)}=\sum_{n=1}^{k}\frac{1}{\cos^4(\pi n/k)}.$$
Similarly, we let
\begin{align*}
z_n=e^{2\pi i n/k} \quad \text{and} \quad R(x)=\frac{16x^2}{(x+1)^4}.
\end{align*}
Then, by \eqref{zcos} and Lemma \ref{1z1}, we deduce that
\begin{align*}
&\sum_{n=1}^{k}\frac{1}{\cos^4(\pi n/k)}=\sum_{n=1}^{k}R(z_n)=
16\sum_{n=1}^{k}\left\{\frac{1}{(z_n+1)^2}-\frac{2}{(z_n+1)^3}+\frac{1}{(z_n+1)^4}\right\}\\
&=16\left\{- \frac{(k-1)^2}{4}+\frac{(k-1)(3k-1)}{4}+\frac{(k-1)(k^3+k^2-21k+3)}{48}+\frac{1}{16}\right\}\\
&=\frac{k^2(k^2+2)}{3}.
\end{align*}
\end{proof}

\begin{corollary}
If $k$ is an odd positive integer, then
\begin{align*}
\sum_{n=1}^k\frac{\sin^2(\pi n/k)}{\cos^4(2\pi n/k)}=
\begin{cases}
\dfrac{k^4}{6}-\dfrac{k^3}{4}+\dfrac{k^2}{3}-\dfrac{k}{4},   & \text{if} \quad k\equiv 1 \pmod{4}, \\[.15in]
\dfrac{k^4}{6}+\dfrac{k^3}{4}+\dfrac{k^2}{3}+\dfrac{k}{4},  & \text{if} \quad k\equiv 3 \pmod{4}.
\end{cases}.
\end{align*}
\end{corollary}

\begin{proof}
From the identity
\begin{align*}
\sum_{n=1}^k\frac{\sin^2(\pi n/k)}{\cos^4(2\pi n/k)}=\sum_{n=1}^k\frac{1-\cos^2(\pi n/k)}{\cos^4(2\pi n/k)}
\end{align*}
and Theorems \ref{cc} and \ref{1cos4}, the corollary immediately follows.
\end{proof}

\begin{theorem}\label{ss}
If $k$ is an odd positive integer, then
\begin{align}\label{triangular}
\sum_{n=1}^{k-1}\frac{\sin^2(\pi n/k)}{\sin^4(2\pi n/k)}=\frac{1}{48}(k^4+6k^2-7).
\end{align}
\end{theorem}

\begin{proof}
We set
\begin{align*}
S=\sum_{n=1}^{k-1}\frac{\sin^2(\pi n/k)}{\sin^4(2\pi n/k)}, \quad z_n=e^{2\pi i n/k},  \quad \text{and} \quad R(x)=-\frac{4x^3}{(x-1)^2(x+1)^4}.
\end{align*}
Observe that
\begin{align*}
\frac{\sin^2(\pi n/k)}{\sin^4(2\pi n/k)}=\frac{1}{16\sin^2(\pi n/k)\cos^4(\pi n/k)},
\end{align*}
and
\begin{align}\label{sq}
\cos^2(\pi n/k)=\frac{(z_n+1)^2}{4z_n}, \quad \sin^2(\pi n/k)=1-\cos^2(\pi n/k)=-\frac{(z_n-1)^2}{4z_n}.
\end{align}
This implies that
\begin{align*}
\frac{\sin^2(\pi n/k)}{\sin^4(2\pi n/k)}=-\frac{4z_n^3}{(z_n-1)^2(z_n+1)^4}=R(z_n).
\end{align*}

Noting that
\begin{align*}
R(x)&=\frac{1}{4(x+1)}+\frac{3}{4(x+1)^2}-\frac{2}{(x+1)^3}+\frac{1}{(x+1)^4}-\frac{1}{4(x-1)}-\frac{1}{4(x-1)^2},
\end{align*}
we obtain
\begin{align}\label{S4}
S&=\sum_{n=1}^{k-1}R(z_n)
=\sum_{n=1}^{k-1}\left\{\frac{1}{4(z_n+1)}+\frac{3}{4(z_n+1)^2}-\frac{2}{(z_n+1)^3}+\frac{1}{(z_n+1)^4} \right. \notag\\
&\qquad \qquad \qquad \qquad \qquad \qquad \left. -\frac{1}{4(z_n-1)}-\frac{1}{4(z_n-1)^2}\right\}.
\end{align}

Employing Lemmas \ref{1z1} and \ref{1z-1} in \eqref{S4}, we arrive at
\begin{align*}
S&=\frac{k-1}{4}-\frac{3(k-1)^2}{16}+\frac{(k-1)(3k-1)}{4}+\frac{(k-1)(k^3+k^2-21k+3)}{48}\\
&\quad +\frac{(k-1)(k-5)}{48}=\frac{k^4}{48}+\frac{k^2}{8}-\frac{7}{48}.
\end{align*}
This completes the proof.
\end{proof}

 \begin{corollary}
Let $k=7$ in \eqref{triangular}.  Then, by an elementary calculation,
\begin{equation}\label{17a}
\sum_{1\leq n\leq 3}\df{\sin^2(\pi n/7)}{\sin^4(2\pi n/7)}=28.
\end{equation}
\end{corollary}

The identity \eqref{17a} was first established by Berndt and Liang-Cheng Zhang \cite[p.~233]{berndtzhang} ``by a direct laborious computation.''

Replace $k$ by $2k+1$ in \eqref{triangular}.  Thus,
\begin{align}\label{t}
\sum_{n=1}^k\df{\sin^2(\pi n/(2k+1))}{\sin^4(2\pi n/(2k+1))}=&\df{(2k+1)^4+6(2k+1)^2-7}{96}\notag\\
=&\df{16k^4+32k^3+48k^2+32k}{96}\notag\\
=&\df{k(k^3+2k^2+3k+2)}{6}\notag\\
=&\df{k(k+1)(k^2+k+2)}{6}.
\end{align}
Recall that the $k$th triangular number $T(k)$ is defined by
\begin{equation*}
T(k):=\df{k(k+1)}{2}.
\end{equation*}
 Now,
 \begin{equation}\label{tt}
 T(T(k))=\df{k(k+1)}{4}\left(\df{k(k+1)}{2}+1\right)=\df{k(k+1)(k^2+k+2)}{8}.
 \end{equation}
 Comparing \eqref{t} with \eqref{tt}, we conclude that
 \begin{equation}\label{ttt}
 \sum_{n=1}^k\df{\sin^2(\pi n/(2k+1))}{\sin^4(2\pi n/(2k+1))}=\df43 T(T(k)).
 \end{equation}

 The beautiful identity \eqref{ttt} was established by \"{O}rs Reb\'{a}k \cite{rebak}.
   Is \eqref{ttt} just an "accident"?  Or is it an example of a deeper phenomenon?

The more general, related sum
\begin{equation*}
\sum_{j=1}^{k-1}\df{\sin^2(\pi aj/k)}{\sin^{2n}(\pi j/k)},
\end{equation*}
where $n, k$, and $a$ are positive integers with $a<k$, has been studied and evaluated by other authors.  See \cite[p.~379]{yeap} for one of these proofs and for further references.

\begin{theorem}\label{1sin4}
Let $k$ be a positive integer. Then
\begin{align*}
\sum_{n=1}^{k-1}\frac{1}{\sin^4(\pi n/k)}=\frac{1}{45}(k^4+10k^2-11).
\end{align*}
\end{theorem}

\begin{proof}
We define
\begin{align*}
z_n=e^{2\pi i n/k} \quad \text{and} \quad R(x)=\frac{16x^2}{(x-1)^4}.
\end{align*}
Note that
\begin{align*}
R(x)=16\left\{\frac{1}{(x-1)^2}+\frac{2}{(x-1)^3}+\frac{1}{(x-1)^4} \right\},
\end{align*}
so that by Lemma \ref{1z-1} and \eqref{sq},
\begin{align*}
\sum_{n=1}^{k-1}\frac{1}{\sin^4(\pi n/k)}&=\sum_{n=1}^{k-1}R(z_n)\\
&=16\sum_{n=1}^{k-1}\left\{\frac{1}{(z_n-1)^2}+\frac{2}{(z_n-1)^3}+\frac{1}{(z_n-1)^4} \right\}\\
&=-\frac{4(k-1)(k-5)}{3}+4(k-1)(k-3)+\frac{1}{45}(k-1)(k^3+k^2-109k+251)\\
&=\frac{1}{45}(k^4+10k^2-11).
\end{align*}
\end{proof}

\begin{theorem} \label{1sin42}
If $k$ is an odd positive integer, then
\begin{align*}
\sum_{n=1}^{k-1}\frac{1}{\sin^4(2\pi n/k)}=\frac{1}{45}(k^4+10k^2-11).
\end{align*}
\end{theorem}

\begin{proof}
Since $k$ is odd, we can easily check that
\begin{align*}
\sum_{n=1}^{k-1}\frac{1}{\sin^4(2\pi n/k)}=\sum_{n=1}^{k-1}\frac{1}{\sin^4(\pi n/k)}.
\end{align*}
Hence, by Theorem \ref{1sin4}, we complete the proof.
\end{proof}

From Theorems \ref{ss} and \ref{1sin42}, we can easily derive the following corollary.
\begin{corollary}
If $k$ is an odd positive integer, then
\begin{align*}
\sum_{n=1}^{k-1}\frac{\cos^2(\pi n/k)}{\sin^4(2\pi n/k)}=\frac{1}{720}(k^4+70k^2-71).
\end{align*}
\end{corollary}

\begin{theorem}\label{1sin2}
If $k$ is a positive integer, then
\begin{align*}
\sum_{n=1}^{k-1}\frac{1}{\sin^2(\pi n/k)}=\frac{k^2-1}{3}.
\end{align*}
\end{theorem}

\begin{proof}
If we define
\begin{align*}
z_n=e^{2\pi i n/k} \quad \text{and} \quad R(x)=-\frac{4x}{(x-1)^2},
\end{align*}
then  by \eqref{sq},
\begin{align*}
\sum_{n=1}^{k-1}\frac{1}{\sin^2(\pi n/k)}=\sum_{n=1}^{k-1}R(z_n).
\end{align*}
From the following identity and Lemma \ref{1z-1},
\begin{align*}
R(x)=-4\left\{\frac{1}{x-1}+\frac{1}{(x-1)^2} \right\},
\end{align*}
it follows that
\begin{align*}
\sum_{n=1}^{k-1}\frac{1}{\sin^2(\pi n/k)}&=-4\sum_{n=1}^{k-1}\left\{\frac{1}{z_n-1}+\frac{1}{(z_n-1)^2} \right\}\\
&=2(k-1)+\frac{(k-1)(k-5)}{3}=\frac{k^2-1}{3}.
\end{align*}
\end{proof}

See \cite[p.~365]{yeap} for another proof of Theorem \ref{1sin2}.

\begin{corollary}\label{mc1} For each positive integer  $k,$
\begin{align*}
S_3(1, k):=-\frac{3}{8k^3}\sum_{n=1}^{k-1}\cot^2(\pi n/k)\csc^2(\pi n/k)=-\frac{(k^2-1)(k^2-4)}{120k^3}.
\end{align*}
\end{corollary}

\begin{proof}
From Theorems \ref{1sin2} and \ref{1sin4}, it follows that
\begin{align*}
S_3(1, k)&=-\frac{3}{8k^3}\sum_{n=1}^{k-1}\frac{\cos^2(\pi n/k)}{\sin^4(\pi n/k)}
=-\frac{3}{8k^3}\sum_{n=1}^{k-1}\left\{\frac{1}{\sin^4(\pi n/k)}-\frac{1}{\sin^2(\pi n/k)}\right\}\\
&=-\frac{1}{8k^3}\left\{\frac{1}{15}(k^4+10k^2-11)-(k^2-1)\right\}=-\frac{(k^2-1)(k^2-4)}{120k^3}.
\end{align*}
\end{proof}

An evaluation related to that in Corollary \ref{mc1} can be found in \cite[p.~367, Corollary 2.7]{yeap}.

\begin{corollary}\label{mc7} For each odd positive integer  $k,$
\begin{align*}
S_3(2, k):=-\frac{3}{8k^3}\sum_{n=1}^{k-1}\cot(2\pi n/k)\cot(\pi n/k)\csc^2(\pi n/k)=-\frac{(k^2-1)(k^2-19)}{240k^3}.
\end{align*}
\end{corollary}

\begin{proof}
Analogously, by Theorems \ref{1sin2} and \ref{1sin4}, we have
\begin{align*}
S_3(2, k)&=-\frac{3}{8k^3}\sum_{n=1}^{k-1}\frac{\cos(2\pi n/k)\cos(\pi n/k)}{\sin(2\pi n/k)\sin^3(\pi n/k)}
=-\frac{3}{8k^3}\sum_{n=1}^{k-1}\frac{\cos(2\pi n/k)}{2\sin^4(\pi n/k)}\\
&=-\frac{3}{8k^3}\sum_{n=1}^{k-1}\left\{\frac{1}{2\sin^4(\pi n/k)}-\frac{1}{\sin^2(\pi n/k)}\right\}\\
&=-\frac{1}{8k^3}\left\{\frac{1}{30}(k^4+10k^2-11)-(k^2-1)\right\}=-\frac{(k^2-1)(k^2-19)}{240k^3}.
\end{align*}
\end{proof}

In fact, $S_3(1,k)$ and $S_3(2,k)$ are generalized Dedekind sums, which we now define.
Let $B_3(x)$ denote the third Bernoulli polynomial.  Then the generalized Dedekind sum $S_3(h,k)$ is defined by \cite[p.~198]{mcintosh}
\begin{align}\label{mc2}
 S_3(h,k)&=\sum_{n=1}^{k-1}\df{n}{k}B_3\left(\df{hn}{k}-\left[\df{hn}{k}\right]\right)\notag\\
 &=-\frac{3}{8k^3}\sum_{n=1}^{k-1}\cot(h\pi n/k)\cot(\pi n/k)\csc^2(\pi n/k),
 \end{align}
where $[x]$ is the greatest integer $\leq x$. We have therefore used McIntosh's formulation involving generalized Dedekind sums \cite[p.~199, Equation (6)]{mcintosh} in our expressions of  Corollaries \ref{mc1} and \ref{mc7}.

\begin{theorem}\label{1cos2}
If $k$ is an odd positive integer, then
\begin{align*}
\sum_{n=1}^{k}\frac{1}{\cos^2(\pi n/k)}=k^2.
\end{align*}
\end{theorem}

\begin{proof}
Defining
\begin{align*}
z_n=e^{2\pi i n/k} \quad \text{and} \quad R(x)=\frac{4x}{(x+1)^2}
\end{align*}
and using  Lemma \ref{1z1}, we have
\begin{align*}
\sum_{n=1}^{k}\frac{1}{\cos^2(\pi n/k)}&=\sum_{n=1}^{k}R(z_n)=4\sum_{n=1}^{k}\left\{\frac{1}{z_n+1}-\frac{1}{(z_n+1)^2} \right\}\\
&=4\left\{ \frac{k}{2}+\frac{k(k-2)}{4}\right\}=k^2.
\end{align*}
\end{proof}

The problem of finding a general method for evaluating
\begin{equation}\label{1cos2general}
\sum_{n=1}^{k-1}\frac{1}{\cos^{2m}(\pi n/k)}, \quad m\geq1,
\end{equation}
was apparently first raised by L.~A.~Gardner \cite{gardner} and M.~E.~Fisher \cite{fisher}.
For each integer $m\geq1$, Y.~He \cite{he} established a generating function for \eqref{1cos2general}, providing their values in terms of polynomials in $k$.  Moreover, Y.~He also showed that
\begin{equation}\label{1sin2general}
\sum_{n=1}^{k-1}\frac{1}{\sin^{2m}(\pi n/k)}
\end{equation}
can be evaluated in terms of a certain sequence of polynomials in the variable $k$ .
Earlier,  W.~Chu and A.~Marini \cite{chumarini} established a generating function for the general sum \eqref{1cos2general}.
See also a further paper by Chu \cite{chu}.

C.~M.~ da Fonseca,  M.~L.~Glasser, and V.~ Kowalenko \cite{fgk1}, \cite{fgk} developed methods for determining values of  \eqref{1cos2general} and \eqref{1sin2general}.  These authors study further trigonometric sums, and their paper also provides an excellent survey of a broad expanse of related literature.

 Via generating functions, D.~ Cvijovi\'{c} \cite{dc} also showed how to evaluate
$$ \sum_{j=1}^k\csc^{2n}\left(\df{(2j-1)\pi}{4k}\right)=\sum_{j=1}^k\sec^{2n}\left(\df{(2j-1)\pi}{4k}\right), \quad n\geq 1.$$

  See the paper by X.~Wang and D.~Y.~Zheng \cite{wangzheng} for the evaluations of several further classes of trigonometric sums.

\begin{corollary} For each odd positive integer  $k,$
\begin{align*}
\sum_{n=1}^{k}\tan^2(\pi n/k)\sec^2(\pi n/k)=\frac{k^2(k^2-1)}{3}.
\end{align*}
\end{corollary}

\begin{proof}
Using Theorems \ref{1cos2} and \ref{1cos4}, we derive
\begin{align*}
\sum_{n=1}^{k}\tan^2(\pi n/k)\sec^2(\pi n/k)&=\sum_{n=1}^{k}\frac{\sin^2(\pi n/k)}{\cos^4(\pi n/k)}=
\sum_{n=1}^{k}\left\{\frac{1}{\cos^4(\pi n/k)}-\frac{1}{\cos^2(\pi n/k)}\right\}\\
&=\frac{k^2(k^2+2)}{3}-k^2=\frac{k^2(k^2-1)}{3}.
\end{align*}
\end{proof}

\begin{corollary} For each odd positive integer  $k,$
\begin{align*}
\sum_{n=1}^{k}\cos(2\pi n/k)\sec^4(\pi n/k)=-\frac{k^2(k^2-4)}{3}.
\end{align*}
\end{corollary}

\begin{proof} From Theorems \ref{1cos2} and \ref{1cos4}, it follows that
\begin{align*}
\sum_{n=1}^{k}\cos(2\pi n/k)\sec^4(\pi n/k)&=\sum_{n=1}^{k}\frac{2\cos^2(\pi n/k)-1}{\cos^4(\pi n/k)}\\
&=\sum_{n=1}^{k}\left\{\frac{2}{\cos^2(\pi n/k)} -\frac{1}{\cos^4(\pi n/k)}\right\}\\
&=2k^2-\frac{k^2(k^2+2)}{3}
=-\frac{k^2(k^2-4)}{3}.
\end{align*}
\end{proof}


\section{Further Lemmas on Sums of Rational Functions of Roots of Unity}

Applying Lemma \ref{1z-1} and partial fraction decompositions, we can readily derive the following lemma.
\begin{lemma}\label{zz1}
If $k$ is a positive integer and  $z_n=e^{2\pi in/k},$ then
\begin{align*}
&\sum_{n=1}^{k-1}\frac{z_n}{z_n-1}=\frac{k-1}{2}, \qquad \qquad \qquad \qquad \qquad
\sum_{n=1}^{k-1}\frac{z_n}{(z_n-1)^2}=-\frac{(k-1)(k+1)}{12}, \\
&\sum_{n=1}^{k-1}\frac{z_n^2}{(z_n-1)^2}=-\frac{(k-1)(k-5)}{12}, \qquad \qquad
\sum_{n=1}^{k-1}\frac{z_n}{(z_n-1)^3}=\frac{(k-1)(k+1)}{24}, \\
&\sum_{n=1}^{k-1}\frac{z_n^2}{(z_n-1)^3}=-\frac{(k-1)(k+1)}{24}, \qquad \qquad
\sum_{n=1}^{k-1}\frac{z_n^3}{(z_n-1)^3}=-\frac{(k-1)(k-3)}{8},\\
&\sum_{n=1}^{k-1}\frac{z_n}{(z_n-1)^4}=\frac{(k-1)(k+1)(k^2-19)}{720}, \quad
\sum_{n=1}^{k-1}\frac{z_n^2}{(z_n-1)^4}=\frac{(k-1)(k+1)(k^2+11)}{720}, \\
&\sum_{n=1}^{k-1}\frac{z_n^3}{(z_n-1)^4}=\frac{(k-1)(k+1)(k^2-19)}{720}, \\
&\sum_{n=1}^{k-1}\frac{z_n^4}{(z_n-1)^4}=\frac{(k-1)(k^3+k^2-109k+251)}{720}.
\end{align*}
\end{lemma}

\begin{lemma}\label{R} Let $p$ and $q$ be positive integers  with $p, q \geq 2$ and $(p, q)=1.$
If $\omega_j=e^{2\pi ij/p}$ and $\xi_j=e^{2\pi ij/q},$ then
\begin{align*}
&\frac{x(x^p+1)(x^q+1)}{(x-1)^2(x^p-1)(x^q-1)}=\frac{p^2+q^2+1}{3pq(x-1)}+\frac{p^2+q^2+13}{3pq(x-1)^2}
+\frac{8}{pq(x-1)^3}+\frac{4}{pq(x-1)^4}\\
&\qquad \qquad +\frac{2}{p}\sum_{j=1}^{p-1}\frac{\omega_j^2(\omega_j^q+1)}{(\omega_j-1)^2(\omega_j^q-1)(x-\o_j)}
+\frac{2}{q}\sum_{j=1}^{q-1}\frac{\xi_j^2(\xi_j^p+1)}{(\xi_j-1)^2(\xi_j^p-1)(x-\xi_j)}.
\end{align*}
\end{lemma}

\begin{proof}
We first note that by taking logarithmic derivatives, we can easily derive
\begin{align}\label{xp1}
\frac{1}{x^p-1}=\frac{1}{p}\sum_{j=1}^p\frac{\o_j}{x-\omega_j} \quad  \text{and} \quad  \frac{1}{x^q-1}=\frac{1}{q}\sum_{j=1}^q\frac{\xi_j}{x-\xi_j}.
\end{align}
So, we have
\begin{align}
&\frac{x(x^p+1)(x^q+1)}{(x-1)^2(x^p-1)(x^q-1)}=\frac{x}{(x-1)^2}\Big(1+\frac{2}{x^p-1}\Big)\Big(1+\frac{2}{x^q-1}\Big)\notag\\
=&\frac{x}{(x-1)^2}\Big(1+\frac{2}{p}\sum_{j=1}^p\frac{\o_j}{x-\omega_j}\Big)
\Big(1+\frac{2}{q}\sum_{j=1}^q\frac{\xi_j}{x-\xi_j}\Big)\notag\\
=&\frac{x}{(x-1)^2}\Big(1+\frac{2}{p(x-1)}+\frac{2}{p}\sum_{j=1}^{p-1}\frac{\o_j}{x-\omega_j}\Big)
\Big(1+\frac{2}{q(x-1)}+\frac{2}{q}\sum_{j=1}^{q-1}\frac{\xi_j}{x-\xi_j}\Big)\notag\\
=&\frac{x}{(x-1)^2}+\frac{2x}{p(x-1)^3}+\frac{2x}{q(x-1)^3}+\frac{4x}{pq(x-1)^4}
+\frac{2}{p}\sum_{j=1}^{p-1}\frac{\o_j x}{(x-1)^2(x-\omega_j)}\notag\\
&+\frac{2}{q}\sum_{j=1}^{q-1}\frac{\xi_j x}{(x-1)^2(x-\xi_j)}+\frac{4}{pq}\sum_{j=1}^{p-1}\frac{\o_j x}{(x-1)^3(x-\omega_j)}
+\frac{4}{pq}\sum_{j=1}^{p-1}\frac{\xi_j x}{(x-1)^3(x-\xi_j)}\notag\\
&+\frac{4}{pq}\sum_{i=1}^{p-1}\sum_{j=1}^{q-1}\frac{\o_i\xi_j x}{(x-1)^2(x-\o_i)(x-\xi_j)}.\label{A}
\end{align}
We divide the far right side of \eqref{A} into six parts, $A_1, A_2, A_3, A_4, A_5, A_6$, and examine each separately.

First, we consider the following partial fraction decompositions:
\begin{align}
&\frac{x}{(x-1)^2}=\frac{1}{x-1}+\frac{1}{(x-1)^2}, \label{p1}\\
&\frac{x}{(x-1)^2(x-a)}=-\frac{a}{(a-1)^2(x-1)}-\frac{1}{(a-1)(x-1)^2}+\frac{a}{(a-1)^2(x-a)}, \label{p2}\\
&\frac{x}{(x-1)^3(x-a)}=-\frac{a}{(a-1)^3(x-1)}-\frac{a}{(a-1)^2(x-1)^2} \notag\\
&\qquad \qquad \qquad \qquad \qquad \,\,-\frac{1}{(a-1)(x-1)^3}+\frac{a}{(a-1)^3(x-a)}, \label{p3}\\
&\frac{x}{(x-1)^2(x-a)(x-b)}=\frac{ab-1}{(a-1)^2(b-1)^2(x-1)}+\frac{1}{(a-1)(b-1)(x-1)^2} \notag \\
&\qquad \qquad \qquad \qquad \qquad\,\,\, +\frac{a}{(a-1)^2(a-b)(x-a)}+\frac{b}{(b-1)^2(b-a)(x-b)}. \label{p4}
\end{align}
By \eqref{p1},
\begin{align}
A_1:=&\frac{x}{(x-1)^2}+\frac{2x}{p(x-1)^3}+\frac{2x}{q(x-1)^3}+\frac{4x}{pq(x-1)^4}\notag\\
 =&\frac{1}{x-1}+\frac{1}{(x-1)^2}+\frac{2}{p}\Big( \frac{1}{(x-1)^2}+\frac{1}{(x-1)^3}\Big)\notag\\
&  +\frac{2}{q}\Big(\frac{1}{(x-1)^2}+\frac{1}{(x-1)^3}\Big)+\frac{4}{pq}\Big( \frac{1}{(x-1)^3}+\frac{1}{(x-1)^4}\Big)\notag\\
 =&\frac{1}{x-1}+\Big(1+\frac{2}{p}+\frac{2}{q}\Big)\frac{1}{(x-1)^2}
+\Big(\frac{2}{p}+\frac{2}{q}+\frac{4}{pq}\Big)\frac{1}{(x-1)^3}+\frac{4}{pq}\frac{1}{(x-1)^4}.\label{A1}
\end{align}

Next, using \eqref{p2}, we have
\begin{align}\label{A2}
A_2:=\frac{2}{p}\sum_{j=1}^{p-1}\frac{\o_jx}{(x-1)^2(x-\o_j)}=\frac{2}{p}\sum_{j=1}^{p-1}
\left\{-\frac{\o_j^2}{(\o_j-1)^2(x-1)}-\frac{\o_j}{(\o_j-1)(x-1)^2}+\frac{\o_j^2}{(\o_j-1)^2(x-\o_j)}\right\},
\end{align}
and
\begin{align}\label{A3}
A_3:=\frac{2}{q}\sum_{j=1}^{q-1}\frac{\x_jx}{(x-1)^2(x-\x_j)}=\frac{2}{q}\sum_{j=1}^{q-1}
\left\{-\frac{\x_j^2}{(\x_j-1)^2(x-1)}-\frac{\x_j}{(\x_j-1)(x-1)^2}+\frac{\x_j^2}{(\x_j-1)^2(x-\x_j)}\right\}.
\end{align}

By \eqref{p3} and \eqref{p4}, we successively obtain
\begin{align}\label{A4}
A_4=:&\frac{4}{pq}\sum_{j=1}^{p-1}\frac{\o_jx}{(x-1)^3(x-\o_j)}\\
=&\frac{4}{pq}\sum_{j=1}^{p-1}
\left\{-\frac{\o_j^2}{(\o_j-1)^3(x-1)}-\frac{\o_j^2}{(\o_j-1)^2(x-1)^2}-\frac{\o_j}{(\o_j-1)(x-1)^3}
+\frac{\o_j^2}{(\o_j-1)^3(x-\o_j)}\right\},\notag
\end{align}

\begin{align}\label{A5}
A_5:=&\frac{4}{pq}\sum_{j=1}^{q-1}\frac{\x_jx}{(x-1)^3(x-\x_j)}\\
=&\frac{4}{pq}\sum_{j=1}^{q-1}
\left\{-\frac{\x_j^2}{(\x_j-1)^3(x-1)}-\frac{\x_j^2}{(\x_j-1)^2(x-1)^2}-\frac{\x_j}{(\x_j-1)(x-1)^3}
+\frac{\x_j^2}{(\x_j-1)^3(x-\x_j)}\right\},\notag
\end{align}
and
\begin{align}\label{A6}
A_6:=&\frac{4}{pq}\sum_{i=1}^{p-1}\sum_{j=1}^{q-1}\frac{\o_i\x_jx}{(x-1)^2(x-\o_i)(x-\x_j)}
=\frac{4}{pq}\sum_{i=1}^{p-1}\sum_{j=1}^{q-1}
\left\{\frac{\o_i\x_j(\o_i\x_j-1)}{(\o_i-1)^2(\x_j-1)^2(x-1)} \right.\notag \\
&   \left. +\frac{\o_i\x_j}{(\o_i-1)(\x_j-1)(x-1)^2} +\frac{\o_i^2\x_j}{(\o_i-1)^2(\o_i-\x_j)(x-\o_i)}
+\frac{\o_i\x_j^2}{(\x_i-1)^2(\x_j-\o_i)(x-\x_j)}\right\}.
\end{align}
We now substitute \eqref{A1}--\eqref{A6} into \eqref{A}.  Then, we calculate the partial fraction decompositions of $\dfrac{1}{(x-1)^n}, n=1,2,3,4$, and $\dfrac{1}{x-\omega_j}$, $\dfrac{1}{x-\xi_j}$, all of which now appear in the new representation of \eqref{A}.

  First, by Lemma \ref{zz1}, the coefficient of $\dfrac{1}{x-1}$ is
\begin{align}
&1-\frac{2}{p}\sum_{j=1}^{p-1}\frac{\o_j^2}{(\o_j-1)^2}-\frac{2}{q}\sum_{j=1}^{q-1}\frac{\x_j^2}{(\x_j-1)^2}
-\frac{4}{pq}\sum_{j=1}^{p-1}\frac{\o_j^2}{(\o_j-1)^3}-\frac{4}{pq}\sum_{j=1}^{q-1}\frac{\x_j^2}{(\x_j-1)^3}\notag\\
&\qquad +\frac{4}{pq}\left\{\sum_{i=1}^{p-1}\frac{\o_i^2}{(\o_i-1)^2}\sum_{j=1}^{q-1}\frac{\x_j^2}{(\x_j-1)^2}
-\sum_{i=1}^{p-1}\frac{\o_i}{(\o_i-1)^2}\sum_{j=1}^{q-1}\frac{\x_j}{(\x_j-1)^2}\right\}\notag\\
&=1+\frac{(p-1)(p-5)}{6p}+\frac{(q-1)(q-5)}{6q}+\frac{(p-1)(p+1)}{6pq}+\frac{(q-1)(q+1)}{6pq}\notag\\
&\qquad +\frac{(p-1)(p-5)(q-1)(q-5)}{36pq}-\frac{(p-1)(p+1)(q-1)(q+1)}{36pq}\notag\\
&=\frac{p^2+q^2+1}{3pq}.\label{B1}
\end{align}

Similarly, we see that the coefficient of $\dfrac{1}{(x-1)^2}$ is
 \begin{align}
&1+\frac{2}{p}+\frac{2}{q}-\frac{2}{p}\sum_{j=1}^{p-1}\frac{\o_j}{\o_j-1}-\frac{2}{q}\sum_{j=1}^{q-1}\frac{\x_j}{\x_j-1}
 -\frac{4}{pq}\sum_{j=1}^{p-1}\frac{\o_j^2}{(\o_j-1)^2} -\frac{4}{pq}\sum_{j=1}^{q-1}\frac{\x_j^2}{(\x_j-1)^2}\notag\\
& +\frac{4}{pq}\sum_{i=1}^{p-1}\frac{\o_i}{\o_i-1}\sum_{j=1}^{q-1}\frac{\x_j}{\x_j-1}\notag\\
&=1+\frac{2}{p}+\frac{2}{q}-\frac{p-1}{p}-\frac{q-1}{q}+\frac{(p-1)(p-5)}{3pq}+\frac{(q-1)(q-5)}{3pq}
+\frac{(p-1)(q-1)}{pq}\notag\\
&=\frac{p^2+q^2+13}{3pq},\label{B2}
  \end{align}
and the coefficient of $\dfrac{1}{(x-1)^3}$ is
 \begin{align}
&\frac{2}{p}+\frac{2}{q}+\frac{4}{pq}-\frac{4}{pq}\sum_{j=1}^{p-1}\frac{\o_j}{\o_j-1}-\frac{4}{pq}\sum_{j=1}^{q-1}\frac{\x_j}{\x_j-1}\notag\\
&\qquad =\frac{2}{p}+\frac{2}{q}+\frac{4}{pq}-\frac{2(p-1)}{pq}-\frac{2(q-1)}{pq}=\frac{8}{pq}.\label{B3}
\end{align}
It is easy to see that the coefficient of $\dfrac{1}{(x-1)^4}$ is
\begin{equation}\dfrac{4}{pq}.\label{B4}\end{equation}

Next, for each $1\leq j\leq p-1,$ the coefficient of $\dfrac{1}{x-\o_j}$ is
\begin{align*}
&\frac{2\o_j^2}{p(\o_j-1)^2}+\frac{4\o_j^2}{pq(\o_j-1)^3}+\frac{4\o_j^2}{pq(\o_j-1)^2}\sum_{i=1}^{q-1}\frac{\x_i}{\o_j-\x_i} \notag\\
&=\frac{2\o_j^2}{p(\o_j-1)^2}\left\{1+\frac{2}{q(\o_j-1)}+\frac{2}{q}\sum_{i=1}^{q-1}
\Big(\frac{\o_j}{\o_j-\x_i}-1\Big)\right\} \notag\\
&=\frac{2\o_j^2}{p(\o_j-1)^2}\left\{1+\frac{2}{q(\o_j-1)}+\frac{2\o_j}{q}\sum_{i=1}^{q-1}\frac{1}{\o_j-\x_i}
-\frac{2(q-1)}{q}\right\}.
\end{align*}

From the identity
\begin{align}\label{1xz}
\sum_{j=1}^{k-1}\frac{1}{x-z_j}=\frac{kx^{k-1}}{x^k-1}-\frac{1}{x-1},
\end{align}
it follows that  the coefficient of $\dfrac{1}{x-\o_j}$ is equal to
\begin{align}
&\frac{2\o_j^2}{p(\o_j-1)^2}\left\{1+\frac{2}{q(\o_j-1)}+\frac{2\o_j}{q}\Big(\frac{q\o_j^{q-1}}{\o_j^q-1}-\frac{1}{\o_j-1}\Big)
-\frac{2(q-1)}{q}\right\}\notag\\
&=\frac{2\o_j^2}{p(\o_j-1)^2}\left\{-1+\frac{2}{q(\o_j-1)}+\frac{2\o_j^q}{\o_j^q-1}-\frac{2\o_j}{q(\o_j-1)}+\frac{2}{q}\right\}\notag\\
&=\frac{2\o_j^2(\o_j^q+1)}{p(\o_j-1)^2(\o_j^q-1)}.\label{B5}
\end{align}
Analogously, for each $1\leq j\leq q-1,$ the coefficient of $\dfrac{1}{x-\x_j}$ is
\begin{equation}\label{B6}
\frac{2\xi_j^2(\xi_j^p+1)}{q(\xi_j-1)^2(\xi_j^p-1)}.
\end{equation}

Putting \eqref{B1}--\eqref{B6} into \eqref{A},  we complete the proof of Lemma \ref{R}.
\end{proof}

The following two beautiful  theorems for sums of roots of unity are apparently new.

\begin{lemma}\label{lemr1}
Let $p$ and $q$ be positive integers  with $q \geq 2$ and $(p, q)=1.$
If $\xi_j=e^{2\pi ij/q}$, then
\begin{align*}
\sum_{j=1}^{q-1}\frac{\xi_j}{(\xi_j-1)^2(\xi_j^p-1)}=\frac{q^2-1}{24}.
\end{align*}
\end{lemma}

\begin{proof}
We first note that by symmetry
\begin{align*}
\sum_{n=1}^{q-1}\cot\Big(\frac{\pi np}{q}\Big)\csc^2\Big(\frac{\pi n}{q}\Big)=0.
\end{align*}
On the other hand,
\begin{align*}
\sum_{n=1}^{q-1}\cot\Big(\frac{\pi np}{q}\Big)\csc^2\Big(\frac{\pi n}{q}\Big)
&=-4i\sum_{n=1}^{q-1}\frac{\xi_n^p+1}{\xi_n^p-1}\frac{\xi_n}{(\xi_n-1)^2} \notag\\
&=-4i\sum_{n=1}^{q-1}\left(1+\frac{2}{\xi_n^p-1}\right)\frac{\xi_n}{(\xi_n-1)^2}.
\end{align*}
Therefore, by Lemma \ref{zz1}, we arrive at
\begin{align*}
\sum_{n=1}^{q-1}\frac{\xi_n}{(\xi_n-1)^2(\xi_n^p-1)}=-\frac12\sum_{n=1}^{q-1}\frac{\xi_n}{(\xi_n-1)^2}
=\frac{q^2-1}{24}.
\end{align*}
\end{proof}

\begin{lemma}\label{lemr2}
Let $p$ and $q$ be  positive integers  with $p, q \geq 2$ and $(p, q)=1.$
If $\omega_j=e^{2\pi ij/p}$ and $\xi_j=e^{2\pi ij/q},$ then
\begin{align*}
&p\sum_{j=1}^{q-1}\frac{\xi_j}{(\xi_j-1)^3(\xi_j^p-1)}+q\sum_{j=1}^{p-1}\frac{\o_j}{(\o_j-1)^3(\o_j^q-1)}\\
&\qquad \qquad =\frac{1}{720}(p^4+q^4-5p^2q^2-15p^2q-15pq^2+15p+15q+3).
\end{align*}
\end{lemma}

\begin{proof}
We first observe that  by \eqref{xp1}, \eqref{p3}, and Lemma \ref{zz1},
\begin{align*}
&\frac{qx}{(x-1)^3(x^q-1)}=\frac{x}{(x-1)^4}+\frac{x}{(x-1)^3}\sum_{i=1}^{q-1}\frac{\xi_i}{x-\xi_i} \\
&\qquad=\frac{1}{(x-1)^3}+\frac{1}{(x-1)^4}-\sum_{i=1}^{q-1}\frac{\xi_i^2}{(\xi_i-1)^3(x-1)}
-\sum_{i=1}^{q-1}\frac{\xi_i^2}{(\xi_i-1)^2(x-1)^2}\\
&\qquad \qquad-\sum_{i=1}^{q-1}\frac{\xi_i}{(\xi_i-1)(x-1)^3}+\sum_{i=1}^{q-1}\frac{\xi_i^2}{(\xi_i-1)^3(x-\xi_i)}\\
&\qquad  =\frac{(q-1)(q+1)}{24}\frac{1}{x-1}+\frac{(q-1)(q-5)}{12}\frac{1}{(x-1)^2}-\frac{q-3}{2}\frac{1}{(x-1)^3}\\
&\qquad \qquad +\frac{1}{(x-1)^4}+\sum_{i=1}^{q-1}\frac{\xi_i^2}{(\xi_i-1)^3(x-\xi_i)}.
\end{align*}
Thus, it follows that
\begin{align}
q\sum_{j=1}^{p-1}\frac{\o_j}{(\o_j-1)^3(\o_j^q-1)}&=\frac{(q-1)(q+1)}{24}
\sum_{j=1}^{p-1}\frac{1}{\o_j-1}+\frac{(q-1)(q-5)}{12}\sum_{j=1}^{p-1}\frac{1}{(\o_j-1)^2}\notag\\
&-\frac{q-3}{2}\sum_{j=1}^{p-1}\frac{1}{(\o_j-1)^3}+\sum_{j=1}^{p-1}\frac{1}{(\o_j-1)^4}
+\sum_{j=1}^{p-1}\sum_{i=1}^{q-1}\frac{\xi_i^2}{(\xi_i-1)^3(\o_j-\xi_i)}.\label{D1}
\end{align}
Using \eqref{1xz}, we deduce that
\begin{align}\label{xi2}
&\sum_{j=1}^{p-1}\sum_{i=1}^{q-1}\frac{\xi_i^2}{(\xi_i-1)^3(\o_j-\xi_i)}
=\sum_{i=1}^{q-1}\frac{\xi_i^2}{(\xi_i-1)^3}\left\{\frac{1}{\xi_i-1}-\frac{p}{\xi_i}-\frac{p}{\xi_i(\xi_i^p-1)} \right\} \notag\\
&\qquad =\sum_{i=1}^{q-1}\frac{\xi_i^2}{(\xi_i-1)^4}-p\sum_{i=1}^{q-1}\frac{\xi_i}{(\xi_i-1)^3}
-p\sum_{i=1}^{q-1}\frac{\xi_i}{(\xi_i-1)^3(\xi_i^p-1)}.
\end{align}
Hence, by Lemmas \ref{1z-1} and \ref{zz1}, \eqref{D1}, and \eqref{xi2},
\begin{align*}
q\sum_{j=1}^{p-1}\frac{\o_j}{(\o_j-1)^3(\o_j^q-1)}=&
\frac{1}{720}(p^4+q^4-5p^2q^2-15p^2q-15pq^2+15p+15q+3)\\
&-p\sum_{i=1}^{q-1}\frac{\xi_i}{(\xi_i-1)^3(\xi_i^p-1)},
\end{align*}
which completes the proof.
\end{proof}

\begin{lemma}\label{xk2}
Let $k>1$ be an integer  and $z_j=e^{2\pi ij/k}.$ Then,
\begin{align*}
\frac{x}{(x-1)^2(x^k-1)^2}=&-\frac{(k-1)(k^2-4k+1)}{12k^2(x-1)}+\frac{(k-1)(5k-13)}{12k^2(x-1)^2}
-\frac{k-2}{k^2(x-1)^3}+\frac{1}{k^2(x-1)^4}\\
&-\frac{1}{k^2}\sum_{j=1}^{k-1}\left\{\frac{kz_j^2}{(z_j-1)^2}
+\frac{2z_j^2}{(z_j-1)^3}\right\}\frac{1}{(x-z_j)}+\frac{1}{k^2}\sum_{j=1}^{k-1}\frac{z_j^3}{(z_j-1)^2(x-z_j)^2}.
\end{align*}
\end{lemma}

\begin{proof}
We first observe that by \eqref{xp1},
\begin{align}\label{xk2-1}
&\frac{k^2x}{(x-1)^2(x^k-1)^2}=\frac{x}{(x-1)^2}\left(\frac{1}{x-1}+\sum_{j=1}^{k-1}\frac{z_j}{x-z_j}\right)^2 \notag\\
&=\frac{x}{(x-1)^4}+\frac{2x}{(x-1)^3}\sum_{j=1}^{k-1}\frac{z_j}{x-z_j}+\frac{x}{(x-1)^2}
\left\{\sum_{j=1}^{k-1}\frac{z^2_j}{(x-z_j)^2}+\sum_{\substack{i,j=1 \\ i\neq j}}^{k-1}\frac{z_iz_j}{(x-z_i)(x-z_j)}\right\}.
\end{align}
Applying \eqref{p3}, \eqref{p4}, and the following partial fraction decomposition
\begin{align*}
\frac{x}{(x-1)^2(x-a)^2}=&\frac{a+1}{(a-1)^3(x-1)}-\frac{a+1}{(a-1)^3(x-a)}\\
&+\frac{1}{(a-1)^2(x-1)^2}+\frac{a}{(a-1)^2(x-a)^2},
\end{align*}
 we deduce that
\begin{align}
&\frac{k^2x}{(x-1)^2(x^k-1)^2}=\frac{1}{(x-1)^3}+\frac{1}{(x-1)^4}\notag\\
&-2\sum_{j=1}^{k-1}\left\{\frac{z_j^2}{(z_j-1)^3(x-1)} +\frac{z_j^2}{(z_j-1)^2(x-1)^2}
+\frac{z_j}{(z_j-1)(x-1)^3}-\frac{z_j^2}{(z_j-1)^3(x-z_j)}\right\}\notag\\
&+\sum_{j=1}^{k-1}\left\{\frac{z_j^2(z_j+1)}{(z_j-1)^3(x-1)} -\frac{z_j^2(z_j+1)}{(z_j-1)^3(x-z_j)}
+\frac{z_j^2}{(z_j-1)^2(x-1)^2}+\frac{z_j^3}{(z_j-1)^2(x-z_j)^2}\right\}\notag\\
&+\sum_{\substack{i,j=1 \\ i\neq j}}^{k-1}\left\{\frac{z_iz_j(z_iz_j-1)}{(z_i-1)^2(z_j-1)^2(x-1)}
+\frac{z_iz_j}{(z_i-1)(z_j-1)(x-1)^2}+\frac{2z_iz_j^2}{(z_j-1)^2(z_j-z_i)(x-z_j)}\right\}.\label{E}
\end{align}

We now calculate the coefficients of $\dfrac{1}{(x-1)^n}$, $1\leq n \leq 4$, say, $E_1, E_2, E_3, E_4$, respectively, and also those of $\dfrac{1}{(x-z_j)^n}$, $n=1,2$, say, $E_5, E_6$, respectively, in \eqref{E}.

First, we find that the coefficient of $\dfrac{1}{x-1}$ is
\begin{align*}
&-2\sum_{j=1}^{k-1}\frac{z_j^2}{(z_j-1)^3}+\sum_{j=1}^{k-1}\frac{z_j^3+z_j^2}{(z_j-1)^3}
+\sum_{\substack{i,j=1 \\ i\neq j}}^{k-1}\frac{z^2_iz^2_j-z_iz_j}{(z_i-1)^2(z_j-1)^2} \\
&\quad =\sum_{j=1}^{k-1}\frac{z_j^3}{(z_j-1)^3}-\sum_{j=1}^{k-1}\frac{z_j^2}{(z_j-1)^3}
+\left(\sum_{j=1}^{k-1}\frac{z_j^2}{(z_j-1)^2}\right)^2-\sum_{j=1}^{k-1}\frac{z_j^4}{(z_j-1)^4}\\
&\qquad -\left(\sum_{j=1}^{k-1}\frac{z_j}{(z_j-1)^2}\right)^2+\sum_{j=1}^{k-1}\frac{z_j^2}{(z_j-1)^4},
\end{align*}
which reduces by Lemma \ref{zz1}  to
\begin{align}\label{E1}
E_1=-\dfrac{(k-1)(k^2-4k+1)}{12}.
\end{align}
Similarly, by Lemma \ref{zz1}  the coefficient of $\dfrac{1}{(x-1)^2}$ is
\begin{align}
E_2=&-2\sum_{j=1}^{k-1}\frac{z_j^2}{(z_j-1)^2}+\sum_{j=1}^{k-1}\frac{z_j^2}{(z_j-1)^2}+
\sum_{\substack{i,j=1 \\ i\neq j}}^{k-1}\frac{z_iz_j}{(z_i-1)(z_j-1)}\notag\\
 =&-2\sum_{j=1}^{k-1}\frac{z_j^2}{(z_j-1)^2}+\left(\sum_{j=1}^{k-1}\frac{z_j}{(z_j-1)}\right)^2
=\frac{(k-1)(5k-13)}{12}.\label{E2}
\end{align}
Also, the coefficient of $\dfrac{1}{(x-1)^3}$  is
\begin{align}
E_3=1-2\sum_{j=1}^{k-1}\frac{z_j}{(z_j-1)}=2-k,\label{E3}
\end{align}
and the coefficient of $\dfrac{1}{(x-1)^4}$ is
\begin{equation}\label{E4}
E_4=1.
\end{equation}

Next, for each $1\leq j \leq k-1,$ the coefficient of $\dfrac{1}{x-z_j}$ is
\begin{align}\label{E5}
E_5=&\frac{2z_j^2}{(z_j-1)^3}-\frac{z_j^3+z_j^2}{(z_j-1)^3}+\frac{2z_j^2}{(z_j-1)^2}
\sum_{\substack{i=1 \\ i \neq j}}^{k-1}\frac{z_i}{(z_j-z_i)}\notag\\
&\quad =-\frac{z_j^2}{(z_j-1)^2}+\frac{2z_j^2}{(z_j-1)^2}\sum_{\substack{i=1 \\ i \neq j}}^{k-1}\Big(\frac{z_j}{(z_j-z_i)}-1\Big)\notag\\
&\quad =-\frac{(2k-3)z_j^2}{(z_j-1)^2}+\frac{2z_j^3}{(z_j-1)^2}\sum_{\substack{i=1 \\ i \neq j}}^{k-1}\frac{1}{(z_j-z_i)}\notag\\
&\quad =-\frac{(2k-3)z_j^2}{(z_j-1)^2}+\frac{2z_j^3}{(z_j-1)^2}
\left\{\sum_{\substack{i=1 \\ i \neq j}}^{k}\frac{1}{(z_j-z_i)}-\frac{1}{z_j-1}\right\}\notag\\
&\quad =-\frac{(2k-3)z_j^2}{(z_j-1)^2}-\frac{2z_j^3}{(z_j-1)^3}+\frac{2z_j^2}{(z_j-1)^2}\sum_{s=1}^{k-1}\frac{1}{1-z_s}\notag\\
&\quad =-\frac{(2k-3)z_j^2}{(z_j-1)^2}-\frac{2z_j^3}{(z_j-1)^3}+\frac{(k-1)z_j^2}{(z_j-1)^2}\notag\\
&\quad =-\frac{kz_j^2}{(z_j-1)^2}-\frac{2z_j^2}{(z_j-1)^3},
\end{align}
by Lemma \ref{1z-1}.

Lastly, for each $1\leq j \leq k-1,$ we can easily see that the coefficient of $\dfrac{1}{(x-z_j)^2}$ is
\begin{align}\label{E6}
E_6=\frac{z_j^3}{(z_j-1)^2}.
\end{align}
Putting \eqref{E1}--\eqref{E6} in \eqref{E}, we finish the proof of Lemma \ref{xk2}.
\end{proof}

\begin{lemma}\label{lemr4}
Let $p$ and $q$ be  positive integers  with $p, q \geq 2$ and $(p, q)=1.$
If $\omega_j=e^{2\pi ij/p}$ and $\xi_j=e^{2\pi ij/q},$ then
\begin{align*}
&p^2\sum_{j=1}^{q-1}\frac{\xi_j}{(\xi_j-1)^2(\xi_j^p-1)^2}-q^2\sum_{j=1}^{p-1}\frac{\o_j}{(\o_j-1)^2(\o_j^q-1)^2}\\
&=-\frac{1}{720}(3p^4-q^4-25p^2+25q^2-5p^2q^2-30p^2q+30q+3)\\
&\qquad \qquad +2q\sum_{j=1}^{p-1}\frac{\o_j}{(\o_j-1)^3(\o_j^q-1)}.
\end{align*}
\end{lemma}

\begin{proof}
We note that applying Lemma \ref{xk2} with $x=\xi_j,$ $k=p$, and $z_j=\o_j$ and summing over $1\leq j\leq q-1$ yields
\begin{align}\label{r4e1}
&p^2\sum_{j=1}^{q-1}\frac{\xi_j}{(\xi_j-1)^2(\xi_j^p-1)^2}=
-\frac{(p-1)(p^2-4p+1)}{12}\sum_{j=1}^{q-1}\frac{1}{\xi_j-1} \notag\\
&\quad +\frac{(p-1)(5p-13)}{12}\sum_{j=1}^{q-1}\frac{1}{(\xi_j-1)^2}
-(p-2)\sum_{j=1}^{q-1}\frac{1}{(\xi_j-1)^3}+\sum_{j=1}^{q-1}\frac{1}{(\xi_j-1)^4}\notag\\
&-\sum_{i=1}^{p-1}\left\{\frac{p\o_i^2}{(\o_i-1)^2}
+\frac{2\o_i^2}{(\o_i-1)^3}\right\}\sum_{j=1}^{q-1}\frac{1}{(\xi_j-\o_i)}
+\sum_{i=1}^{p-1}\frac{\o_i^3}{(\o_i-1)^2}\sum_{j=1}^{q-1}\frac{1}{(\xi_j-\o_i)^2}.
\end{align}
By Lemma \ref{1z-1}, the sum of the first four summations in \eqref{r4e1} can be written as
\begin{align}\label{r4e2}
\frac{1}{720}(q^4+30p^3q-30p^3-25p^2q^2+25p^2+5q^2-30pq+30p-6).
\end{align}
Next, we consider that by \eqref{1xz} and Lemmas \ref{zz1} and \ref{lemr1},
\begin{align}\label{r4e3}
&\sum_{i=1}^{p-1}\left\{\frac{p\o_i^2}{(\o_i-1)^2}+\frac{2\o_i^2}{(\o_i-1)^3}\right\}\sum_{j=1}^{q-1}\frac{1}{(\xi_j-\o_i)} \notag\\
&=\sum_{i=1}^{p-1}\left\{\frac{p\o_i^2}{(\o_i-1)^2}+\frac{2\o_i^2}{(\o_i-1)^3}\right\}
\left\{\frac{1}{\o_i-1}-\frac{q}{\o_i}-\frac{q}{\o_i(\o_i^q-1)}\right\} \notag\\
&=\sum_{i=1}^{p-1}\left\{\frac{p\o_i^2}{(\o_i-1)^3}+\frac{2\o_i^2}{(\o_i-1)^4}-\frac{pq\o_i}{(\o_i-1)^2}-\frac{2q\o_i}{(\o_i-1)^3}\right.\notag\\
&\left.\qquad\qquad  -\frac{pq\o_i}{(\o_i-1)^2(\o_i^q-1)}-\frac{2q\o_i}{(\o_i-1)^3(\o_i^q-1)}\right\} \notag\\
&=\frac{1}{360}(p+1)(p-1)(p^2+15pq-15p-30q+11)\notag\\
&\qquad \qquad -2q\sum_{i=1}^{p-1}\frac{\o_i}{(\o_i-1)^3(\o_i^q-1)}.
\end{align}
Lastly, taking the derivative of both sides of the identity in \eqref{1xz}, we see that
\begin{align}\label{1xz2}
\sum_{j=1}^{k-1}\frac{1}{(x-z_j)^2}=-\frac{1}{(x-1)^2}-\frac{k(k-1)x^{k-2}}{x^k-1}+\frac{k^2x^{2k-2}}{(x^k-1)^2}.
\end{align}
With the use of \eqref{1xz2} and Lemmas \ref{zz1} and \ref{lemr1}, we have
\begin{align}\label{r4e4}
&\sum_{i=1}^{p-1}\frac{\o_i^3}{(\o_i-1)^2}\sum_{j=1}^{q-1}\frac{1}{(\xi_j-\o_i)^2}=-
\sum_{i=1}^{p-1}\frac{\o_i^3}{(\o_i-1)^2}\left\{\frac{1}{(\o_i-1)^2}+\frac{q(q-1)\o_i^{q-2}}{\o_i^q-1}-
\frac{q^2\o_i^{2q-2}}{(\o_i^q-1)^2}\right\} \notag\\
&=-\sum_{i=1}^{p-1}\frac{\o_i^3}{(\o_i-1)^4}-q(q-1)\sum_{i=1}^{p-1}\frac{\o_i^{q+1}}{(\o_i-1)^2(\o_i^q-1)}
+q^2\sum_{i=1}^{p-1}\frac{\o_i^{2q+1}}{(\o_i-1)^2(\o_i^q-1)^2} \notag\\
&=-\sum_{i=1}^{p-1}\frac{\o_i^3}{(\o_i-1)^4}-q(q-1)\sum_{i=1}^{p-1}\left(\frac{\o_i}{(\o_i-1)^2}
+\frac{\o_i}{(\o_i-1)^2(\o_i^q-1)}\right) \notag\\
&\quad +q^2\sum_{i=1}^{p-1}\left(\frac{\o_i}{(\o_i-1)^2}+\frac{2\o_i}{(\o_i-1)^2(\o_i^q-1)}+
\frac{\o_i}{(\o_i-1)^2(\o_i^q-1)^2}\right)\notag\\
&=-\sum_{i=1}^{p-1}\frac{\o_i^3}{(\o_i-1)^4}+\sum_{i=1}^{p-1}\frac{q\o_i}{(\o_i-1)^2}
+\sum_{i=1}^{p-1}\frac{q(q+1)\o_i}{(\o_i-1)^2(\o_i^q-1)}+\sum_{i=1}^{p-1}\frac{q^2\o_i}{(\o_i-1)^2(\o_i^q-1)^2} \notag\\
&=-\frac{1}{720}(p^2-1)(p^2+60q-19)+\frac{q(q+1)(p^2-1)}{24}+\sum_{i=1}^{p-1}\frac{q^2\o_i}{(\o_i-1)^2(\o_i^q-1)^2} \notag\\
&=-\frac{1}{720}(p^2-1)(p^2-30q^2+30q-19)+\sum_{i=1}^{p-1}\frac{q^2\o_i}{(\o_i-1)^2(\o_i^q-1)^2}.
\end{align}
From \eqref{r4e1},  \eqref{r4e2}, \eqref{r4e3}, and \eqref{r4e4}, the lemma follows.
\end{proof}

\section{The Determination of a Trigonometric Sum with Two Independent Periods}

\begin{theorem}\label{ccspq}
Let $p$ and $q$ be positive integers  with $p, q \geq 2$ and $(p, q)=1,$ and let $\o_j=e^{2\pi ij/p}.$ Then,
\begin{align*}
&\sum_{\substack{n=1\\ p \nmid n, \, q \nmid n}}^{pq-1}
\dfrac{\cot\Big(\dfrac{\pi n}{p}\Big)\cot\Big(\dfrac{\pi n}{q}\Big)}{\sin^2\Big(\dfrac{\pi n}{pq}\Big)}\\
&\quad =\frac{1}{45p}(p^4q^3-5p^4q+4q^3-5p^2q^3+55p^2q+30p^2-50q-30)\\
&\qquad+\frac{32}{p}\sum_{j=1}^{p-1}\frac{\o_j}{(\o_j-1)^3(\o_j^q-1)}
+\frac{32q}{p}\sum_{j=1}^{p-1}\frac{\o_j}{(\o_j-1)^2(\o_j^q-1)^2}.
\end{align*}
\end{theorem}

\begin{proof}
We define
\begin{align*}
z_n=e^{2\pi i n/(pq)}, \quad f(x)=\dfrac{(x-1)(x^{pq}-1)}{(x^p-1)(x^q-1)}
=\prod_{\substack{n=1\\ p \nmid n, \, q \nmid n}}^{pq-1}(x-z_n),
\end{align*}
and
\begin{align*}
R(x)=\dfrac{4x(x^p+1)(x^q+1)}{(x-1)^2(x^p-1)(x^q-1)}.
\end{align*}
From the identities
\begin{align}\label{s2cc}
\sin^2\Big(\dfrac{\pi n}{pq}\Big)=-\dfrac{(z_n-1)^2}{4z_n},
\quad \cot\Big(\dfrac{\pi n}{p}\Big)=i\dfrac{z_n^q+1}{z_n^q-1},
\quad \cot\Big(\dfrac{\pi n}{q}\Big)=i\dfrac{z_n^p+1}{z_n^p-1},
\end{align}
it follows that
\begin{align}\label{ccsR}
\sum_{\substack{n=1\\ p \nmid n, \, q \nmid n}}^{pq-1}
\dfrac{\cot\Big(\dfrac{\pi n}{p}\Big)\cot\Big(\dfrac{\pi n}{q}\Big)}{\sin^2\Big(\dfrac{\pi n}{pq}\Big)}=
\sum_{\substack{n=1\\ p \nmid n, \, q \nmid n}}^{pq-1}R(z_n).
\end{align}

By Lemma \ref{R},
\begin{align}\label{Rz}
\sum_{\substack{n=1\\ p \nmid n, \, q \nmid n}}^{pq-1}R(z_n)&
=4\sum_{\substack{n=1\\ p \nmid n, \, q \nmid n}}^{pq-1}\left\{\frac{p^2+q^2+1}{3pq(z_n-1)}+\frac{p^2+q^2+13}{3pq(z_n-1)^2}
+\frac{8}{pq(z_n-1)^3}+\frac{4}{pq(z_n-1)^4} \right. \notag\\
&\quad \left. +\sum_{j=1}^{p-1}\frac{2\omega_j^2(\omega_j^q+1)}{p(\omega_j-1)^2(\omega_j^q-1)(z_n-\o_j)}
+\sum_{j=1}^{q-1}\frac{2\xi_j^2(\xi_j^p+1)}{q(\xi_j-1)^2(\xi_j^p-1)(z_n-\xi_j)} \right\},
\end{align}
where $\x_j=z_j^p=e^{2\pi ij/q}.$
\nopagebreak
Now, we examine 
\begin{align*}
\sum_{\substack{n=1\\ p \nmid n, \, q \nmid n}}^{pq-1}\frac{1}{x-z_n}=&\frac{f'}{f}
=\frac{pqx^{pq-1}}{x^{pq}-1}+\frac{1}{x-1}
-\sum_{a=p, q}\frac{ax^{a-1}}{x^{a}-1},\notag\\ 
\sum_{\substack{n=1\\ p \nmid n, \, q \nmid n}}^{pq-1}\frac{1}{(x-z_n)^2}=&-\Big(\frac{f'}{f}\Big)'
=\frac{pqx^{pq-2}(x^{pq}+pq-1)}{(x^{pq}-1)^2} \notag \\
& +\frac{1}{(x-1)^2} -\sum_{a=p, q}\frac{ax^{a-2}(x^{a}+a-1)}{(x^{a}-1)^2},\notag\\
\sum_{\substack{n=1\\ p \nmid n, \, q \nmid n}}^{pq-1}\frac{1}{(x-z_n)^3}=&\frac12\Big(\frac{f'}{f}\Big)'' 
=\frac{pqx^{pq-3}\big(2x^{2pq}+(p^2q^2+3pq-4)x^{pq}+p^2q^2-3pq+2\big)}{2(x^{pq}-1)^3} \notag\\
&+\frac{1}{(x-1)^3}-\sum_{a=p, q}\frac{ax^{a-3}\big(2x^{2a}+(a^2+3a-4)x^{a}+a^2-3a+2\big)}{2(x^{a}-1)^3}, \notag\\
\sum_{\substack{n=1\\ p \nmid n, \, q \nmid n}}^{pq-1}\frac{1}{(x-z_n)^4}=&-\frac16\Big(\frac{f'}{f}\Big)'''
=\frac{pqx^{pq-4}\big(6x^{3pq}+(p^3q^3+6p^2q^2+11pq-18)x^{2pq}\big)}{6(x^{pq}-1)^4} \notag\\
&+\frac{pqx^{pq-4}\big((4p^3q^3-22pq+18)x^{pq}+p^3q^3-6p^2q^2+11pq-6\big)}{6(x^{pq}-1)^4} \notag\\
&+\frac{1}{(x-1)^4}-\sum_{a=p, q}\left\{\frac{ax^{a-4}\big(6x^{3a}+(a^3+6a^2+11a-18)x^{2a}\big)}{6(x^{a}-1)^4}\right. \notag\\
&   \left.+ \frac{ax^{a-4}\big((4a^3-22a+18)x^{a}+a^3-6a^2+11a-6\big)}{6(x^{a}-1)^4} \right\}.
\end{align*}

Evaluating these identities at $x=1,$  we obtain

\begin{equation}\label{pqz1}
  \begin{split}
\sum_{\substack{n=1\\ p \nmid n, \, q \nmid n}}^{pq-1}\frac{1}{z_n-1}&=-\frac{(p-1)(q-1)}{2}, \\
\sum_{\substack{n=1\\ p \nmid n, \, q \nmid n}}^{pq-1}\frac{1}{(z_n-1)^2}&=-\frac{(p^2-1)(q^2-1)}{12}+\frac{(p-1)(q-1)}{2}, \\
\sum_{\substack{n=1\\ p \nmid n, \, q \nmid n}}^{pq-1}\frac{1}{(z_n-1)^3}&
=\frac{1}{8}(pq+p+q-3)(pq-p-q+1),\\
\sum_{\substack{n=1\\ p \nmid n, \, q \nmid n}}^{pq-1}\frac{1}{(z_n-1)^4}&
=\frac{1}{720}(p^4q^4 -110p^2q^2 -p^4 +110p^2 -q^4 +110q^2-109)+\frac12(p-1)(q-1).
\end{split}
\end{equation}
Next, note that
\begin{align*}
\frac{f'}{f}
&=\frac{pqx^{pq-1}}{x^{pq}-1}+\frac{1}{x-1}-\frac{px^{p-1}}{x^{p}-1}-\frac{qx^{q-1}}{x^{q}-1}\\
&=p\frac{(q-1)x^{pq}-x^{p(q-1)}-\cdots-x^{2p}-x^p}{x(x^{pq}-1)}+\frac{1}{x-1}-\frac{qx^{q-1}}{x^{q}-1}\\
&=p\frac{(q-1)x^{p(q-1)}+(q-2)x^{p(q-2)}+\cdots+x^p}{x(x^{p(q-1)}+\cdots+x^p+1)}
+\frac{1}{x-1}-\frac{qx^{q-1}}{x^{q}-1}.
\end{align*}
Therefore, it follows that
\begin{align}\label{pqo}
\sum_{\substack{n=1\\ p \nmid n, \, q \nmid n}}^{pq-1}\frac{1}{z_n-\o_j}&=-\Big[\frac{f'}{f}\Big]_{\o_j}
=-\frac{p(q-1)}{2\o_j}-\frac{1}{\o_j-1}+\frac{q\o_j^{q-1}}{\o_j^q-1}.
\end{align}
Analogously, we have
\begin{align}\label{pqxi}
\sum_{\substack{n=1\\ p \nmid n, \, q \nmid n}}^{pq-1}\frac{1}{z_n-\xi_j}&=-\Big[\frac{f'}{f}\Big]_{\xi_j}
=-\frac{q(p-1)}{2\xi_j}-\frac{1}{\xi_j-1}+\frac{p\xi_j^{p-1}}{\xi_j^p-1}.
\end{align}

Now, by \eqref{pqz1}, we obtain
\begin{align}\label{pqe1}
&\sum_{\substack{n=1\\ p \nmid n, \, q \nmid n}}^{pq-1}\left\{\frac{p^2+q^2+1}{3pq(z_n-1)}+\frac{p^2+q^2+13}{3pq(z_n-1)^2}
+\frac{8}{pq(z_n-1)^3}+\frac{4}{pq(z_n-1)^4} \right\} \notag\\
&=\frac{1}{180pq}(p^4q^4-5p^4q^2-5p^2q^4+4p^4+4q^4+15p^2q^2-10q^2-10p^2+6).
\end{align}
Also, with the use of \eqref{pqo} and Lemmas \ref{zz1} and \ref{lemr1}, we derive
\begin{align}\label{pqe2}
&\sum_{\substack{n=1\\ p \nmid n, \, q \nmid n}}^{pq-1}\sum_{j=1}^{p-1}\frac{2\omega_j^2(\omega_j^q+1)}{p(\omega_j-1)^2(\omega_j^q-1)(z_n-\o_j)} \notag\\
&=\sum_{j=1}^{p-1}\left\{\frac{2\omega_j^2}{p(\omega_j-1)^2}+\frac{4\omega_j^2}{p(\omega_j-1)^2(\omega^q_j-1)}\right\}
\left\{-\frac{p(q-1)}{2\o_j}-\frac{1}{\o_j-1}+\frac{q\o_j^{q-1}}{\o_j^q-1}\right\}\notag\\
&=-(q-1)\sum_{j=1}^{p-1}\frac{\o_j}{(\o_j-1)^2}-\frac{2}{p}\sum_{j=1}^{p-1}\frac{\o_j^2}{(\o_j-1)^3}
+\frac{2q}{p}\sum_{j=1}^{p-1}\frac{\o_j^{q+1}}{(\o_j-1)^2(\o_j^q-1)}\notag\\
&\quad -2(q-1)\sum_{j=1}^{p-1}\frac{\o_j}{(\o_j-1)^2(\o_j^q-1)}-\frac{4}{p}\sum_{j=1}^{p-1}\frac{\o_j^2}{(\o_j-1)^3(\o_j^q-1)}
+\frac{4q}{p}\sum_{j=1}^{p-1}\frac{\o_j^{q+1}}{(\o_j-1)^2(\o_j^q-1)^2} \notag\\
&=-(q-1)\sum_{j=1}^{p-1}\frac{\o_j}{(\o_j-1)^2}-\frac{2}{p}\sum_{j=1}^{p-1}\frac{\o_j^2}{(\o_j-1)^3}
+\frac{2q}{p}\sum_{j=1}^{p-1}\left\{\frac{\o_j}{(\o_j-1)^2}+\frac{\o_j}{(\o_j-1)^2(\o_j^q-1)}\right\}\notag\\
&\qquad -2(q-1)\sum_{j=1}^{p-1}\frac{\o_j}{(\o_j-1)^2(\o_j^q-1)}
-\frac{4}{p}\sum_{j=1}^{p-1}\left\{\frac{\o_j}{(\o_j-1)^2(\o_j^q-1)}+\frac{\o_j}{(\o_j-1)^3(\o_j^q-1)}\right\}\notag\\
&\qquad \qquad \qquad +\frac{4q}{p}\sum_{j=1}^{p-1}\left\{\frac{\o_j}{(\o_j-1)^2(\o_j^q-1)}+
\frac{\o_j}{(\o_j-1)^2(\o_j^q-1)^2}\right\}\notag\\
&=\frac{2q-pq+p}{p}\sum_{j=1}^{p-1}\frac{\omega_j}{(\omega_j-1)^2}-\frac{2}{p}
\sum_{j=1}^{p-1}\frac{\omega^2_j}{(\omega_j-1)^3}+\frac{6q-2pq+2p-4}{p}\sum_{j=1}^{p-1}\frac{\o_j}{(\o_j-1)^2(\o_j^q-1)}\notag\\
&\qquad -\frac{4}{p}\sum_{j=1}^{p-1}\frac{\o_j}{(\o_j-1)^3(\o_j^q-1)}+\frac{4q}{p}\sum_{j=1}^{p-1}\frac{\o_j}{(\o_j-1)^2(\o_j^q-1)^2}\notag\\
&=\frac{1}{12p}(p^2-1)(q-1)-\frac{4}{p}\sum_{j=1}^{p-1}\frac{\o_j}{(\o_j-1)^3(\o_j^q-1)}
+\frac{4q}{p}\sum_{j=1}^{p-1}\frac{\o_j}{(\o_j-1)^2(\o_j^q-1)^2}.
\end{align}
Analogously, by  \eqref{pqxi} and Lemma \ref{zz1}, we find that
\begin{align}\label{pqe4}
&\sum_{\substack{n=1\\ p \nmid n, \, q \nmid n}}^{pq-1}
\sum_{j=1}^{q-1}\frac{2\xi_j^2(\xi_j^p+1)}{q(\xi_j-1)^2(\xi_j^p-1)(z_n-\xi_j)} \notag\\
&=\frac{1}{12q}(q^2-1)(p-1)-\frac{4}{q}\sum_{j=1}^{q-1}\frac{\xi_j}{(\xi_j-1)^3(\xi_j^p-1)}
+\frac{4p}{q}\sum_{j=1}^{q-1}\frac{\xi_j}{(\xi_j-1)^2(\xi_j^p-1)^2}\notag\\
&=\frac{1}{180pq}(-4p^4+10p^2+25p^2q^2+45p^2q-25q^2-45q-6) \notag\\
&\qquad +\frac{12}{p}\sum_{j=1}^{p-1}\frac{\o_j}{(\o_j-1)^3(\o_j^q-1)}+\frac{4q}{p}\sum_{j=1}^{p-1}\frac{\o_j}{(\o_j-1)^2(\o_j^q-1)^2},
\end{align}
where we employed Lemmas  \ref{lemr2} and \ref{lemr4}.
Referring to  \eqref{ccsR} and \eqref{Rz}, and adding the identities in \eqref{pqe1}, \eqref{pqe2} and \eqref{pqe4}, we complete our proof.
\end{proof}

McIntosh \cite[p.~202]{mcintosh}, and G.~R.~H.~Greaves, R.~R.~Hall, M.~N.~Huxley, and J.~C.~ Wilson \cite[p.~65, Equation (29)]{greaves}  established a representation for
\begin{equation}\label{mc6}
 \sum_{\substack{n=1\\ p \nmid n, \, q \nmid n}}^{pq-1}
\dfrac{\cot\Big(\dfrac{\pi n}{p}\Big)\cot\Big(\dfrac{\pi n}{q}\Big)}{\sin^2\Big(\dfrac{\pi n}{pq}\Big)}
\end{equation}
in terms of the generalized Dedekind sums $S_3(p,q)$ and $S_3(q,p)$, defined in \eqref{mc2}.

\begin{corollary}\label{pqcor1}
Let $p$ and $q$ be  positive integers such that $p, q\geq 2$ and $q\equiv \pm1 \pmod{p}.$ Then,
\begin{align*}
&\sum_{\substack{n=1\\ p \nmid n, \, q \nmid n}}^{pq-1}
\dfrac{\cot\Big(\dfrac{\pi n}{p}\Big)\cot\Big(\dfrac{\pi n}{q}\Big)}{\sin^2\Big(\dfrac{\pi n}{pq}\Big)}\\
&\qquad =\begin{cases}
\dfrac{1}{45p}(p^2-1)(p^2-4)(q-1)^2(q+2) & \text{if} ~ q \equiv 1 \pmod{p},\\[0.2in]
\dfrac{1}{45p}(p^2-1)(p^2-4)(q+1)^2(q-2) & \text{if} ~ q \equiv -1 \pmod{p}.
\end{cases}
\end{align*}
\end{corollary}

\begin{proof} Suppose that $q\equiv 1 \pmod{p}.$ Then, Theorem \ref{ccspq} and Lemma \ref{zz1} imply
\begin{align*}
&\sum_{\substack{n=1\\ p \nmid n, \, q \nmid n}}^{pq-1}
\dfrac{\cot\Big(\dfrac{\pi n}{p}\Big)\cot\Big(\dfrac{\pi n}{q}\Big)}{\sin^2\Big(\dfrac{\pi n}{pq}\Big)}\\
&=\frac{1}{45p}(p^4q^3-5p^4q+4q^3-5p^2q^3+55p^2q+30p^2-50q-30) \\
&\qquad +\frac{32}{p}\sum_{j=1}^{p-1}\frac{\o_j}{(\o_j-1)^4}+\frac{32q}{p}\sum_{j=1}^{p-1}\frac{\o_j}{(\o_j-1)^4}\\
&=\frac{1}{45p}(p^4q^3-5p^4q-5p^2q^3-5p^2q+4q^3+10q)+\frac{32(q+1)}{p}\cdot\frac{(p-1)(p+1)(p^2-19)}{720}\\
&=\frac{1}{45p}(p^4q^3-3p^4q-5p^2q^3+2p^4+4q^3+15p^2q-10p^2-12q+8) \\
&=\frac{1}{45p}(p^2-1)(p^2-4)(q-1)^2(q+2).
\end{align*}
Similarly, if $q\equiv -1 \pmod{p},$ then
\begin{align*}
&\sum_{\substack{n=1\\ p \nmid n, \, q \nmid n}}^{pq-1}
\dfrac{\cot\Big(\dfrac{\pi n}{p}\Big)\cot\Big(\dfrac{\pi n}{q}\Big)}{\sin^2\Big(\dfrac{\pi n}{pq}\Big)} \notag\\
&=\frac{1}{45p}(p^4q^3-5p^4q+4q^3-5p^2q^3+55p^2q+30p^2-50q-30)\\
&\qquad -\frac{32}{p}\sum_{j=1}^{p-1}\frac{\o_j^2}{(\o_j-1)^4}+\frac{32q}{p}\sum_{j=1}^{p-1}\frac{\o_j^3}{(\o_j-1)^4}\\
&=\frac{1}{45p}(p^4q^3-5p^4q+4q^3-5p^2q^3+55p^2q+30p^2-50q-30)\\
&\qquad -\frac{32}{p}\cdot\frac{(p^2-1)(p^2+11)}{720}+\frac{32q}{p}\cdot\frac{(p^2-1)(p^2-19)}{720}\\
&=\frac{1}{45p}(p^4q^3-3p^4q-5p^2q^3-2p^4+4q^3+15p^2q+10p^2-12q-8)\\
&=\dfrac{1}{45p}(p^2-1)(p^2-4)(q+1)^2(q-2).
\end{align*}
Hence, we complete the proof.
\end{proof}

\section{Further Lemmas on Sums Involving Roots of Unity}

From Lemma \ref{1z1}, we can easily derive the following lemma.
\begin{lemma}\label{zz+1}
Let $k$ be an odd positive integer and  $z_n=e^{2\pi in/k}.$ Then, we have
\begin{align*}
&\sum_{n=1}^{k-1}\frac{z_n}{z_n+1}=\frac{k-1}{2},
\qquad \qquad \qquad \quad \sum_{n=1}^{k-1}\frac{z_n}{(z_n+1)^2}=\frac{k^2-1}{4}, \\
 &\sum_{n=1}^{k-1}\frac{z_n^2}{(z_n+1)^2}=-\frac{(k-1)^2}{4},
\qquad \qquad  \sum_{n=1}^{k-1}\frac{z_n}{(z_n+1)^3}=\frac{k^2-1}{8},  \\
&\sum_{n=1}^{k-1}\frac{z_n^2}{(z_n+1)^3}=\frac{k^2-1}{8},
\qquad \qquad  \qquad \sum_{n=1}^{k-1}\frac{z_n^3}{(z_n+1)^3}=-\frac{(k-1)(3k-1)}{8}, \\
&\sum_{n=1}^{k-1}\frac{z_n}{(z_n+1)^4}=-\frac{(k^2-1)(k^2-3)}{48},
\quad \sum_{n=1}^{k-1}\frac{z_n^2}{(z_n+1)^4}=\frac{(k^2-1)(k^2+3)}{48}, \\
&\sum_{n=1}^{k-1}\frac{z_n^3}{(z_n+1)^4}=-\frac{(k^2-1)(k^2-3)}{48},
\quad \sum_{n=1}^{k-1}\frac{z_n^4}{(z_n+1)^4}=\frac{(k-1)(k^3+k^2-21k+3)}{48}.
\end{align*}
\end{lemma}

\begin{lemma} \label{R2}
Let $p$ and $q$ be positive integers  with $p, q \geq 2$ and $(p, q)=1.$
If $\omega_j=e^{2\pi ij/p}$ and $\xi_j=e^{2\pi ij/q},$ then
\begin{align*}
&\frac{x(x^p+1)(x^q+1)}{(x+1)^2(x^p-1)(x^q-1)}=\frac{1}{pq(x-1)}+\frac{1}{pq(x-1)^2} \\
&\qquad \qquad  +\frac{2}{p}\sum_{j=1}^{p-1}\frac{\omega_j^2(\omega_j^q+1)}{(\omega_j+1)^2(\omega_j^q-1)(x-\o_j)}
+\frac{2}{q}\sum_{j=1}^{q-1}\frac{\xi_j^2(\xi_j^p+1)}{(\xi_j+1)^2(\xi_j^p-1)(x-\xi_j)}.
\end{align*}
\end{lemma}

\begin{proof} We first use \eqref{xp1} to obtain
\begin{align}\label{B-}
&\frac{x(x^p+1)(x^q+1)}{(x+1)^2(x^p-1)(x^q-1)}=\frac{x}{(x+1)^2}\Big(1+\frac{2}{x^p-1}\Big)\Big(1+\frac{2}{x^q-1}\Big)\notag\\
=&\frac{x}{(x+1)^2}\Big(1+\frac{2}{p}\sum_{j=1}^p\frac{\o_j}{x-\omega_j}\Big)
\Big(1+\frac{2}{q}\sum_{j=1}^q\frac{\xi_j}{x-\xi_j}\Big)\notag\\
=&\frac{x}{(x+1)^2}\Big(1+\frac{2}{p(x-1)}+\frac{2}{p}\sum_{j=1}^{p-1}\frac{\o_j}{x-\omega_j}\Big)
\Big(1+\frac{2}{q(x-1)}+\frac{2}{q}\sum_{j=1}^{q-1}\frac{\xi_j}{x-\xi_j}\Big)\notag\\
=&\frac{x}{(x+1)^2}+\frac{2x}{p(x+1)^2(x-1)}+\frac{2x}{q(x+1)^2(x-1)}+\frac{4x}{pq(x+1)^2(x-1)^2}\notag\\
&+\frac{2}{p}\sum_{j=1}^{p-1}\frac{\o_j x}{(x+1)^2(x-\omega_j)}
+\frac{2}{q}\sum_{j=1}^{q-1}\frac{\xi_j x}{(x+1)^2(x-\xi_j)}
+\frac{4}{pq}\sum_{j=1}^{p-1}\frac{\o_j x}{(x+1)^2(x-1)(x-\omega_j)}\notag\\
&+\frac{4}{pq}\sum_{j=1}^{p-1}\frac{\xi_j x}{(x+1)^2(x-1)(x-\xi_j)}
+\frac{4}{pq}\sum_{i=1}^{p-1}\sum_{j=1}^{q-1}\frac{\o_i\xi_j x}{(x+1)^2(x-\o_i)(x-\xi_j)}.
\end{align}
Now, we consider the following partial fraction decompositions:
\begin{align}
&\frac{x}{(x+1)^2}=\frac{1}{x+1}-\frac{1}{(x+1)^2}, \label{p5}\\
&\frac{x}{(x+1)^2(x-a)}=-\frac{a}{(a+1)^2(x+1)}+\frac{1}{(a+1)(x+1)^2}+\frac{a}{(a+1)^2(x-a)}, \label{p6}\\
&\frac{x}{(x+1)^2(x-1)^2}=\frac{1}{4}\left(\frac{1}{(x-1)^2}-\frac{1}{(x+1)^2}\right), \label{p7}\\
&\frac{x}{(x+1)^2(x-a)(x-b)}=\frac{ab-1}{(a+1)^2(b+1)^2(x+1)}-\frac{1}{(a+1)(b+1)(x+1)^2} \notag \\
&\qquad \qquad \qquad \qquad \qquad\,\,\, +\frac{a}{(a+1)^2(a-b)(x-a)}+\frac{b}{(b+1)^2(b-a)(x-b)}. \label{p8}
\end{align}
We divide the right-hand side of \eqref{B-} into six parts, $B_1, \ldots, B_6$, and find the partial fraction decomposition of each.
By \eqref{p5}, \eqref{p6}, and \eqref{p7},
\begin{align}\label{B-1}
&B_1:=\frac{x}{(x+1)^2}+\frac{2x}{p(x+1)^2(x-1)}+\frac{2x}{q(x+1)^2(x-1)}+\frac{4x}{pq(x+1)^2(x-1)^2} \notag\\
&=\frac{1}{x+1}-\frac{1}{(x+1)^2}+\left(\frac{2}{p}+\frac{2}{q}\right)\left\{-\frac{1}{4(x+1)}+\frac{1}{2(x+1)^2}+\frac{1}{4(x-1)}\right\}
\notag\\
&\qquad \qquad +\frac{1}{pq}\left(\frac{1}{(x-1)^2}-\frac{1}{(x+1)^2}\right) \notag\\
&=\frac{(p+q)}{2pq}\frac{1}{(x-1)}+\frac{1}{pq(x-1)^2}+\frac{(2pq-p-q)}{2pq}\frac{1}{(x+1)}-\frac{(p-1)(q-1)}{pq}\frac{1}{(x+1)^2}.
\end{align}
From \eqref{p6}, it follows that
\begin{align} \label{B-2}
B_2&:=\frac{2}{p}\sum_{j=1}^{p-1}\frac{\o_j x}{(x+1)^2(x-\omega_j)} \notag\\
&=\frac{2}{p}\sum_{j=1}^{p-1}\left\{\frac{-\o_j^2}{(\o_j+1)^2(x+1)}+\frac{\o_j}{(\o_j+1)(x+1)^2}+\frac{\o_j^2}{(\o_j+1)^2(x-\o_j)}\right\},
\end{align}
and
\begin{align} \label{B-3}
B_3&:=\frac{2}{q}\sum_{j=1}^{q-1}\frac{\xi_j x}{(x+1)^2(x-\xi_j)} \notag\\
&=\frac{2}{q}\sum_{j=1}^{q-1}\left\{\frac{-\xi_j^2}{(\xi_j+1)^2(x+1)}+\frac{\xi_j}{(\xi_j+1)(x+1)^2}+\frac{\xi_j^2}{(\xi_j+1)^2(x-\xi_j)}\right\}.
\end{align}
Similarly, using \eqref{p8}, we obtain
\begin{align} \label{D4}
B_4&:=\frac{4}{pq}\sum_{j=1}^{p-1}\frac{\o_j x}{(x+1)^2(x-1)(x-\o_j)}\notag\\
&=\frac{4}{pq}\sum_{j=1}^{p-1}\left\{\frac{\o_j(\o_j-1)}{4(\o_j+1)^2(x+1)}-\frac{\o_j}{2(\o_j+1)(x+1)^2}\right. \notag \\
&\qquad \qquad  \left.+\frac{\o_j}{4(1-\o_j)(x-1)}+\frac{\o_j^2}{(\o_j+1)^2(\o_j-1)(x-\o_j)}\right\},
\end{align}
\begin{align} \label{B-5}
B_5&:=\frac{4}{pq}\sum_{j=1}^{p-1}\frac{\xi_j x}{(x+1)^2(x-1)(x-\xi_j)}\notag\\
&=\frac{4}{pq}\sum_{j=1}^{p-1}\left\{\frac{\xi_j(\xi_j-1)}{4(\xi_j+1)^2(x+1)}-\frac{\xi_j}{2(\xi_j+1)(x+1)^2}\right. \notag \\
&\qquad \qquad  \left.+\frac{\xi_j}{4(1-\xi_j)(x-1)}+\frac{\xi_j^2}{(\xi_j+1)^2(\xi_j-1)(x-\xi_j)}\right\},
\end{align}\and
\begin{align} \label{B-6}
B_6&:=\frac{4}{pq}\sum_{i=1}^{p-1}\sum_{j=1}^{q-1}\frac{\o_i\xi_j x}{(x+1)^2(x-\o_i)(x-\xi_j)}\notag\\
&=\frac{4}{pq}\sum_{i=1}^{p-1}\left\{\sum_{j=1}^{q-1}\frac{\o_i\xi_j(\o_i\xi_j-1)}{(\o_i+1)^2(\xi_j+1)^2(x+1)}
-\frac{\o_i\xi_j}{(\o_i+1)(\xi_j+1)(x+1)^2} \right.\notag \\
&\qquad \qquad \left. +\frac{\o_i^2\xi_j}{(\o_i+1)^2(\o_i-\xi_j)(x-\o_i)}+\frac{\o_i\xi_j^2}{(\xi_j+1)^2(\xi_j-\o_i)(x-\xi_j)}\right\}.
\end{align}
Next, we substitute \eqref{B-1}--\eqref{B-6} into \eqref{B-} and calculate the coefficient of each term.
First, by Lemma \ref{zz1}, we see that the coefficient of $\dfrac{1}{x-1}$ is
\begin{align*}
\frac{(p+q)}{2pq}+\frac{1}{pq}\sum_{j=1}^{p-1}\frac{\o_j}{(1-\o_j)}+\frac{1}{pq}\sum_{j=1}^{q-1}\frac{\xi_j}{(1-\xi_j)}=\frac{1}{pq},
\end{align*}
and the coefficient of $\dfrac{1}{(x-1)^2}$ is $\dfrac{1}{pq}.$
Also, the coefficient of $\dfrac{1}{x+1}$ is equal to
\begin{align*}
&\frac{(2pq-p-q)}{2pq}-\frac{2}{p}\sum_{j=1}^{p-1}\frac{\o_j^2}{(\o_j+1)^2}
-\frac{2}{q}\sum_{j=1}^{q-1}\frac{\xi_j^2}{(\xi_j+1)^2} \notag\\
&+\frac{1}{pq}\sum_{j=1}^{p-1}\frac{\o_j(\o_j-1)}{(\o_j+1)^2}
+\frac{1}{pq}\sum_{j=1}^{p-1}\frac{\xi_j(\xi_j-1)}{(\xi_j+1)^2}+
\frac{4}{pq}\sum_{i=1}^{p-1}\sum_{j=1}^{q-1}\frac{\o_i\xi_j(\o_i\xi_j-1)}{(\o_i+1)^2(\xi_j+1)^2} \notag\\
=&\frac{(2pq-p-q)}{2pq}+\frac{(p-1)^2}{2p}+\frac{(q-1)^2}{2q}-\frac{1}{4pq}\Big((p-1)^2+(p^2-1)\Big)\notag\\
&-\frac{1}{4pq}\Big((q-1)^2+(q^2-1)\Big)+\frac{1}{4pq}\Big((p-1)^2(q-1)^2-(p^2-1)(q^2-1)\Big)=0,
\end{align*}
where we used Lemma \ref{zz+1}. Analogously, the coefficient of $\dfrac{1}{(x+1)^2}$ is
\begin{align*}
&-\frac{(p-1)(q-1)}{pq}+\frac{2}{p}\sum_{j=1}^{p-1}\frac{\o_j}{(\o_j+1)}+\frac{2}{q}\sum_{j=1}^{q-1}\frac{\xi_j}{(\xi_j+1)}
-\frac{2}{pq}\sum_{j=1}^{p-1}\frac{\o_j}{(\o_j+1)}\notag\\
&-\frac{2}{pq}\sum_{j=1}^{q-1}\frac{\xi_j}{(\xi_j+1)}
-\frac{4}{pq}\sum_{i=1}^{p-1}\frac{\o_i}{(\o_i+1)}\sum_{j=1}^{q-1}\frac{\xi_j}{(\xi_j+1)}=0.
\end{align*}
Lastly, using the same argument in \eqref{B5}, we find that the coefficient of $\dfrac{1}{x-\o_j},$ $1\leq j \leq p-1,$ is
\begin{align*}
&\frac{2\o_j^2}{p(\o_j+1)^2}\left\{1+\frac{2}{q}\frac{1}{\o_j-1}
+\frac{2}{q}\sum_{i=1}^{q-1}\frac{\xi_i}{(\o_j-\xi_i)}\right\}=\frac{2\o_j^2(\o_j^q+1)}{p(\o_j+1)^2(\o_j^q-1)},
\end{align*}
and the coefficient of $\dfrac{1}{x-\xi_j},$ $1\leq j \leq p-1,$ is
\begin{align*}
\frac{2\xi_j^2(\xi_j^p+1)}{q(\xi_j+1)^2(\xi_j^p-1)}.
\end{align*}
Putting all these together in \eqref{B-}, we finish our proof.
\end{proof}

\begin{lemma}\label{lemR21}
Let $p$ and $q$ be relatively prime positive integers such that  $q \geq 3$ and $q$ is odd.
If $\xi_j=e^{2\pi ij/q}$, then
\begin{align*}
\sum_{j=1}^{q-1}\frac{\xi_j}{(\xi_j+1)^2(\xi_j^p-1)}=-\frac{q^2-1}{8}.
\end{align*}
\end{lemma}

\begin{proof}
We observe that  by symmetry
\begin{align*}
\sum_{n=1}^{q-1}\cot\Big(\frac{\pi np}{q}\Big)\sec^2\Big(\frac{\pi n}{q}\Big)=0.
\end{align*}
Also,
\begin{align*}
\sum_{n=1}^{q-1}\cot\Big(\frac{\pi np}{q}\Big)\sec^2\Big(\frac{\pi n}{q}\Big)
&=4i\sum_{n=1}^{q-1}\frac{\xi_n^p+1}{\xi_n^p-1}\frac{\xi_n}{(\xi_n+1)^2} \notag\\
&=4i\sum_{n=1}^{q-1}\left(1+\frac{2}{\xi_n^p-1}\right)\frac{\xi_n}{(\xi_n+1)^2}.
\end{align*}
Hence, by Lemma \ref{zz+1}, we have
\begin{align*}
\sum_{n=1}^{q-1}\frac{\xi_n}{(\xi_n+1)^2(\xi_n^p-1)}=-\frac12\sum_{n=1}^{q-1}\frac{\xi_n}{(\xi_n+1)^2}
=-\frac{q^2-1}{8}.
\end{align*}
\end{proof}

\begin{lemma}\label{lemR22}
Let $p$ and $q$ be relatively prime positive integers  such that $p, q \geq 2.$ If
$\omega_j=e^{2\pi ij/p}$ and $\xi_j=e^{2\pi ij/q},$ then
\begin{align*}
q\sum_{j=1}^{p-1}\frac{1}{(\o_j-1)(\o_j^q-1)}+p\sum_{j=1}^{q-1}\frac{1}{(\xi_j-1)(\xi_j^p-1)}
=-\frac{1}{12}(p^2+q^2-9pq+3p+3q+1).
\end{align*}
\end{lemma}

\begin{proof}
From \eqref{xp1}, \eqref{1xz} and Lemmas \ref{1z-1} and \ref{zz1}, we can derive
\begin{align*}
&q\sum_{j=1}^{p-1}\frac{1}{(\o_j-1)(\o_j^q-1)}=\sum_{j=1}^{p-1}\left\{\frac{1}{(\o_j-1)^2}
+\sum_{i=1}^{q-1}\frac{\xi_i}{(\o_j-1)(\o_j-\xi_i)} \right\} \notag\\
&=-\frac{(p-1)(p-5)}{12}+\sum_{i=1}^{q-1}\frac{\xi_i}{(1-\xi_i)}\sum_{j=1}^{p-1}\left\{\frac{1}{\o_j-1}-\frac{1}{\o_j-\xi_i}\right\} \notag\\
&=-\frac{(p-1)(p-5)}{12}+\frac{(q-1)}{2}\frac{(p-1)}{2}-
\sum_{i=1}^{q-1}\frac{\xi_i}{(1-\xi_i)}\left\{\frac{1}{\xi_i-1}-\frac{p}{\xi_i}-\frac{p}{\xi_i(\xi_i^p-1)} \right\}\notag\\
&=-\frac{(p-1)(p-5)}{12}+\frac{(q-1)}{2}\frac{(p-1)}{2}-\frac{(q-1)(q+1)}{12}+\frac{p(q-1)}{2}-p\sum_{j=1}^{q-1}\frac{1}{(\xi_j-1)(\xi_j^p-1)} \notag\\
&=-\frac{1}{12}(p^2+q^2-9pq+3p+3q+1)-p\sum_{j=1}^{q-1}\frac{1}{(\xi_j-1)(\xi_j^p-1)}.
\end{align*}
This completes our proof.
\end{proof}

\begin{lemma}\label{lemr21+}
Let $p$ and $q$ be relatively prime positive integers  such that $p, q \geq 2.$ If
$\omega_n=e^{2\pi in/p}$ and $\xi_n=e^{2\pi in/q},$ then
\begin{align*}
q\sum_{n=1}^{p-1}\frac{1}{(\o_n+1)(\o_n^q-1)}+p\sum_{n=1}^{q-1}\frac{1}{(\xi_n+1)(\xi_n^p-1)}
=-\frac{1}{4}(3pq-p-q-1).
\end{align*}
\end{lemma}

\begin{proof}
Using \eqref{xp1} and Lemmas \ref{1z1} and \ref{1z-1}, we have
\begin{align}\label{r21+e1}
&q\sum_{n=1}^{p-1}\frac{1}{(\o_n+1)(\o_n^q-1)}
=\sum_{n=1}^{p-1}\frac{1}{(\o_n+1)}\left\{\frac{1}{\o_n-1} +\sum_{j=1}^{q-1}\frac{\xi_j}{\o_n-\xi_j}\right\} \notag\\
&=\sum_{n=1}^{p-1}\frac{1}{(\o_n-1)(\o_n+1)}+\sum_{j=1}^{q-1}\xi_j\sum_{n=1}^{p-1}\frac{1}{(\o_n+1)(\o_n-\xi_j)}\notag\\
&=\frac{1}{2}\sum_{n=1}^{p-1}\left(\frac{1}{\o_n-1}-\frac{1}{\o_n+1}\right)
-\sum_{j=1}^{q-1}\frac{\xi_j}{\xi_j+1}\sum_{n=1}^{p-1}\left(\frac{1}{\o_n+1}-\frac{1}{\o_n-\xi_j}\right)\notag\\
&=-\frac{(p-1)(q+1)}{4}+\sum_{j=1}^{q-1}\frac{\xi_j}{\xi_j+1}\sum_{n=1}^{p-1}\frac{1}{\o_n-\xi_j}.
\end{align}
Now, by \eqref{1xz},
\begin{align}\label{r21+e2}
&\sum_{j=1}^{q-1}\frac{\xi_j}{\xi_j+1}\sum_{n=1}^{p-1}\frac{1}{\o_n-\xi_j}=
\sum_{j=1}^{q-1}\frac{\xi_j}{\xi_j+1}\left\{\frac{1}{\xi_j-1}-\frac{p}{\xi_j}-\frac{p}{\xi_j(\xi_j^p-1)}\right\}\notag\\
&=\frac{1}{2}\sum_{j=1}^{q-1}\left(\frac{1}{\xi_j+1}+\frac{1}{\xi_j-1}\right)
-p\sum_{j=1}^{q-1}\frac{1}{\xi_j+1}-p\sum_{n=1}^{q-1}\frac{1}{(\xi_n+1)(\xi_n^p-1)} \notag\\
&=-\frac{p(q-1)}{2}-p\sum_{n=1}^{q-1}\frac{1}{(\xi_n+1)(\xi_n^p-1)},
\end{align}
where we applied Lemmas \ref{1z1} and \ref{1z-1}.
Putting \eqref{r21+e2} into \eqref{r21+e1}, we finish the proof of the lemma.
\end{proof}

\begin{lemma}\label{lemR23}
Let $p$ and $q$ be relatively prime odd positive integers  such that $p, q \geq 3.$ If
$\omega_n=e^{2\pi in/p}$ and $\xi_n=e^{2\pi in/q},$ then
\begin{align*}
q\sum_{n=1}^{p-1}\frac{1}{(\o_n+1)^2(\o_n^q-1)}+p\sum_{n=1}^{q-1}\frac{1}{(\xi_n+1)^2(\xi_n^p-1)}
=\frac{1}{8}(p^2q+pq^2+p+q-6pq+2).
\end{align*}
\end{lemma}

\begin{proof}
Applying Lemmas \ref{lemR21} and \ref{lemr21+}, we obtain
\begin{align*}
&q\sum_{n=1}^{p-1}\frac{1}{(\o_n+1)^2(\o_n^q-1)}+p\sum_{n=1}^{q-1}\frac{1}{(\xi_n+1)^2(\xi_n^p-1)} \notag\\
&=q\sum_{n=1}^{p-1}\left\{\frac{1}{(\o_n+1)(\o_n^q-1)}-\frac{\o_j}{(\o_n+1)^2(\o_n^q-1)}\right\}\notag\\
&\quad+p\sum_{n=1}^{q-1}\left\{\frac{1}{(\xi_n+1)(\xi_n^p-1)}-\frac{\xi_j}{(\xi_n+1)^2(\xi_n^p-1)}\right\} \notag\\
&=-\frac{1}{4}(3pq-p-q-1)+\frac{q}{8}(p^2-1)+\frac{p}{8}(q^2-1) \notag\\
&=\frac{1}{8}(p^2q+pq^2+p+q-6pq+2),
\end{align*}
which completes the proof.
\end{proof}

\begin{lemma}\label{+xk2}
Let $k>1$ be an integer  and $z_n=e^{2\pi in/k}.$ Then,
\begin{align*}
&\frac{x}{(x+1)^2(x^k-1)^2}=-\frac{k-1}{4k^2(x-1)}+\frac{1}{4k^2(x-1)^2}
-\frac{(k-1)}{4(x+1)}-\frac{1}{4(x+1)^2}\\
&\qquad -\frac{1}{k^2}\sum_{n=1}^{k-1}\left\{\frac{kz_n^2}{(z_n+1)^2}
-\frac{2z_n^2}{(z_n+1)^3}\right\}\frac{1}{(x-z_n)}
+\frac{1}{k^2}\sum_{n=1}^{k-1}\frac{z_n^3}{(z_n+1)^2(x-z_n)^2}.
\end{align*}
\end{lemma}

\begin{proof}
From \eqref{xk2-1}, we can deduce that
\begin{align*}
\frac{k^2x}{(x+1)^2(x^k-1)^2}&=\frac{x}{(x+1)^2(x-1)^2}+\frac{2x}{(x+1)^2(x-1)}\sum_{j=1}^{k-1}\frac{z_j}{x-z_j}\notag\\
&\quad +\frac{x}{(x+1)^2}
\left\{\sum_{j=1}^{k-1}\frac{z^2_j}{(x-z_j)^2}+\sum_{\substack{i,j=1 \\ i\neq j}}^{k-1}\frac{z_iz_j}{(x-z_i)(x-z_j)}\right\}.
\end{align*}
Employing \eqref{p7}, \eqref{p8} and the following partial fraction decomposition
\begin{align*}
\frac{x}{(x+1)^2(x-a)^2}=&\frac{a-1}{(a+1)^3(x+1)}-\frac{1}{(a+1)^2(x+1)^2}\\
&-\frac{a-1}{(a+1)^3(x-a)}+\frac{a}{(a+1)^2(x-a)^2},
\end{align*}
we derive
\begin{align}\label{+xk2e1}
&\frac{k^2x}{(x+1)^2(x^k-1)^2}=\frac{1}{4(x-1)^2}-\frac{1}{4(x+1)^2}\notag\\
&+\sum_{j=1}^{k-1}\left\{\frac{z_j(z_j-1)}{2(z_j+1)^2(x+1)}-\frac{z_j}{(z_j+1)(x+1)^2}
+\frac{z_j}{2(1-z_j)(x-1)}+\frac{2z_j^2}{(z_j+1)^2(z_j-1)(x-z_j)}\right\}\notag\\
&+\sum_{j=1}^{k-1}\left\{\frac{z_j^2(z_j-1)}{(z_j+1)^3(x+1)}-\frac{z_j^2}{(z_j+1)^2(x+1)^2}
-\frac{z_j^2(z_j-1)}{(z_j+1)^3(x-z_j)}+\frac{z_j^3}{(z_j+1)^2(x-z_j)^2}\right\}\notag\\
&+\sum_{\substack{i,j=1 \\ i\neq j}}^{k-1}\left\{\frac{z_iz_j(z_iz_j-1)}{(z_i+1)^2(z_j+1)^2(x+1)}-\frac{z_iz_j}{(z_i+1)(z_j+1)(x+1)^2}
+\frac{2z_i^2z_j}{(z_i+1)^2(z_i-z_j)(x-z_i)}\right\}.
\end{align}
Now, we write the right-hand side of \eqref{+xk2e1} as
\begin{align}\label{H}
\frac{H_1}{x-1}+\frac{H_2}{(x-1)^2}+\frac{H_3}{x+1}+\frac{H_4}{(x+1)^2}
+\sum_{n=1}^{k-1}\frac{H_{5,n}}{x-z_n}+\sum_{n=1}^{k-1}\frac{H_{6,n}}{(x-z_n)^2}
\end{align}
and calculate each $H_i$ with $1\leq i\leq 4$ and $H_{j, n}$ with  $j=5, 6$ and $1\leq n \leq k-1.$

We can easily see that by Lemma \ref{zz1},
\begin{align*}
H_1=\sum_{j=1}^{k-1}\frac{z_j}{2(1-z_j)}=-\frac{k-1}{4}, \quad \text{and} \quad H_2=\dfrac{1}{4}.
\end{align*}
Next, using Lemma \ref{zz+1}, we have
\begin{align*}
H_3&=\sum_{j=1}^{k-1}\left\{\frac{z_j(z_j-1)}{2(z_j+1)^2}+\frac{z_j^2(z_j-1)}{(z_j+1)^3}\right\}
+\sum_{\substack{i,j=1 \\ i\neq j}}^{k-1}\frac{z_iz_j(z_iz_j-1)}{(z_i+1)^2(z_j+1)^2} \notag\\
&=\sum_{j=1}^{k-1}\left\{\frac{z_j^2-z_j}{2(z_j+1)^2}+\frac{z_j^3-z_j^2}{(z_j+1)^3}\right\}
+\left(\sum_{j=1}^{k-1}\frac{z_j^2}{(z_i+1)^2}\right)^2-\sum_{j=1}^{k-1}\frac{z_j^4}{(z_i+1)^4} \notag\\
&\qquad \qquad-\left(\sum_{j=1}^{k-1}\frac{z_j}{(z_i+1)^2}\right)^2+\sum_{j=1}^{k-1}\frac{z_j^2}{(z_i+1)^4}  \notag\\
&=-\frac{k^2(k-1)}{4},
\end{align*}
and
\begin{align*}
H_4&=-\frac14-\sum_{j=1}^{k-1}\frac{z_j}{z_j+1}-\sum_{j=1}^{k-1}\frac{z_j^2}{(z_j+1)^2}
-\sum_{\substack{i,j=1 \\ i\neq j}}^{k-1}\frac{z_iz_j}{(z_i+1)(z_j+1)} \notag\\
&=-\frac14-\sum_{j=1}^{k-1}\frac{z_j}{z_j+1}-\left(\sum_{j=1}^{k-1}\frac{z_j}{z_j+1}\right)^2=-\frac{k^2}{4}.
\end{align*}
Now, for each $1\leq n\leq k-1,$
\begin{align}\label{H5}
H_{5, n}&=\frac{2z_n^2}{(z_n+1)^2(z_n-1)}-\frac{z_n^2(z_n-1)}{(z_n+1)^3}
+\sum_{\substack{j=1 \\ j\neq n}}^{k-1}\frac{2z_n^2z_j}{(z_n+1)^2(z_n-z_j)} \notag\\
&=\frac{2z_n^2}{(z_n+1)^2(z_n-1)}-\frac{z_n^3-z_n^2}{(z_n+1)^3}
+\frac{2z_n^2}{(z_n+1)^2}\sum_{\substack{j=1 \\ j\neq n}}^{k-1}\frac{z_j}{z_n-z_j}.
\end{align}
Observe that
\begin{align}\label{H5-1}
\sum_{\substack{j=1 \\ j\neq n}}^{k-1}\frac{z_j}{z_n-z_j}&=\sum_{\substack{j=1 \\ j\neq n}}^{k-1}\left(\frac{z_n}{z_n-z_j}-1\right)
=\sum_{\substack{j=1 \\ j\neq n}}^{k}\frac{z_n}{z_n-z_j}-\frac{z_n}{z_n-1}-(k-2)\notag\\
&=\sum_{s=1}^{k-1}\frac{1}{1-z_s}-\frac{z_n}{z_n-1}-(k-2)
=-\frac{(k-3)}{2}-\frac{z_n}{z_n-1},
\end{align}
where we used Lemma \ref{1z-1}. Putting \eqref{H5-1} into \eqref{H5} yields
\begin{align*}
H_{5, n}&=\frac{2z_n^2}{(z_n+1)^2(z_n-1)}-\frac{z_n^3-z_n^2}{(z_n+1)^3}
-\frac{(k-3)z_n^2}{(z_n+1)^2}-\frac{2z_n^3}{(z_n+1)^2(z_n-1)} \notag\\
&=-\frac{(k-1)z_n^2}{(z_n+1)^2}-\frac{z_n^3-z_n^2}{(z_n+1)^3}
=-\frac{kz_n^2}{(z_n+1)^2}+\frac{2z_n^2}{(z_n+1)^3}.
\end{align*}
Clearly, we see that for each $1\leq n\leq k-1,$
\begin{align*}
H_{6, n}&=\frac{z_n^3}{(z_n+1)^2}.
\end{align*}
Putting all these together into \eqref{H}, we complete our proof.
\end{proof}

\begin{lemma}\label{lemR24}
Let $p$ and $q$ be relatively prime odd positive integers  such that $p, q \geq 3,$ and let
$\omega_n=e^{2\pi in/p}$ and $\xi_n=e^{2\pi in/q}.$ Then,
\begin{align*}
&q^2\sum_{j=1}^{p-1}\frac{\o_j}{(\o_j+1)^2(\o_j^q-1)^2}
-p^2\sum_{j=1}^{q-1}\frac{\xi_j}{(\xi_j+1)^2(\xi_j^p-1)^2} \notag\\
&=\frac{1}{48}(3p^2q^2+6p^2q+5p^2-5q^2-6q-3)+2q\sum_{n=1}^{p-1}\frac{\o_n}{(\o_n+1)^3(\o_n^q-1)}.
\end{align*}
\end{lemma}

\begin{proof}
We first apply Lemma \ref{+xk2} to obtain
\begin{align}\label{R24e1}
&p^2\sum_{j=1}^{q-1}\frac{\xi_j}{(\xi_j+1)^2(\xi_j^p-1)^2}
=-\frac{p-1}{4}\sum_{j=1}^{q-1}\frac{1}{(\xi_j-1)}+\frac{1}{4}\sum_{j=1}^{q-1}\frac{1}{(\xi_j-1)^2}
-\frac{p^2(p-1)}{4}\sum_{j=1}^{q-1}\frac{1}{(\xi_j+1)} \notag\\
&\qquad \qquad \qquad\qquad-\frac{p^2}{4}\sum_{j=1}^{q-1}\frac{1}{(\xi_j+1)^2}
-\sum_{n=1}^{p-1}\left\{\frac{p\o_n^2}{(\o_n+1)^2}
-\frac{2\o_n^2}{(\o_n+1)^3}\right\}\sum_{j=1}^{q-1}\frac{1}{(\xi_j-\o_n)}\notag\\
&\qquad\qquad\qquad\qquad+\sum_{n=1}^{p-1}\frac{\o_n^3}{(\o_n+1)^2}\sum_{j=1}^{q-1}\frac{1}{(\xi_j-\o_n)^2}.
\end{align}
Using Lemmas \ref{1z1} and \ref{1z-1}, we can calculate the sum of the first four summations of \eqref{R24e1} as follows:
\begin{align}\label{R24e2}
&\frac{(p-1)(q-1)}{8}-\frac{(q-1)(q-5)}{48}
-\frac{p^2(p-1)(q-1)}{8}+\frac{p^2(q-1)^2}{16} \notag\\
&\quad =-\frac{1}{48}(q-1)(6p^3-3p^2q-3p^2-6p+q+1).
\end{align}

Next, by \eqref{1xz} and Lemmas \ref{1z1}, \ref{1z-1}, \ref{zz+1} and \ref{lemR21},
\begin{align}\label{R24e3}
&\sum_{n=1}^{p-1}\left\{\frac{p\o_n^2}{(\o_n+1)^2}
-\frac{2\o_n^2}{(\o_n+1)^3}\right\}\sum_{j=1}^{q-1}\frac{1}{(\xi_j-\o_i)} \notag\\
=&\sum_{n=1}^{p-1}\left\{\frac{p\o_n^2}{(\o_n+1)^2}
-\frac{2\o_n^2}{(\o_n+1)^3}\right\}
\left\{\frac{1}{\o_n-1}-\frac{q}{\o_n}-\frac{q}{\o_n(\o_n^q-1)}\right\} \notag\\
=&\sum_{n=1}^{p-1}\left\{\frac{p\o_n^2}{(\o_n+1)^2(\o_n-1)}-\frac{2\o_n^2}{(\o_n+1)^3(\o_n-1)}
-\frac{pq\o_n}{(\o_n+1)^2}+\frac{2q\o_n}{(\o_n+1)^3}\right.\notag\\
&\left.\qquad\qquad  -\frac{pq\o_n}{(\o_n+1)^2(\o_n^q-1)}+\frac{2q\o_n}{(\o_n+1)^3(\o_n^q-1)}\right\} \notag\\
=&\frac{p-1}{4}\sum_{n=1}^{p-1}\frac{1}{(\o_n-1)}+\frac{3p+1}{4}\sum_{n=1}^{p-1}\frac{1}{(\o_n+1)}
-\frac{p+3}{2}\sum_{n=1}^{p-1}\frac{1}{(\o_n+1)^2}\notag\\
&-pq\sum_{n=1}^{p-1}\frac{\o_n}{(\o_n+1)^2}+\sum_{n=1}^{p-1}\frac{1}{(\o_n+1)^3}+2q\sum_{n=1}^{p-1}\frac{\o_n}{(\o_n+1)^3}
\notag\\
&-pq\sum_{n=1}^{p-1}\frac{\o_n}{(\o_n+1)^2(\o_n^q-1)} +2q\sum_{n=1}^{p-1}\frac{\o_n}{(\o_n+1)^3(\o_n^q-1)} \notag\\
=&-\frac{(p-1)^2}{8}+\frac{(3p+1)(p-1)}{8}+\frac{(p+3)(p-1)^2}{8}-\frac{pq(p^2-1)}{4}-\frac{(p-1)(3p-1)}{8}\notag\\
&\quad +\frac{q(p^2-1)}{4} +\frac{pq(p^2-1)}{8}+2q\sum_{n=1}^{p-1}\frac{\o_n}{(\o_n+1)^3(\o_n^q-1)}\notag\\
=&\frac{1}{8}(p-1)(-p^2q+p^2+pq+p+2q)+2q\sum_{n=1}^{p-1}\frac{\o_n}{(\o_n+1)^3(\o_n^q-1)}.
\end{align}
Similarly, we use \eqref{1xz2} to obtain
\begin{align}\label{R24e4}
&\sum_{n=1}^{p-1}\frac{\o_n^3}{(\o_n+1)^2}\sum_{j=1}^{q-1}\frac{1}{(\xi_j-\o_n)^2}\notag\\
&=-\sum_{n=1}^{p-1}\frac{\o_n^3}{(\o_n+1)^2(\o_n-1)^2}
-q(q-1)\sum_{n=1}^{p-1}\frac{\o_n^{q+1}}{(\o_n+1)^2(\o_n^q-1)}
+q^2\sum_{n=1}^{p-1}\frac{\o_n^{2q+1}}{(\o_n+1)^2(\o_n^q-1)^2} \notag\\
&=-\sum_{n=1}^{p-1}\left\{\frac{1}{2(\o_n+1)}-\frac{1}{4(\o_n+1)^2}+\frac{1}{2(\o_n-1)}+\frac{1}{4(\o_n-1)^2}\right\} \notag\\
&\qquad -q(q-1)\sum_{n=1}^{p-1}\left\{\frac{\o_n}{(\o_n+1)^2}+\frac{\o_n}{(\o_n+1)^2(\o_n^q-1)}\right\} \notag\\
&\qquad +q^2\sum_{n=1}^{p-1}\left\{\frac{\o_n}{(\o_n+1)^2}+\frac{2\o_n}{(\o_n+1)^2(\o_n^q-1)}
+\frac{\o_n}{(\o_n+1)^2(\o_n^q-1)^2}\right\} \notag\\
&=-\frac{1}{24}(p^2-1)(3q^2-3q+1)+q^2\sum_{n=1}^{p-1}\frac{\o_n}{(\o_n+1)^2(\o_n^q-1)^2}.
\end{align}
where we applied Lemmas \ref{1z1}, \ref{1z-1}, \ref{zz+1} and \ref{lemR21}.

Putting \eqref{R24e2}--\eqref{R24e4} into \eqref{R24e1}, we arrive at
\begin{align*}
&p^2\sum_{j=1}^{q-1}\frac{\xi_j}{(\xi_j+1)^2(\xi_j^p-1)^2}=-\frac{1}{48}(q-1)(6p^3-3p^2q-3p^2-6p+q+1) \notag\\
&\qquad -\frac{1}{8}(p-1)(-p^2q+p^2+pq+p+2q)-2q\sum_{n=1}^{p-1}\frac{\o_n}{(\o_n+1)^3(\o_n^q-1)} \notag\\
&\qquad -\frac{1}{24}(p^2-1)(3q^2-3q+1)+q^2\sum_{n=1}^{p-1}\frac{\o_n}{(\o_n+1)^2(\o_n^q-1)^2} \notag\\
&=-\frac{1}{48}(3p^2q^2+6p^2q+5p^2-5q^2-6q-3)-2q\sum_{n=1}^{p-1}\frac{\o_n}{(\o_n+1)^3(\o_n^q-1)}  \notag\\
&\qquad +q^2\sum_{n=1}^{p-1}\frac{\o_n}{(\o_n+1)^2(\o_n^q-1)^2}.
\end{align*}
This completes the proof of the lemma.
\end{proof}

\section{A Second Trigonometric Sum with Two Independent Periods}

\begin{theorem}\label{cccpq}
Let $p$ and $q$ be odd positive integers  with $p, q \geq 3$ and $(p, q)=1,$ and let $\o_j=e^{2\pi ij/p}.$ Then,
\begin{align*}
&\sum_{\substack{n=1\\ p \nmid n, \, q \nmid n}}^{pq-1}
\dfrac{\cot\Big(\dfrac{\pi n}{p}\Big)\cot\Big(\dfrac{\pi n}{q}\Big)}{\cos^2\Big(\dfrac{\pi n}{pq}\Big)}
=\frac{2}{3p}(p^2-1)(5q+3)\\
&\qquad\qquad +\frac{32}{p}\sum_{j=1}^{p-1}\frac{\o_j}{(\o_j+1)^3(\o_j^q-1)}
-\frac{32q}{p}\sum_{j=1}^{p-1}\frac{\o_j}{(\o_j+1)^2(\o_j^q-1)^2}.
\end{align*}
\end{theorem}

\begin{proof}
We set
\begin{align*}
z_n=e^{2\pi i n/(pq)}, \quad f(x)=\dfrac{(x-1)(x^{pq}-1)}{(x^p-1)(x^q-1)}
=\prod_{\substack{n=1\\ p \nmid n, \, q \nmid n}}^{pq-1}(x-z_n),
\end{align*}
and
\begin{align*}
R(x)=\dfrac{x(x^p+1)(x^q+1)}{(x+1)^2(x^p-1)(x^q-1)}.
\end{align*}
Then, by the identities in \eqref{s2cc} we see that
\begin{align}\label{R2z}
\sum_{\substack{n=1\\ p \nmid n, \, q \nmid n}}^{pq-1}
\dfrac{\cot\Big(\dfrac{\pi n}{p}\Big)\cot\Big(\dfrac{\pi n}{q}\Big)}{\cos^2\Big(\dfrac{\pi n}{pq}\Big)}
=-4\sum_{\substack{n=1\\ p \nmid n, \, q \nmid n}}^{pq-1}R(z_n).
\end{align}
Now, by Lemma \ref{R2},
\begin{align}\label{R2z1}
\sum_{\substack{n=1\\ p \nmid n, \, q \nmid n}}^{pq-1}R(z_n)
&=\frac{1}{pq}\sum_{\substack{n=1\\ p \nmid n, \, q \nmid n}}^{pq-1}\left\{\frac{1}{(z_n-1)}+\frac{1}{(z_n-1)^2} \right\}
+\frac{2}{p}\sum_{\substack{n=1\\ p \nmid n, \, q \nmid n}}^{pq-1}\sum_{j=1}^{p-1}\frac{\omega_j^2(\omega_j^q+1)}{(\omega_j+1)^2(\omega_j^q-1)(z_n-\o_j)} \notag\\
&\qquad+\frac{2}{q}\sum_{\substack{n=1\\ p \nmid n, \, q \nmid n}}^{pq-1}\sum_{j=1}^{q-1}\frac{\xi_j^2(\xi_j^p+1)}{(\xi_j+1)^2(\xi_j^p-1)(z_n-\xi_j)},
\end{align}
where $\xi_j=e^{2\pi ij/q}.$
Applying the first two identities in \eqref{pqz1} yields
\begin{align}\label{R2e1}
\frac{1}{pq}\sum_{\substack{n=1\\ p \nmid n, \, q \nmid n}}^{pq-1}\left\{\frac{1}{(z_n-1)}+\frac{1}{(z_n-1)^2} \right\}=
-\frac{(p^2-1)(q^2-1)}{12pq}.
\end{align}
Also, by \eqref{pqo} we obtain
\begin{align}\label{R2e2}
&\frac{2}{p}\sum_{\substack{n=1\\ p \nmid n, \, q \nmid n}}^{pq-1}\sum_{j=1}^{p-1}\frac{\omega_j^2(\omega_j^q+1)}{(\omega_j+1)^2(\omega_j^q-1)(z_n-\o_j)} \notag\\
&=\frac{2}{p}\sum_{j=1}^{p-1}\left\{\frac{\o_j^2}{(\o_j+1)^2}+\frac{2\omega_j^2}{(\o_j+1)^2(\omega^q_j-1)}\right\}
\left\{-\frac{p(q-1)}{2\o_j}-\frac{1}{\o_j-1}+\frac{q\o_j^{q-1}}{\o_j^q-1}\right\}\notag\\
&=-(q-1)\sum_{j=1}^{p-1}\frac{\o_j}{(\o_j+1)^2}-\frac{2}{p}\sum_{j=1}^{p-1}\frac{\o_j^2}{(\o_j+1)^2(\o_j-1)}
+\frac{2q}{p}\sum_{j=1}^{p-1}\frac{\o_j^{q+1}}{(\o_j+1)^2(\o_j^q-1)}\notag\\
&\quad -2(q-1)\sum_{j=1}^{p-1}\frac{\o_j}{(\o_j+1)^2(\o_j^q-1)}
-\frac{4}{p}\sum_{j=1}^{p-1}\frac{\o_j^2}{(\o_j+1)^2(\o_j-1)(\o_j^q-1)}\notag\\
&\quad  +\frac{4q}{p}\sum_{j=1}^{p-1}\frac{\o_j^{q+1}}{(\o_j+1)^2(\o_j^q-1)^2} \notag\\
&=-(q-1)\sum_{j=1}^{p-1}\frac{\o_j}{(\o_j+1)^2}
-\frac{2}{p}\sum_{j=1}^{p-1}\left\{\frac{3\o_j+1}{4(\o_j+1)^2}+\frac{1}{4(\o_j-1)}\right\} \notag\\
&\quad +\frac{2q}{p}\sum_{j=1}^{p-1}\left\{\frac{\o_j}{(\o_j+1)^2}+\frac{\o_j}{(\o_j+1)^2(\o_j^q-1)}\right\}
-2(q-1)\sum_{j=1}^{p-1}\frac{\o_j}{(\o_j+1)^2(\o_j^q-1)}\notag\\
&\quad -\frac{4}{p}\sum_{j=1}^{p-1}\left\{\frac{3\o_j}{4(\o_j+1)^2(\o_j^q-1)}+\frac{1}{4(\o_j+1)^2(\o_j^q-1)}
+\frac{1}{4(\o_j-1)(\o_j^q-1)}\right\}\notag\\
&\quad +\frac{4q}{p}\sum_{j=1}^{p-1}\left\{\frac{\o_j}{(\o_j+1)^2(\o_j^q-1)}+\frac{\o_j}{(\o_j+1)^2(\o_j^q-1)^2}\right\}\notag\\
&=-\frac{1}{2p}(2pq-2p-4q+3)\sum_{j=1}^{p-1}\frac{\o_j}{(\o_j+1)^2}-\frac{1}{2p}\sum_{j=1}^{p-1}\frac{1}{(\o_j+1)^2}
-\frac{1}{2p}\sum_{j=1}^{p-1}\frac{1}{(\o_j-1)} \notag\\
&\quad-\frac{1}{p}(2pq-2p-6q+3)\sum_{j=1}^{p-1}\frac{\o_j}{(\o_j+1)^2(\o_j^q-1)}
-\frac{1}{p}\sum_{j=1}^{p-1}\frac{1}{(\o_j+1)^2(\o_j^q-1)} \notag\\
&\quad-\frac{1}{p}\sum_{j=1}^{p-1}\frac{1}{(\o_j-1)(\o_j^q-1)}+\frac{4q}{p}\sum_{j=1}^{p-1}\frac{\o_j}{(\o_j+1)^2(\o_j^q-1)^2}\notag\\
&=-\frac{1}{8p}(2pq-2p-4q+3)(p^2-1)+\frac{1}{8p}(p-1)^2+\frac{1}{4p}(p-1) \notag\\
&\quad +\frac{1}{8p}(2pq-2p-6q+3)(p^2-1)-\frac{1}{p}\sum_{j=1}^{p-1}\frac{1}{(\o_j+1)^2(\o_j^q-1)} \notag\\
&\quad-\frac{1}{p}\sum_{j=1}^{p-1}\frac{1}{(\o_j-1)(\o_j^q-1)}+\frac{4q}{p}\sum_{j=1}^{p-1}\frac{\o_j}{(\o_j+1)^2(\o_j^q-1)^2}\notag\\
&=-\frac{1}{8p}(p^2-1)(2q-1)-\frac{1}{p}\sum_{j=1}^{p-1}\frac{1}{(\o_j+1)^2(\o_j^q-1)}
-\frac{1}{p}\sum_{j=1}^{p-1}\frac{1}{(\o_j-1)(\o_j^q-1)} \notag\\
&\quad+\frac{4q}{p}\sum_{j=1}^{p-1}\frac{\o_j}{(\o_j+1)^2(\o_j^q-1)^2},
\end{align}
where we employed Lemmas \ref{1z1} and \ref{1z-1}, \ref{zz+1} and \ref{lemR21} for the penultimate equality.

Analogously, we can deduce that
\begin{align}\label{R2e3}
&\frac{2}{q}\sum_{\substack{n=1\\ p \nmid n, \, q \nmid n}}^{pq-1}\sum_{j=1}^{q-1}\frac{\xi_j^2(\xi_j^p+1)}{(\xi_j+1)^2(\xi_j^p-1)(z_n-\xi_j)}=-\frac{1}{8q}(q^2-1)(2p-1)\notag \\
&\quad -\frac{1}{q}\sum_{j=1}^{q-1}\frac{1}{(\xi_j+1)^2(\xi_j^p-1)}-\frac{1}{q}
\sum_{j=1}^{q-1}\frac{1}{(\xi_j-1)(\xi_j^p-1)}+\frac{4p}{q}\sum_{j=1}^{q-1}\frac{\xi_j}{(\xi_j+1)^2(\xi_j^p-1)^2}.
\end{align}
From \eqref{R2z1}--\eqref{R2e3}, with the aid of Lemmas \ref{lemR22}, \ref{lemR23}, and \ref{lemR24}, it follows that
\begin{align*}
&\sum_{\substack{n=1\\ p \nmid n, \, q \nmid n}}^{pq-1}R(z_n)=
\frac{1}{24pq}(8p^2+8q^2-14p^2q^2+3p^2q+3pq^2-3p-3q-2) \notag \\
&\quad -\frac{1}{pq}\left\{q\sum_{j=1}^{p-1}\frac{1}{(\o_j-1)(\o_j^q-1)}
+p\sum_{j=1}^{q-1}\frac{1}{(\xi_j-1)(\xi_j^p-1)}\right\}\notag\\
&\quad-\frac{1}{pq}\left\{q\sum_{j=1}^{p-1}\frac{1}{(\o_j+1)^2(\o_j^q-1)}
+p\sum_{j=1}^{q-1}\frac{1}{(\xi_j+1)^2(\xi_j^p-1)}\right\}
\notag\\
&\quad+\frac{4}{pq}\left\{q^2\sum_{j=1}^{p-1}\frac{\o_j}{(\o_j+1)^2(\o_j^q-1)^2}
+p^2\sum_{j=1}^{q-1}\frac{\xi_j}{(\xi_j+1)^2(\xi_j^p-1)^2}\right\} \notag\\
&=\frac{1}{24pq}(8p^2+8q^2-14p^2q^2+3p^2q+3pq^2-3p-3q-2) \notag \\
&\quad+\frac{1}{12pq}(p^2+q^2-9pq+3p+3q+1)-\frac{1}{8pq}(p^2q+pq^2+p+q-6pq+2) \notag\\
&\quad -\frac{1}{12pq}(3p^2q^2+6p^2q+5p^2-5q^2-6q-3) \notag\\
&\quad +\frac{8q}{p}\sum_{j=1}^{p-1}\frac{\o_j}{(\o_j+1)^2(\o_j^q-1)^2}
-\frac{8}{p}\sum_{j=1}^{p-1}\frac{\o_j}{(\o_j+1)^3(\o_j^q-1)} \notag\\
&=-\frac{1}{6p}(p^2-1)(5q+3)+\frac{8q}{p}\sum_{j=1}^{p-1}\frac{\o_j}{(\o_j+1)^2(\o_j^q-1)^2}
-\frac{8}{p}\sum_{j=1}^{p-1}\frac{\o_j}{(\o_j+1)^3(\o_j^q-1)}.
\end{align*}
By \eqref{R2z}, we complete our proof.
\end{proof}

\begin{corollary}
Let $p$ and $q$ be relatively prime odd positive integers such that $p, q \geq 3$ and $q\equiv \pm1 \pmod{p}.$
Then,
\begin{align*}
&\sum_{\substack{n=1\\ p \nmid n, \, q \nmid n}}^{pq-1}
\dfrac{\cot\Big(\dfrac{\pi n}{p}\Big)\cot\Big(\dfrac{\pi n}{q}\Big)}{\cos^2\Big(\dfrac{\pi n}{pq}\Big)}
=\begin{cases}
\dfrac{2}{p}(p^2-1)(q-1), & \text{if} \quad  q\equiv 1 \pmod{p}, \\[0.2in]
\dfrac{2}{p}(p^2-1)(q+1), & \text{if} \quad  q\equiv -1 \pmod{p}.
\end{cases}
\end{align*}
\end{corollary}

\begin{proof}
Suppose that $q\equiv 1 \pmod{p}.$ Then, Theorem \ref{cccpq} implies that
\begin{align*}
&\sum_{\substack{n=1\\ p \nmid n, \, q \nmid n}}^{pq-1}
\dfrac{\cot\Big(\dfrac{\pi n}{p}\Big)\cot\Big(\dfrac{\pi n}{q}\Big)}{\cos^2\Big(\dfrac{\pi n}{pq}\Big)}
=\frac{2}{3p}(p^2-1)(5q+3)\\
&\qquad\qquad +\frac{32}{p}\sum_{j=1}^{p-1}\frac{\o_j}{(\o_j+1)^3(\o_j-1)}
-\frac{32q}{p}\sum_{j=1}^{p-1}\frac{\o_j}{(\o_j+1)^2(\o_j-1)^2} \notag\\
&=\frac{2}{3p}(p^2-1)(5q+3)+\frac{4}{p}\sum_{j=1}^{p-1}\left\{-\frac{1}{\o_j+1} -\frac{2}{(\o_j+1)^2}
+\frac{4}{(\o_j+1)^3}+\frac{1}{\o_j-1}\right\} \notag\\
&\qquad \qquad -\frac{8q}{p}\sum_{j=1}^{p-1}\left\{\frac{1}{(\o_j-1)^2}-\frac{1}{(\o_j+1)^2}\right\} \notag\\
&=\frac{2}{3p}(p^2-1)(5q+3)-\frac{4(p^2-1)}{p}-\frac{4q(p^2-1)}{3p}
=\frac{2}{p}(p^2-1)(q-1),
\end{align*}
where we applied Lemmas \ref{1z1} and \ref{1z-1}.

Next, if $q\equiv -1 \pmod{p},$ with the use of  Lemmas \ref{1z1} and \ref{1z-1},  we derive
\begin{align*}
&\sum_{\substack{n=1\\ p \nmid n, \, q \nmid n}}^{pq-1}
\dfrac{\cot\Big(\dfrac{\pi n}{p}\Big)\cot\Big(\dfrac{\pi n}{q}\Big)}{\cos^2\Big(\dfrac{\pi n}{pq}\Big)}
=\frac{2}{3p}(p^2-1)(5q+3)\\
&\qquad\qquad -\frac{32}{p}\sum_{j=1}^{p-1}\frac{\o_j^2}{(\o_j+1)^3(\o_j-1)}
-\frac{32q}{p}\sum_{j=1}^{p-1}\frac{\o_j^3}{(\o_j+1)^2(\o_j-1)^2} \notag\\
&=\frac{2}{3p}(p^2-1)(5q+3)+\frac{4}{p}\sum_{j=1}^{p-1}\left\{\frac{1}{\o_j+1} -\frac{6}{(\o_j+1)^2}
+\frac{4}{(\o_j+1)^3}-\frac{1}{\o_j-1}\right\} \notag\\
&\qquad \qquad-\frac{8q}{p}\sum_{j=1}^{p-1}\left\{\frac{2}{\o_j+1} -\frac{1}{(\o_j+1)^2}
+\frac{2}{\o_j-1}+\frac{1}{(\o_j-1)^2}\right\} \notag\\
&=\frac{2}{3p}(p^2-1)(5q+3)-\frac{4}{3p}q(p^2-1)=\frac{2}{p}(p^2-1)(q+1).
\end{align*}
Therefore, we complete the proof.
\end{proof}


\section{Reciprocity Theorems}

\begin{theorem}\label{cpc}
Let $p$ and $q$ be relatively prime positive integers  with $p, q \geq 2,$ and let $\xi_j=e^{2\pi ij/q}.$ Then,
\begin{align*}
\sum_{n=1}^{q-1}\cot\Big(\frac{\pi np}{q}\Big)\cot\Big(\frac{\pi n}{q}\Big)=q-1-4\sum_{n=1}^{q-1}\frac{1}{(\xi_n-1)(\xi_n^p-1)}
\end{align*}
\end{theorem}

\begin{proof} Employing the identities
\begin{align}\label{cot}
\cot\Big(\dfrac{\pi n}{q}\Big)=i\dfrac{\xi_n+1}{\xi_n-1},
\quad \cot\Big(\dfrac{\pi np}{q}\Big)=i\dfrac{\xi_n^p+1}{\xi_n^p-1}
\end{align}
and Lemma  \ref{1z-1}, we deduce that
\begin{align*}
&\sum_{n=1}^{q-1}\cot\Big(\frac{\pi np}{q}\Big)\cot\Big(\frac{\pi n}{q}\Big)
=-\sum_{n=1}^{q-1}\frac{\xi_n^p+1}{\xi_n^p-1}\frac{\xi_n+1}{\xi_n-1}
=-\sum_{n=1}^{q-1}\left\{1+\frac{2}{\xi_n^p-1}\right\}\left\{1+\frac{2}{\xi_n-1}\right\} \notag\\
&=-\sum_{n=1}^{q-1}\left\{1+\frac{2}{\xi_n-1}+\frac{2}{\xi_n^p-1}+\frac{4}{(\xi_n-1)(\xi_n^p-1)}\right\} \notag\\
&=-(q-1)+2\frac{(q-1)}{2}+2\frac{(q-1)}{2}-4\sum_{n=1}^{q-1}\frac{1}{(\xi_n-1)(\xi_n^p-1)}\notag\\
&=q-1-4\sum_{n=1}^{q-1}\frac{1}{(\xi_n-1)(\xi_n^p-1)},
\end{align*}
where we used  the assumption $(p, q)=1$ for the penultimate equality.
\end{proof}

\begin{corollary}
Let $p$ and $q$ be relatively prime positive integers  with $p, q \geq 2.$ Then,
\begin{align}\label{dedekind}
p\sum_{n=1}^{q-1}\cot\Big(\frac{\pi np}{q}\Big)\cot\Big(\frac{\pi n}{q}\Big)
+q\sum_{n=1}^{p-1}\cot\Big(\frac{\pi nq}{p}\Big)\cot\Big(\frac{\pi n}{p}\Big)=\frac{1}{3}(p^2+q^2-3pq+1).
\end{align}
\end{corollary}

\begin{proof}
By Theorem \ref{cpc} and Lemma \ref{lemR22}, we have
\begin{align*}
&p\sum_{n=1}^{q-1}\cot\Big(\frac{\pi np}{q}\Big)\cot\Big(\frac{\pi n}{q}\Big)
+q\sum_{n=1}^{p-1}\cot\Big(\frac{\pi nq}{p}\Big)\cot\Big(\frac{\pi n}{p}\Big) \notag\\
&=p(q-1)-4p\sum_{n=1}^{q-1}\frac{1}{(\xi_n-1)(\xi_n^p-1)}
+q(p-1)-4q\sum_{j=1}^{p-1}\frac{1}{(\o_j-1)(\o_j^q-1)}\notag\\
&=2pq-p-q+\frac{1}{3}(p^2+q^2-9pq+3p+3q+1)\notag\\
&=\frac{1}{3}(p^2+q^2-3pq+1).
\end{align*}
\end{proof}

The reciprocity relation \eqref{dedekind} is equivalent to the reciprocity theorem for Dedekind sums. Let
\begin{equation*}
((x))=\begin{cases} x-[x]-\frac12,\quad &\text{if $x$ is not an integer}, \\
0, &\text{if $x$ is  an integer}.
\end{cases}
\end{equation*}
Then, for integers $p,q$ with $(p,q)=1$,  the Dedekind sum $s(p,q)$ is defined by
\begin{equation*}
s(p,q):=\sum_{n=1}^q\left(\left(\df{pn}{q}\right)\right)\left(\left(\df{n}{q}\right)\right).
\end{equation*}
The Dedekind sum $s(p,q)$ satisfies a famous reciprocity theorem \cite[p.~4]{rademachergrosswald}
\begin{equation}\label{dedekind2}
s(p,q)+s(q,p)=-\df14+\df{1}{12}\left(\df{p}{q}+\df{1}{pq}+\df{q}{p}\right).
\end{equation}
H.~Rademacher \cite{rademacher}, \cite[p.~18]{rademachergrosswald} proved that
\begin{equation}\label{dedekind3}
s(p,q)=\df{1}{4q}\sum_{n=1}^{q-1}\cot\left(\df{\pi pn}{q}\right)\cot\left(\df{\pi n}{q}\right),
\end{equation}
which he used to give a proof of \eqref{dedekind2}.  Thus, from \eqref{dedekind3}, we see that the reciprocity theorem \eqref{dedekind} is equivalent to the reciprocity theorem for Dedekind sums \eqref{dedekind2}.  The beautiful book \cite{rademachergrosswald} by Rademacher and Grosswald gives perhaps a more accessible account of \eqref{dedekind2}, \eqref{dedekind3}, and, more generally, of Dedekind sums. Other proofs of \eqref{dedekind} have been devised by Berndt and Yeap \cite[pp.~370, 371]{yeap} and Fukuhara \cite{fukuhara}.  Further references can be found in both aforementioned papers.

\begin{theorem}\label{cpc3}
Let $p$ and $q$ be relatively prime positive integers  with $p, q \geq 2,$ and let $\xi_j=e^{2\pi ij/q}.$ Then,
\begin{align*}
\sum_{n=1}^{q-1}\cot\Big(\frac{\pi np}{q}\Big)\cot^3\Big(\frac{\pi n}{q}\Big)
=\frac{(q-1)^2}{2}+\sum_{n=1}^{q-1}\frac{12}{(\xi_n-1)^2(\xi_n^p-1)}
+\sum_{n=1}^{q-1}\frac{16}{(\xi_n-1)^3(\xi_n^p-1)}.
\end{align*}

\begin{proof}
Applying \eqref{cot} and Lemmas \ref{1z-1} and \ref{lemr1}, we derive
\begin{align*}
&\sum_{n=1}^{q-1}\cot\Big(\frac{\pi np}{q}\Big)\cot^3\Big(\frac{\pi n}{q}\Big)
=\sum_{n=1}^{q-1}\frac{\xi_n^p+1}{\xi_n^p-1}\left(\frac{\xi_n+1}{\xi_n-1}\right)^3\\
&=\sum_{n=1}^{q-1}\left(1+\frac{2}{\xi_n^p-1}\right)\left(1+\frac{2}{\xi_n-1}\right)^3\\
&=\sum_{n=1}^{q-1}\left\{1+\frac{6}{\xi_n-1}+\frac{12}{(\xi_n-1)^2} +\frac{8}{(\xi_n-1)^3}+\frac{2}{\xi_n^p-1} \right. \\
&\qquad \qquad \left. +\frac{12}{(\xi_n-1)(\xi_n^p-1)}+\frac{24}{(\xi_n-1)^2(\xi_n^p-1)}+\frac{16}{(\xi_n-1)^3(\xi_n^p-1)}  \right\}\\
&=-(q-1)+\sum_{n=1}^{q-1}\left\{\frac{12\xi_n}{(\xi_n-1)^2(\xi_n^p-1)}
+\frac{12}{(\xi_n-1)^2(\xi_n^p-1)}+\frac{16}{(\xi_n-1)^3(\xi_n^p-1)}\right\}\\
&=\frac{(q-1)^2}{2}+\sum_{n=1}^{q-1}\frac{12}{(\xi_n-1)^2(\xi_n^p-1)}
+\sum_{n=1}^{q-1}\frac{16}{(\xi_n-1)^3(\xi_n^p-1)}.
\end{align*}
This completes our proof.
\end{proof}
\end{theorem}

\begin{lemma}\label{lemr1-1}
Let $p$ and $q$ be relatively prime positive integers  with $p, q \geq 2.$ Then,
\begin{align*}
&p\sum_{n=1}^{q-1}\frac{1}{(\xi_n-1)^2(\xi_n^p-1)}+q\sum_{n=1}^{p-1}\frac{1}{(\o_n-1)^2(\o_n^q-1)} \notag\\
&\qquad =\frac{1}{24}(p^2q+pq^2+2p^2+2q^2-18pq+5p+5q+2).
\end{align*}
\end{lemma}

\begin{proof}
From Lemmas \ref{lemr1} and \ref{lemR22}, we can easily obtain
\begin{align*}
&p\sum_{n=1}^{q-1}\frac{1}{(\xi_n-1)^2(\xi_n^p-1)}+q\sum_{n=1}^{p-1}\frac{1}{(\o_n-1)^2(\o_n^q-1)} \notag\\
&=p\sum_{n=1}^{q-1}\left\{\frac{\xi_n}{(\xi_n-1)^2(\xi_n^p-1)}-\frac{1}{(\xi_n-1)(\xi_n^p-1)}\right\} \notag\\
&\qquad +q\sum_{n=1}^{p-1}\left\{\frac{\o_n}{(\o_n-1)^2(\o_n^q-1)}-\frac{1}{(\o_n-1)(\o_n^q-1)} \right\} \notag\\
&=\frac{1}{24}(p^2q+pq^2-p-q)+\frac{1}{12}(p^2+q^2-9pq+3p+3q+1)\notag\\
&=\frac{1}{24}(p^2q+pq^2+2p^2+2q^2-18pq+5p+5q+2).
\end{align*}
\end{proof}

\begin{lemma}\label{lemr2-1}
If $p$ and $q$ are relatively prime positive integers  with $p, q \geq 2,$ then
\begin{align*}
&p\sum_{n=1}^{q-1}\frac{1}{(\xi_n-1)^3(\xi_n^p-1)}+q\sum_{n=1}^{p-1}\frac{1}{(\o_n-1)^3(\o_n^q-1)}\notag\\
&\quad =\frac{1}{720}(p^4+q^4-5p^2q^2-45p^2q-45pq^2-60p^2-60q^2+540pq-135p-135q-57).
\end{align*}
\end{lemma}

\begin{proof}
Similarly, from Lemmas \ref{lemr2} and \ref{lemr1-1}, it follows that
\begin{align*}
&p\sum_{n=1}^{q-1}\frac{1}{(\xi_n-1)^3(\xi_n^p-1)}+q\sum_{n=1}^{p-1}\frac{1}{(\o_n-1)^3(\o_n^q-1)}\notag\\
&=p\sum_{n=1}^{q-1}\left\{\frac{\xi_n}{(\xi_n-1)^3(\xi_n^p-1)}-\frac{1}{(\xi_n-1)^2(\xi_n^p-1)}\right\}\notag\\
&\qquad +q\sum_{n=1}^{p-1}\left\{\frac{\o_n}{(\o_n-1)^3(\o_n^q-1)}-\frac{1}{(\o_n-1)^2(\o_n^q-1)} \right\}\notag\\
&=\frac{1}{720}(p^4+q^4-5p^2q^2-15p^2q-15pq^2+15p+15q+3)\notag\\
&\qquad -\frac{1}{24}(p^2q+pq^2+2p^2+2q^2-18pq+5p+5q+2) \notag\\
&=\frac{1}{720}(p^4+q^4-5p^2q^2-45p^2q-45pq^2-60p^2-60q^2 \notag\\
&\qquad \qquad \qquad +540pq-135p-135q-57).
\end{align*}
\end{proof}

\begin{corollary}
If $p$ and $q$ are relatively prime positive integers  with $p, q \geq 2,$ then
\begin{align*}
&p\sum_{n=1}^{q-1}\cot\Big(\frac{\pi np}{q}\Big)\cot^3\Big(\frac{\pi n}{q}\Big)
+q\sum_{n=1}^{p-1}\cot\Big(\frac{\pi nq}{p}\Big)\cot^3\Big(\frac{\pi n}{p}\Big) \notag\\
&\qquad =\frac{1}{45}(p^4+q^4-5p^2q^2-15p^2-15q^2+45pq-12).
\end{align*}
\end{corollary}

\begin{proof} By Theorem \ref{cpc3}, we have
\begin{align}\label{cpc3e}
&p\sum_{n=1}^{q-1}\cot\Big(\frac{\pi np}{q}\Big)\cot^3\Big(\frac{\pi n}{q}\Big)
+q\sum_{n=1}^{p-1}\cot\Big(\frac{\pi nq}{p}\Big)\cot^3\Big(\frac{\pi n}{p}\Big) \notag\\
&=\frac{p(q-1)^2}{2}+12p\sum_{n=1}^{q-1}\frac{1}{(\xi_n-1)^2(\xi_n^p-1)}
+16p\sum_{n=1}^{q-1}\frac{1}{(\xi_n-1)^3(\xi_n^p-1)}\notag\\
&\qquad +\frac{q(p-1)^2}{2}+12q\sum_{n=1}^{p-1}\frac{1}{(\o_n-1)^2(\o_n^q-1)}
+16q\sum_{n=1}^{p-1}\frac{1}{(\o_n-1)^3(\o_n^q-1)}.
\end{align}
Employing Lemmas \ref{lemr1-1} and \ref{lemr2-1} in \eqref{cpc3e} yields
\begin{align*}
&p\sum_{n=1}^{q-1}\cot\Big(\frac{\pi np}{q}\Big)\cot^3\Big(\frac{\pi n}{q}\Big)
+q\sum_{n=1}^{p-1}\cot\Big(\frac{\pi nq}{p}\Big)\cot^3\Big(\frac{\pi n}{p}\Big) \notag\\
=&\frac{p(q-1)^2}{2}+\frac{q(p-1)^2}{2}+\frac{1}{2}(p^2q+pq^2+2p^2+2q^2-18pq+5p+5q+2) \notag\\
&+\frac{1}{45}(p^4+q^4-5p^2q^2-45p^2q-45pq^2-60p^2-60q^2+540pq-135p-135q-57) \notag\\
=&\frac{1}{45}(p^4+q^4-5p^2q^2-15p^2-15q^2+45pq-12).
\end{align*}
\end{proof}

\begin{theorem}\label{S3}
Let $p$ and $q$ be positive integers  with $q \geq 2$ and $(p, q)=1,$ and let $\xi_n=e^{2\pi in/q}.$ Then,
\begin{align*}
T(p, q):=\sum_{n=1}^{q-1}
\frac{\cot\Big(\dfrac{\pi np}{q}\Big)\cot\Big(\dfrac{\pi n}{q}\Big)}{\sin^2\Big(\dfrac{\pi n}{q}\Big)}
=\frac{(q^2-1)}{3}+16\sum_{n=1}^{q-1}\frac{\xi_n}{(\xi_n-1)^3(\xi_n^p-1)}.
\end{align*}
\end{theorem}

\begin{proof}
We apply the identities \eqref{cot} and
\begin{align*}
\sin^2\Big(\dfrac{\pi n}{q}\Big)=-\dfrac{(\xi_n-1)^2}{4\xi_n}
\end{align*}
to obtain
\begin{align*}
T(p, q)&=4\sum_{n=1}^{q-1}\frac{\xi_n(\xi_n+1)(\xi_n^p+1)}{(\xi_n-1)^3(\xi_n^p-1)}
=4\sum_{n=1}^{q-1}\left\{\frac{\xi_n(\xi_n+1)}{(\xi_n-1)^3}+\frac{2\xi_n(\xi_n+1)}{(\xi_n-1)^3(\xi_n^p-1)}\right\}\\
&=4\sum_{n=1}^{q-1}\left\{\frac{\xi_n}{(\xi_n-1)^2}+\frac{2\xi_n}{(\xi_n-1)^3}
+\frac{2\xi_n}{(\xi_n-1)^2(\xi_n^p-1)}+\frac{4\xi_n}{(\xi_n-1)^3(\xi_n^p-1)}\right\}.
\end{align*}
Invoking Lemmas \ref{zz1} and \ref{lemr1}, we have
\begin{align*}
T(p, q)&=4
\left\{\frac{(q^2-1)}{12}-\frac{2(q^2-1)}{24}
+\frac{(q^2-1)}{12}\right\}+16\sum_{n=1}^{q-1}\frac{\xi_n}{(\xi_n-1)^3(\xi_n^p-1)} \\
&=\df{(q^2-1)}{3}+16\sum_{n=1}^{q-1}\frac{\xi_n}{(\xi_n-1)^3(\xi_n^p-1)}.
\end{align*}
Hence, we complete the proof.
\end{proof}

Observe from Corollary \ref{mc1} and \eqref{mc2} that
\begin{equation*}
T(p,q)=-\df{8k^3}{3}S_3(1,k).
\end{equation*}
A proof of Corollary \ref{mc3} below was given by McIntosh \cite[Theorem 3, p.~199]{mcintosh}.

\begin{corollary}\label{mc3}
Let $p$ and $q$ be positive integers  with $p, q \geq 2$ and $(p, q)=1.$ Then,
\begin{align*}
45pT(p, q)+45qT(q,p)=p^4+q^4-5p^2q^2+3.
\end{align*}
\end{corollary}

\begin{proof} Let  $\o_n=e^{2\pi in/p}$ and $\xi_n
=e^{2\pi in/q}.$
By Theorem \ref{S3}, we have
\begin{align*}
&45pT(p, q)+45qT(q,p)=p^4+q^4-5p^2q^2+3\\
&=15p(q^2-1)+720p\sum_{n=1}^{q-1}\frac{\xi_n}{(\xi_n-1)^3(\xi_n^p-1)}+15q(p^2-1)\\
&\qquad  +720q\sum_{n=1}^{p-1}\frac{\o_n}{(\o_n-1)^3(\o_n
^q-1)}\\
&=15(p^2q+pq^2-p-q)+(p^4+q^4-5p^2q^2-15p^2q-15pq^2+15p+15q+3)\\
&=p^4+q^4-5p^2q^2+3,
\end{align*}
where we employed Lemma  \ref{lemr2}.
\end{proof}

\begin{theorem}\label{cpcc2}
Let $p$ and $q$ be relatively prime positive integers  such that $q \geq 3$ is odd, and let $\xi_j=e^{2\pi ij/q}.$ Then,
\begin{align*}
\sum_{n=1}^{q-1}
\frac{\cot\Big(\dfrac{\pi np}{q}\Big)\cot\Big(\dfrac{\pi n}{q}\Big)}{\cos^2\Big(\dfrac{\pi n}{q}\Big)}
=-8\sum_{n=1}^{q-1}\frac{\xi_n}{(\xi_n^2-1)(\xi_n^p-1)}.
\end{align*}
\end{theorem}

\begin{proof}
From \eqref{cot} and the identity
\begin{align}\label{c2}
\cos^2\Big(\dfrac{\pi n}{q}\Big)=\dfrac{(\xi_n+1)^2}{4\xi_n},
\end{align}
it follows that
\begin{align*}
\sum_{n=1}^{q-1}
&\frac{\cot\Big(\dfrac{\pi np}{q}\Big)\cot\Big(\dfrac{\pi n}{q}\Big)}{\cos^2\Big(\dfrac{\pi n}{q}\Big)}
=-4\sum_{n=1}^{q-1}\frac{\xi_n(\xi_n^p+1)}{(\xi_n-1)(\xi_n+1)(\xi_n^p-1)}\notag\\
&=-4\sum_{n=1}^{q-1}\frac{\xi_n}{(\xi_n^2-1)}\left(1+\frac{2}{\xi_n^p-1}\right) \notag\\
&=-2\sum_{n=1}^{q-1}\left(\frac{1}{\xi_n-1}+\frac{1}{\xi_n+1}\right)
-8\sum_{n=1}^{q-1}\frac{\xi_n}{(\xi_n^2-1)(\xi_n^p-1)}.
\end{align*}
By Lemmas \ref{1z1} and \ref{1z-1}, we see that
\begin{align*}
\sum_{n=1}^{q-1}\left(\frac{1}{\xi_n-1}+\frac{1}{\xi_n+1}\right)=-\frac{q-1}{2}+\frac{q-1}{2}=0,
\end{align*}
which completes the proof.
\end{proof}

\begin{corollary}
Let $p$ and $q$ be odd positive integers  with $p, q \geq 3$ and $(p, q)=1.$ Then,
\begin{align*}
p\sum_{n=1}^{q-1}\frac{\cot\Big(\dfrac{\pi np}{q}\Big)\cot\Big(\dfrac{\pi n}{q}\Big)}{\cos^2\Big(\dfrac{\pi n}{q}\Big)}
+q\sum_{n=1}^{p-1}\frac{\cot\Big(\dfrac{\pi nq}{p}\Big)\cot\Big(\dfrac{\pi n}{p}\Big)}{\cos^2\Big(\dfrac{\pi n}{p}\Big)}
=\frac{1}{3}(p^2+q^2-2).
\end{align*}
\end{corollary}

\begin{proof}
Applying Theorem \ref{cpcc2}, and Lemmas \ref{lemR22} and \ref{lemr21+}, we obtain
\begin{align*}
&p\sum_{n=1}^{q-1}\frac{\cot\Big(\dfrac{\pi np}{q}\Big)\cot\Big(\dfrac{\pi n}{q}\Big)}{\cos^2\Big(\dfrac{\pi n}{q}\Big)}
+q\sum_{n=1}^{p-1}\frac{\cot\Big(\dfrac{\pi nq}{p}\Big)\cot\Big(\dfrac{\pi n}{p}\Big)}{\cos^2\Big(\dfrac{\pi n}{p}\Big)}\notag\\
&=-8p\sum_{n=1}^{q-1}\frac{\xi_n}{(\xi_n^2-1)(\xi_n^p-1)}-8q\sum_{n=1}^{p-1}\frac{\o_n}{(\o_n^2-1)(\o_n^q-1)} \notag\\
&=-4p\sum_{n=1}^{q-1}\left(\frac{1}{\xi_n-1}+\frac{1}{\xi_n+1}\right)\frac{1}{(\xi_n^p-1)}
-4q\sum_{n=1}^{p-1}\left(\frac{1}{\o_n-1}+\frac{1}{\o_n+1}\right)\frac{1}{(\o_n^q-1)}  \notag\\
&=\frac{1}{3}(p^2+q^2-9pq+3p+3q+1)+3pq-p-q-1\\
&=\frac{1}{3}(p^2+q^2-2).
\end{align*}
\end{proof}

\begin{theorem}\label{cpc3c2}
Let $p$ and $q$ be  positive integers  with $p, q \geq 2$ and $(p, q)=1.$ Then,
\begin{align*}
\sum_{n=1}^{q-1}\frac{\cot\Big(\dfrac{\pi np}{q}\Big)\cot^3\Big(\dfrac{\pi n}{q}\Big)}{\cos^2\Big(\dfrac{\pi n}{q}\Big)}
=\frac{q^2-1}{3}+16\sum_{n=1}^{q-1}\frac{\xi_n}{(\xi_n-1)^3(\xi_n^p-1)}.
\end{align*}
\end{theorem}

\begin{proof}
Similarly, by \eqref{cot} and \eqref{c2}, we see that
\begin{align*}
&\sum_{n=1}^{q-1}\frac{\cot\Big(\dfrac{\pi np}{q}\Big)\cot^3\Big(\dfrac{\pi n}{q}\Big)}{\cos^2\Big(\dfrac{\pi n}{q}\Big)}
=4\sum_{n=1}^{q-1}\frac{\xi_n(\xi_n+1)(\xi_n^p+1)}{(\xi_n-1)^3(\xi_n^p-1)}\notag\\
&=4\sum_{n=1}^{q-1}\left(1+\frac{2}{\xi_n^p-1}\right)\left(1+\frac{2}{\xi_n-1}\right)\frac{\xi_n}{(\xi_n-1)^2} \notag\\
&=4\sum_{n=1}^{q-1}\left\{\frac{\xi_n}{(\xi_n-1)^2}+\frac{2\xi_n}{(\xi_n-1)^3}+\frac{2\xi_n}{(\xi_n-1)^2(\xi_n^p-1)}
+\frac{4\xi_n}{(\xi_n-1)^3(\xi_n^p-1)}\right\} \notag\\
&=\frac{q^2-1}{3}+16\sum_{n=1}^{q-1}\frac{\xi_n}{(\xi_n-1)^3(\xi_n^p-1)},
\end{align*}
where we applied Lemmas \ref{zz1} and \ref{lemr1}.
\end{proof}

\begin{corollary}
Let $p$ and $q$ be positive integers  with $p, q \geq 2$ and $(p, q)=1.$ Then,
\begin{align*}
p\sum_{n=1}^{q-1}\frac{\cot\Big(\dfrac{\pi np}{q}\Big)\cot^3\Big(\dfrac{\pi n}{q}\Big)}{\cos^2\Big(\dfrac{\pi n}{q}\Big)}
+q\sum_{n=1}^{p-1}\frac{\cot\Big(\dfrac{\pi nq}{p}\Big)\cot^3\Big(\dfrac{\pi n}{p}\Big)}{\cos^2\Big(\dfrac{\pi n}{p}\Big)}
=\frac{1}{45}(p^4+q^4-5p^2q^2+3).
\end{align*}
\end{corollary}

\begin{proof}
From Theorem \ref{cpc3c2} and Lemma \ref{lemr2}, we can readily derive
\begin{align*}
&p\sum_{n=1}^{q-1}\frac{\cot\Big(\dfrac{\pi np}{q}\Big)\cot^3\Big(\dfrac{\pi n}{q}\Big)}{\cos^2\Big(\dfrac{\pi n}{q}\Big)}
+q\sum_{n=1}^{p-1}\frac{\cot\Big(\dfrac{\pi nq}{p}\Big)\cot^3\Big(\dfrac{\pi n}{p}\Big)}{\cos^2\Big(\dfrac{\pi n}{p}\Big)} \\&=\frac{p(q^2-1)}{3}+\frac{q(p^2-1)}{3}+\frac{1}{45}(p^4+q^4-5p^2q^2-15p^2q-15pq^2+15p+15q+3)\\
&=\frac{1}{45}(p^4+q^4-5p^2q^2+3).
\end{align*}
\end{proof}

There exists further reciprocity theorems for cotangent sums, some of which are related to reciprocity theorems for analogues or generalizations of Dedekind sums.  For example, see papers of U.~Dieter \cite{dieter} and S.~Fukuhara \cite{fukuhara}.  These papers also contain several references to further work.

\section{Evaluations of Two Trigonometric Sums}

\begin{theorem}\label{theorem1a}
Let $k$ denote an odd positive integer, and let $a$ denote a positive integer.  Then
\begin{equation}\label{thm1a}
\sum_{0<n<k/2}\df{\cos^{2a}(\pi n/k)}{\cos^{2a+2}(2\pi n/k)}=-\df12-\pi i\left(R_{k/4}+R_{3k/4}\right),
\end{equation}
where $ R_{k/4}$ and $R_{3k/4}$, respectively, denote the residues of the function
\begin{equation*}
F(z):=\df{\cos^{2a}(\pi z/k)}{\cos^{2a+2}(2\pi z/k)(e^{2\pi i z}-1)}
\end{equation*}
at the poles $k/4$ and $3k/4$.
\end{theorem}

\begin{proof}
Let $C_N$ denote an indented, positively oriented rectangle with horizonal sides passing through $\pm iN$, $N>0$, and vertical sides passing through $0$ and $k$. The left vertical side is indented by a semi-circle $C_0$ of radius $\epsilon$, $0<\epsilon<1$, in the left half-plane.  The right vertical line is indented at $z=k$ by the semi-circle $C_0+k$.  Let $I(C_N)$ denote the interior of $C_N$. Let $R_{\alpha}$ denote the residue of a pole of $F(z)$ at $z=\alpha$ on $I(C_N)$.  We apply the residue theorem to $F(z)$, integrated over $C_N$.  On $I(C_N)$, there are simple poles at $z=0,1,2,\dots,k-1$ and poles of order $2a+2$ at $z=k/4,3k/4$.
We easily see that
\begin{equation}\label{3a}
R_n=\df{1}{2\pi i}\df{\cos^{2a}(\pi n/k)}{\cos^{2a+2}(2\pi n/k)}, \qquad n=0,1,2, \dots, k-1.
\end{equation}
Also observe that
\begin{equation}\label{4a}
R_{n}=R_{k-n}, \qquad 0<n<k/2.
\end{equation}
Hence, by the residue theorem,  \eqref{3a}, and \eqref{4a},
\begin{equation}\label{5a}
\int_{C_N}F(z)dz=1+2\sum_{0<n<k/2}\df{\cos^{2a}(\pi n/k)}{\cos^{2a+2}(2\pi n/k)}+2\pi i\left(R_{k/4}+R_{3k/4}\right).
\end{equation}

We now directly calculate $\int_{C_N}F(z)dz$.  Observe that
$F(z)=F(z+k)$.  Hence the integrals over the vertical sides of $C_N$ cancel.
 Let $C_{NT}$ and $C_{NB}$, respectively, denote those portions of $C_N$ over the top and bottom sides. It is easy to see that the integrals over $C_{NT}$ and $C_{NB}$ tend to 0 as $N\to\infty$.  Hence,
 \begin{equation*}
 \lim_{N\to\infty}\int_{C_N}F(z)dz=0.
 \end{equation*}
Hence, by \eqref{5a},
\begin{equation}\label{5.5a}
1+2\sum_{0<n<k/2}\df{\cos^{2a}(\pi n/k)}{\cos^{2a+2}(2\pi n/k)}+2\pi i\left(R_{k/4}+R_{3k/4}\right)=0,
\end{equation}
which is easily seen to be equivalent to \eqref{thm1a}.
\end{proof}

The following corollary is identical to \eqref{cc}.

\begin{corollary}\label{cor1} For any positive odd integer $k$,
\begin{equation}
\sum_{0<n<k/2}\df{\cos^{2}(\pi n/k)}{\cos^{4}(2\pi n/k)}=\begin{cases}
-\df12+\df{k}{24}(3+4k+3k^2+2k^3), \quad &\text{if }\,k\equiv 1\pmod4,\\[0.1in]
-\df12+\df{k}{24}(-3+4k-3k^2+2k^3),  &\text{if }\,k\equiv 3\pmod4.
\end{cases}
\end{equation}
\end{corollary}

\begin{proof}
Let $a=1$ in Theorem \ref{theorem1a}.
To calculate $R_{k/4}, R_{3k/4}$ we use \emph{Mathematica}.  First,
\begin{equation}\label{7a}
R_{k/4}= \begin{cases}
\df{ki}{48\pi}\left((3-3i)+4k+3k^2+2k^3\right), \quad &\emph{if }\,k\equiv1\pmod4,\\[0.1in]
\df{ki}{48\pi}\left((-3-3i)+4k-3k^2+2k^3\right),  &\emph{if }\,k\equiv3\pmod4.
\end{cases}
\end{equation}
Secondly,
\begin{equation}\label{8a}
R_{3k/4}= \begin{cases}
\df{ki}{48\pi}\left((3+3i)+4k+3k^2+2k^3\right), \quad &\emph{if }\,k\equiv1\pmod4,\\[0.1in]
\df{ki}{48\pi}\left((-3+3i)+4k-3k^2+2k^3\right),  &\emph{if }\,k\equiv3\pmod4.
\end{cases}
\end{equation}
It follows from \eqref{7a} and \eqref{8a} that
\begin{equation}\label{9a}
2\pi i\left(R_{k/4}+R_{3k/4}\right)=
\begin{cases}
-\df{k}{12}(3+4k+3k^2+2k^3),\quad &\emph{if }\,k\equiv 1\pmod4,\\[0.1in]
-\df{k}{12}(-3+4k-3k^2+2k^3), &\emph{if }\,k\equiv 3\pmod4.
\end{cases}
\end{equation}
Using \eqref{thm1a} and \eqref{9a}, we find that
\begin{equation*}
\sum_{0<n<k/2}\df{\cos^2(\pi n/k)}{\cos^4(2\pi n/k)}=
\begin{cases}
-\df12+\df{k}{24}(3+4k+3k^2+2k^3), \quad &\emph{if }\,k\equiv 1\pmod4,\\[0.1in]
-\df12+\df{k}{24}(-3+4k-3k^2+2k^3),  &\emph{if }\,k\equiv 3\pmod4,
\end{cases}
\end{equation*}
which completes the proof of Corollary \ref{cor1}.
\end{proof}

\begin{theorem}\label{theorem2a}
Let $k$ denote an odd positive integer, and let $a$ denote a positive integer.  Then
\begin{equation}\label{thm2a}
\sum_{0<n<k/2}\df{\sin^{2a}(\pi n/k)}{\sin^{2a+2}(2\pi n/k)}=-\pi i\left(R_{0}+R_{k/2}\right),
\end{equation}
where $ R_{0}$ and $R_{k/2}$, respectively, denote the residues of the function
\begin{equation}\label{1aa}
G(z):=\df{\sin^{2a}(\pi z/k)}{\sin^{2a+2}(2\pi z/k)(e^{2\pi i z}-1)}
\end{equation}
at the poles $0$ and $k/2$.
\end{theorem}

\begin{proof}
As in the previous proof, let $C_N$ denote the  indented, positively oriented rectangle with horizonal sides passing through $\pm iN$, $N>0$, and vertical sides passing through $0$ and $k$. Let $I(C_N)$ denote the interior of $C_N$. Let $R_{\alpha}$ denote the residue of a pole of $G(z)$ at $z=\alpha$ on $I(C_N)$.  We apply the residue theorem to $G(z)$, integrated over $C_N$.  On $I(C_N)$, there are simple poles at $z=1,2,\dots,k-1$, a pole of order $3$ at $z=0$, and a pole of order $2a+2$ at $z=k/2$.
We easily see that
\begin{equation}\label{3aa}
R_n=\df{1}{2\pi i}\df{\sin^{2a}(\pi n/k)}{\sin^{2a+2}(2\pi n/k)}, \qquad n= 1,2, \dots, k-1.
\end{equation}
Also observe that
\begin{equation}\label{4aa}
R_{n}=R_{k-n}, \qquad 0<n<k/2.
\end{equation}
Hence, by the residue theorem,  \eqref{3aa}, and \eqref{4aa},
\begin{equation}\label{5aa}
\int_{C_N}G(z)dz=2\sum_{0<n<k/2}\df{\sin^{2a}(\pi n/k)}{\sin^{2a+2}(2\pi n/k)}+2\pi i\left(R_{0}+R_{k/2}\right).
\end{equation}

We now directly calculate $\int_{C_N}F(z)dz$.  The details are analogous to those in the proof of Theorem \ref{theorem1a}. Thus,
\begin{equation}\label{5aaa}
\lim_{N\to\infty}\int_{C_N}G(z)dz=0.
\end{equation}
Hence, by \eqref{5aa} and \eqref{5aaa},
\begin{equation*}
2\sum_{0<n<k/2}\df{\sin^{2a}(\pi n/k)}{\sin^{2a+2}(2\pi n/k)}+2\pi i\left(R_{0}+R_{k/2}\right)=0,
\end{equation*}
which is easily seen to be equivalent to \eqref{thm2a}.
\end{proof}

We provide a second proof of \eqref{triangular}.

\begin{corollary}\label{cor2a}
If $k$ is an odd positive integer, then
\begin{equation}\label{16.5a}
\sum_{1\leq n<k/2}\df{\sin^2(\pi n/k)}{\sin^4(2\pi n/k)}=\df{k^4+6k^2-7}{96}.
\end{equation}
\end{corollary}

\begin{proof}
Let $a=1$ in Theorem \ref{theorem2a}.
It is not difficult to use the definition \eqref{1aa} to calculate $R_0$ and $R_{k/2}$ directly. Or, one can resort to \emph{Mathematica}.  In either case,
\begin{equation}\label{14a}
R_0=i\df{-7+k^2}{96\pi}
\end{equation}
and
\begin{equation}\label{15a}
R_{k/2}=\df{ik^2}{96\pi}(5+k^2).
\end{equation}
If we use \eqref{14a} and \eqref{15a} in Theorem \ref{theorem2a} when $a=1$ and simplify, we easily obtain \eqref{16.5a}.
\end{proof}

\section{Further Reciprocity Theorems}

The following theorem is identical to Corollary \ref{mc3}.  We provide another proof, because it leads to a generalization mentioned after the proof.

\begin{theorem}\label{aaa1} Let $p$ and $q$ be distinct, positive, odd integers, each $\geq3$.  Then
\begin{equation*}
 p\sum_{n=1}^{q-1}\df{\cot(\pi n/q)\cot(\pi np/q)}{\sin^2(\pi n/q)}
+q\sum_{n=1}^{p-1}\df{\cot(\pi n/p)\cot(\pi nq/p)}{\sin^2(\pi n/p)}=\df{p^4+q^4-5p^2q^2+3}{45}.
\end{equation*}
\end{theorem}

\begin{proof} Let
\begin{equation*}
f(z):=\df{\cot(\pi z/q)\cot(\pi pz/q)}{(e^{2\pi iz}-1)\sin^2(\pi z/q)}.
\end{equation*}
Integrate $f(z)$ over a rectangle $C_N$, with horizontal sides $z=x\pm iN, 0\leq x\leq q$, and with vertical sides $z=0+iy, q+iy, -N\leq y\leq N$. The contour contains semi-circular indentations of radius $\epsilon, 0<\epsilon<1$, at $0$ and $q$, with the former lying in the left-half plane, and the latter lying in the half-plane $x\leq q$. On the interior of $C_N$,  $f(z)$ has simple poles at $z=n, 1\leq n\leq q-1$; simple poles at $z=nq/p, 1\leq n\leq p-1$; and a pole of order $5$ at the origin. Let $R_{\alpha}$ denote the residue of a pole $\alpha$ of $f(z)$. By the residue theorem,
\begin{align}\label{aaa11}
\int_{C_N}f(z)dz=&2\pi i\sum_{n=1}^{q-1}R_n+2\pi i\sum_{n=1}^{p-1}R_{nq/p}+2\pi i R_0\notag\\
=&\sum_{n=1}^{q-1}\df{\cot(\pi n/q)\cot(\pi np/q)}{\sin^2(\pi n/q)}
+\df{2iq}{p}\sum_{n=1}^{p-1}\df{\cot(\pi n/p)}{(e^{2\pi inq/p}-1)\sin^2(\pi n/p)}+2\pi iR_0.
\end{align}
Now,
\begin{align}\label{aaa12}
\df{\cot(\pi n/p)}{(e^{2\pi inq/p}-1)\sin^2(\pi n/p)}=&\df{\cot(\pi n/p)}{e^{\pi inq/p}(2i)\sin(\pi nq/p)\sin^2(\pi n/p)}\notag\\
=&-\df{i}{2}\,\df{\cos(\pi nq/p)-i\sin(\pi nq/p)}{\sin(\pi nq/p)\sin^2(\pi n/p)}\cot(\pi n/p)\notag\\
=&-\df{i}{2}\,\df{\cos(\pi nq/p)\cot(\pi n/p)}{\sin(\pi nq/p)\sin^2(\pi n/p)}-\df12\,\df{\sin(\pi nq/p)\cot(\pi n/p)}{\sin(\pi nq/p)\sin^2(\pi n/p)}.
\end{align}
When the far right side of \eqref{aaa12} is substituted into \eqref{aaa11}, the contribution of the latter quotient in \eqref{aaa12} will be equal to 0, since the terms with index $n$ and $p-n$, $0<n<p/2$, have opposite signs and so cancel.  Hence, we find that \eqref{aaa11} can be written in the form
\begin{equation}\label{aaa13}
\int_{C_N}f(z)dz
=\sum_{n=1}^{q-1}\df{\cot(\pi n/q)\cot(\pi np/q)}{\sin^2(\pi n/q)}
+\df{q}{p}\sum_{n=1}^{p-1}\df{\cot(\pi n/p)\cot(\pi nq/p)}{\sin^2(\pi n/p)}+2\pi iR_0.
\end{equation}

By the use of \emph{Mathematica},
\begin{equation}\label{aaa14}
2\pi i R_0=-\df{p^4+q^4-5p^2q^2+3}{45p}.
\end{equation}
Putting \eqref{aaa14} in \eqref{aaa13}, we deduce that
\begin{equation}\label{aaa15}
\int_{C_N}f(z)dz
=\sum_{n=1}^{q-1}\df{\cot(\pi n/q)\cot(\pi np/q)}{\sin^2(\pi n/q)}
+\df{q}{p}\sum_{n=1}^{p-1}\df{\cot(\pi n/p)\cot(\pi nq/p)}{\sin^2(\pi n/p)}-\df{p^4+q^4-5p^2q^2+3}{45p}.
\end{equation}
By a now familiar argument, a direct calculation gives
 \begin{equation}\label{aaa16}
 \lim_{N\to\infty}\int_{C_N}f(z)dz=0.
 \end{equation}
 Combining \eqref{aaa15} and \eqref{aaa16} and then multiplying both sides by $p$, we conclude that
 \begin{equation*}
 p\sum_{n=1}^{q-1}\df{\cot(\pi n/q)\cot(\pi np/q)}{\sin^2(\pi n/q)}
+q\sum_{n=1}^{p-1}\df{\cot(\pi n/p)\cot(\pi nq/p)}{\sin^2(\pi n/p)}=\df{p^4+q^4-5p^2q^2+3}{45},
\end{equation*}
which completes the proof of Theorem \ref{aaa1}.
\end{proof}

Generalizing our argument, we can derive a reciprocity theorem for
$$\sum_{n=1}^{p-1}\df{\cot(\pi n/p)\cot(\pi nq/p)}{\sin^{2m}(\pi n/p)}$$
for any positive integer $m$.
The next two results give these reciprocity theorems for $m=2,3$, respectively.

 \begin{theorem} Let $p$ and $q$ be distinct, positive, odd integers, each $\geq3$.  Then
\begin{gather*}
 p\sum_{n=1}^{q-1}\df{\cot(\pi n/q)\cot(\pi np/q)}{\sin^4(\pi n/q)}
+q\sum_{n=1}^{p-1}\df{\cot(\pi n/p)\cot(\pi nq/p)}{\sin^4(\pi n/p)}\\=
\df{2p^6+2q^6-7p^4q^2-7p^2q^4+7p^4+7q^4-35p^2q^2+31}{945}.
\end{gather*}
\end{theorem}



 \begin{theorem} Let $p$ and $q$ be distinct, positive, odd integers, each $\geq3$.  Then
\begin{gather*}
 p\sum_{n=1}^{q-1}\df{\cot(\pi n/q)\cot(\pi np/q)}{\sin^6(\pi n/q)}
+q\sum_{n=1}^{p-1}\df{\cot(\pi n/p)\cot(\pi nq/p)}{\sin^6(\pi n/p)}\\=
\df{3 p^8+3 q^8+20p^6+20q^6 +56p^4+56q^4-280p^2q^2-10p^6q^2-10p^2q^6-7p^4q^4+289}{14175}.
\end{gather*}
\end{theorem}

\begin{theorem}\label{r0} (Four-sum reciprocity relation) Let $p,q,r,$ and $s$ be distinct, odd positive integers, relatively prime in pairs.  Then
\begin{align}
&\df{1}{q}\sum_{n=1}^{q-1}\df{\cot(\pi n/q)\cot(\pi pn/q)\cot(\pi rn/q)\cot(\pi sn/q)}{\sin^2(\pi n/q)}\notag\\
&+\df{1}{p}\sum_{n=1}^{p-1}\df{\cot(\pi n/p)\cot(\pi qn/p)\cot(\pi rn/p)\cot(\pi sn/p)}{\sin^2(\pi n/p)}\notag\\
&+\df{1}{r}\sum_{n=1}^{r-1}\df{\cot(\pi n/r)\cot(\pi pn/r)\cot(\pi qn/r)\cot(\pi sn/r)}{\sin^2(\pi n/r)}\notag\\
&+\df{1}{s}\sum_{n=1}^{s-1}\df{\cot(\pi n/s)\cot(\pi pn/s)\cot(\pi rn/s)\cot(\pi qn/s)}{\sin^2(\pi n/s)}\notag\\
=&\df{1}{1890pqrs}\left\{35(p^2q^2r^2+p^2q^2s^2+q^2r^2s^2+p^2r^2s^2)\right.\notag\\
&-7\left(p^4q^2 + p^4r^2 + q^4p^2 + r^4p^2 + s^4p^2 + q^4s^2 + q^4r^2 + r^4q^2 + s^4q^2 + s^4r^2 +r^4s^2 + p^4s^2\right)\notag\\
&\left.+2(p^6+q^6+r^6+s^6)-21(p^2+q^2+r^2+s^2)+20\right\}.\label{r1}
\end{align}
\end{theorem}

\begin{proof} Let
\begin{equation*}
f(z):=\df{\cot(\pi z/q)\cot(\pi pz/q)\cot(\pi rz/q)\cot(\pi sz/q)}{\sin^2(\pi z/q)(e^{2\pi i z}-1)},\quad z=x+iy.
\end{equation*}
Let $C_N$ denote a positively oriented rectangle with vertical sides passing through $x=0$ and $x=q$, horizontal sides $0\leq x \leq q, y=N$, and $0\leq x \leq q, y=-N$.  Furthermore, the left vertical line is indented by a semicircle of radius $\epsilon <1$ around $z=0$ and lying in the left half-plane, and lastly the right vertical line has an indention of radius $\epsilon<1$ around $z=q$ and lying in the half-plane $x\leq q$.
On the interior of $C_N$, $f(z)$ has a pole of order 7 at the origin, and simple poles at $z=n, 1\leq n \leq q-1; z= qn/p, 1\leq n\leq p-1; z=qn/r, 1\leq n\leq r-1$; and $z=qn/s, 1\leq n\leq s-1$.

We apply the residue theorem to $f(z)$ on $C_N$.  Let $R_a$ denote the residue of $f(z)$ at a pole $a$. First,
\begin{equation*}
R_n=\df{\cot(\pi n/q)\cot(\pi pn/q)\cot(\pi rn/q)\cot(\pi sn/q)}{2\pi i\sin^2(\pi n/q)},\quad 1\leq n\leq q-1.
\end{equation*}
Thus, from the residue theorem, we obtain the contribution
\begin{equation}\label{r2}
\sum_{n=1}^{q-1}\df{\cot(\pi n/q)\cot(\pi pn/q)\cot(\pi rn/q)\cot(\pi sn/q)}{\sin^2(\pi n/q)}.
\end{equation}

Secondly, for $1\leq n \leq p-1$,
\begin{align}\label{Ab}
R_{qn/p}=&\df{\cot(\pi n/p)\cot(\pi rn/p)\cot(\pi sn/p)}{(\pi p/q)(e^{2\pi iqn/p}-1)\sin^2(\pi n/p)}\notag\\
=&\df{q}{2\pi ip}\df{\cot(\pi n/p)\cot(\pi rn/p)\cot(\pi sn/p)\left(\cos(\pi qn/p)-i\sin(\pi qn/p)\right)}
{\sin(\pi qn/p)\sin^2(\pi n/p)}\notag\\
=&\df{q}{2\pi ip}\df{\cot(\pi n/p)\cot(\pi rn/p)\cot(\pi sn/p)\cot(\pi qn/p)}{\sin^2(\pi n/p)}\notag\\
&-\df{q}{2\pi p}\df{\cot(\pi n/p)\cot(\pi rn/p)\cot(\pi sn/p)}{\sin^2(\pi n/p)}.
\end{align}
Hence, from the residue theorem, the contributions from \eqref{Ab} total
\begin{align}
&\df{q}{p}\sum_{n=1}^{p-1}\left\{\df{\cot(\pi n/p)\cot(\pi rn/p)\cot(\pi sn/p)\cot(\pi qn/p)}{\sin^2(\pi n/p)}\right.\notag\\
&-i\left.\df{\cot(\pi n/p)\cot(\pi rn/p)\cot(\pi sn/p)}{\sin^2(\pi n/p)}\right\}\notag\\
=& \df{q}{p}\sum_{n=1}^{p-1}\df{\cot(\pi n/p)\cot(\pi rn/p)\cot(\pi sn/p)\cot(\pi qn/p)}{\sin^2(\pi n/p)},\label{r3}
\end{align}
because in the latter sum on the left-hand side of \eqref{r3}, the terms with index $n$ and $p-n$, $1\leq n\leq (p-1)/2$, cancel.

By analogous calculations, the contributions of the residue theorem from the poles $z=qn/r$, $1\leq n\leq r-1$, and $z=qn/s$, $1\leq n \leq s-1$, are, respectively,
\begin{equation}
\df{q}{r}\sum_{n=1}^{r-1}\df{\cot(\pi n/r)\cot(\pi pn/r)\cot(\pi qn/r)\cot(\pi sn/r)}{\sin^2(\pi n/r)}\label{r4}
\end{equation}
and
\begin{equation}
\df{q}{s}\sum_{n=1}^{s-1}\df{\cot(\pi n/s)\cot(\pi pn/s)\cot(\pi rn/s)\cot(\pi qn/s)}{\sin^2(\pi n/s)}.\label{r5}
\end{equation}

Lastly, using \emph{Mathematica}, we find that
\begin{align}\label{r6}
2\pi i R_0 =&-\df{1}{1890prs}\left\{35(p^2q^2r^2+p^2q^2s^2+q^2r^2s^2+p^2r^2s^2)\right.\notag\\
&-7\left(p^4q^2 + p^4r^2 + q^4p^2 + r^4p^2 + s^4p^2 + q^4s^2 + q^4r^2 + r^4q^2 + s^4q^2 + s^4r^2 +r^4s^2 + p^4s^2\right)\notag\\
&\left.+2(p^6+q^6+r^6+s^6)-21(p^2+q^2+r^2+s^2)+20\right\}.
\end{align}

In summary, bringing together \eqref{r2}, \eqref{r3}, \eqref{r4}, and\eqref{r5}, and dividing both sides by $q$, we deduce that
\begin{align}
&\df{1}{q} \int_{C_N}f(z)dz\notag \\
=&\df{1}{q}\sum_{n=1}^{q-1}\df{\cot(\pi n/q)\cot(\pi pn/q)\cot(\pi rn/q)\cot(\pi sn/q)}{\sin^2(\pi n/q)}\notag\\
&+\df{1}{p}\sum_{n=1}^{p-1}\df{\cot(\pi n/p)\cot(\pi qn/p)\cot(\pi rn/p)\cot(\pi sn/p)}{\sin^2(\pi n/p)}\notag\\
&+\df{1}{r}\sum_{n=1}^{r-1}\df{\cot(\pi n/r)\cot(\pi pn/r)\cot(\pi qn/r)\cot(\pi sn/r)}{\sin^2(\pi n/r)}\notag\\
&+\df{1}{s}\sum_{n=1}^{s-1}\df{\cot(\pi n/s)\cot(\pi pn/s)\cot(\pi rn/s)\cot(\pi qn/s)}{\sin^2(\pi n/s)}\notag\\
=&-\df{1}{1890pqrs}\left\{35(p^2q^2r^2+p^2q^2s^2+q^2r^2s^2+p^2r^2s^2)\right.\notag\\
&-7\left(p^4q^2 + p^4r^2 + q^4p^2 + r^4p^2 + s^4p^2 + q^4s^2 + q^4r^2 + r^4q^2 + s^4q^2 + s^4r^2 +r^4s^2 + p^4s^2\right)\notag\\
&\left.+2(p^6+s^6+r^6+q^6)-21(p^2+r^2+s^2+q^2)+20\right\}.\label{r7}
\end{align}

By a familiar argument, we can easily show that
\begin{equation}\label{r8}
\lim_{N\to\infty}\int_{C_N}f(z)dz =0.
\end{equation}

If we now combine \eqref{r7} and \eqref{r8}, we readily deduce \eqref{r1} to complete the proof of Theorem \ref{r0}.
\end{proof}

   Let $p_1, p_2, \dots , p_{2m}$ denote $2m$ positive odd integers, relatively prime in pairs.  By an extension of the argument that we employed to prove Theorem \ref{r0}, we can prove a $2m$-term reciprocity theorem for the $2m$ sums
   \begin{equation*}
   \df{1}{p_r}\sum_{n=1}^{p_r-1}\df{\cot(\pi n/p_r)\cot(\pi np_1/p_r)\cdots\cot(\pi np_{r-1}/p_r)\cot(\pi np_{r+1}/p_r)\cdots
   \cot(\pi n p_{2m}/p_r)}{\sin^2(\pi n/p_r)},\label{r9}
   \end{equation*}
   where $1\leq r\leq 2m$.

   It is clear that we can also obtain a $2m$-term reciprocity theorem if we replace $\sin^2(\pi n/p_r)$ by  $\sin^{2N}(\pi n/p_r),$ $1\leq r \leq 2m$, for any positive integer $N$.

\section{Two Theorems on Analogues of Gauss Sums}

Let $\chi$ denote a character modulo $k$.  Define the Gauss sum,
\begin{equation}  \label{gausssumcomplex}
G(z,\chi):=\sum_{n=1}^{k}\chi(n)e^{2\pi i nz/k}, \qquad z\in \mathbb{C}.
\end{equation}
Trivially, from \eqref{gausssumcomplex},
\begin{equation}\label{gausssumadd}
G(z,\chi)=\begin{cases} {\displaystyle\sum_{n=1}^k\chi(n)\cos(2\pi  nz/k),} \quad &\emph{if }\, \chi \,\,\emph{is even},\\
i{\displaystyle\sum_{n=1}^k\chi(n)\sin(2\pi  nz/k),} \quad &\emph{if }\, \chi \,\, \emph{is odd}.
\end{cases}
\end{equation}
For each integer $n$ and $\chi$ non-principal, we have the factorization theorem \cite[p.~9]{bew}
\begin{equation}\label{factorization}
G(n,\chi)=\chi(n)G(1,\chi)=:\chi(n)G(\chi).
\end{equation}
Recall also that if $\chi$ is a non-principal character modulo $k$ \cite[p.~10]{bew},
\begin{equation}\label{gausssum}
G(\chi)=\begin{cases}\sqrt{k}, \quad &\text{if $\chi$ is even},\\
i\sqrt{k}, &\text{if $\chi$ is odd}.
\end{cases}
\end{equation}

Set
\begin{equation}\label{M}
M(k,\chi):=\sum_{j=1}^kj\chi(j).
\end{equation}
We denote the class number of the imaginary quadratic field $Q(\sqrt{-k})$ by $h(-k)$. For $k\geq7$ and odd, the classical formula for the class number is given by \cite[p.~299]{ayoub}
\begin{equation}\label{classnumber}
h(-k)=-\df{1}{k}M(k,\chi).
\end{equation}


In this section, very roughly, we prove two general theorems in which for odd $\chi$, $\sin$ is replaced by a quotient of $\sin$'s, and which for even $\chi$, $\cos$ is replaced by a quotient of $\cos$'s in \eqref{gausssumadd}.

\begin{theorem}\label{theorem1} Let $\chi$ denote an even, non-principal character modulo  $p$. Let $a>1$ denote an odd positive integer.  Assume that $a-3<2p$. Then,
\begin{equation}\label{evaluation}
S(a,p,\chi):=\sum_{0<n<p/2}\chi(n)\df{\sin(a\pi n/p)}{\sin(\pi n/p)}=C_0p,
\end{equation}
where $C_0$ is the constant defined by
\begin{equation}\label{constants}
C_0:=\df{1}{G(\chi)}\sum_{\substack{m\geq0,n\geq1, r\geq0\\2m+2n+2ar=a-1}}\sum_{n=1}^{p-1}\df{\chi(n)}{\mu^{2n}}\,\mu^{a-1}\sum_{r=0}^1 \b_r \mu^{-2ar}
\sum_{m=0}^{\infty}\mu^{-2m},\quad |\mu|<1,
\end{equation}
where $\b_0=1$ and $\b_1=-1$.
\end{theorem}

\begin{proof} Let
\begin{equation}\label{proof1}
f(z):=\df{G(z,\chi)}{G(\chi)}\df{\sin(a \pi z/p)}{(e^{2\pi iz}-1)\sin(\pi z/p)}.
\end{equation}
Integrate over the positively oriented indented rectangle $C_N$ with vertical sides passing through the origin and $p$, and horizontal sides $z=x\pm i N$, $0\leq x \leq p$. The indented rectangle $C_N$ contains a semi-circular indentation of radius $\epsilon$, $0<\epsilon<1$, about $0$ in the left half-plane.  The rectangle also possesses an analogous indentation about $p$, lying in the half-plane to the left of the line $z=p+iy$, $-N\leq y\leq N$.  Since $G(0,\chi)=0$, the function $f$ has a removable singularity at the origin.  Thus, $f(z)$ is analytic on $C_N$ and its interior except for simple poles at $z=1,2,\dots,p-1$.  Their residues $R_n$ are easily seen to be
\begin{equation*}
R_n =\df{G(n,\chi)}{G(\chi)}\df{\sin(a \pi n/p)}{2\pi i\,\sin(\pi n/p)}, \quad 1\leq n\leq p-1.
\end{equation*}
Therefore, by the residue theorem and \eqref{factorization},
\begin{equation}\label{proof3}
\int_{C_N}f(z)dz=\sum_{n=1}^{p-1}\chi(n)\df{\sin(a\pi n/p)}{\sin(\pi n/p)}.
\end{equation}
Since $\chi$ is even and $a$ is odd,
\begin{equation*}
\chi(p-n)\df{\sin(a\pi(p-n)/p)}{\sin(\pi(p-n)/p)}=\chi(n)\df{\sin(a\pi n/p)}{\sin(\pi n/p)},\qquad 0<n<p/2.
\end{equation*}
Hence, we can rewrite \eqref{proof3} in the form
\begin{equation}\label{proof4}
\int_{C_N}f(z)dz=2\sum_{0<n<p/2}\chi(n)\df{\sin(a\pi n/p)}{\sin(\pi n/p)}=2S(a,p,\chi),
\end{equation}
by \eqref{evaluation}.

We next directly integrate $\int_{C_N}f(z)dz$. Using the definition \eqref{proof1}, we easily see that $f(z)=f(z+p)$, and so the integrals over the vertical sides of $C_N$ cancel.  Let $C_{NT} $ and $C_{NB}$ denote the upper and lower sides of $C_N$, respectively. For $z=x+iN, $ set
\begin{equation*}
\mu:=e^{-\pi ix/p}e^{\pi N/p}.
\end{equation*}
We examine $f(z)$ on each portion of $C_{NT}$.  First,
\begin{equation}\label{proof6}
\df{1}{e^{2\pi i (x+iN)}-1}=\df{1}{\mu^{-2p}-1}=-1+O(\mu^{-2p}), \qquad N\to\infty.
\end{equation}
Second, we readily see that
\begin{equation}\label{proof7}
G(x+iN,\chi)=\sum_{n=1}^{p-1}\df{\chi(n)}{\mu^{2n}}.
\end{equation}
Third, an elementary calculation gives
\begin{equation}\label{proof8}
\sin(\pi a(x+iN)/p)=-\df{1}{2i}\mu^a(1-\mu^{-2a}).
\end{equation}
Fourth,
\begin{equation}\label{proof9}
\df{1}{\sin(\pi(x+iN)/p)}=-\df{2i}{\mu(1-\mu^{-2})}=-\df{2i}{\mu}\sum_{m=0}^{\infty}\mu^{-2m}.
\end{equation}
Gathering \eqref{proof6}--\eqref{proof9} together, we find that
\begin{align}
f(x+iN)=&\df{1}{G(\chi)}\sum_{n=1}^{p-1}\df{\chi(n)}{\mu^{2n}}\,\mu^a(1-\mu^{-2a})\mu^{-1}
\sum_{m=0}^{\infty}\mu^{-2m}\left(-1+O(\mu^{-2p})\right)\notag\\
=&\df{1}{G(\chi)}\sum_{n=1}^{p-1}\df{\chi(n)}{\mu^{2n}}\,\mu^{a-1}\sum_{r=0}^1\b_r\mu^{-2ar}
\sum_{m=0}^{\infty}\mu^{-2m}\left(-1+O(\mu^{-2p})\right)\label{constant}\\
=& b_d e^{\pi Nd/p}+b_{d-1}e^{\pi N(d-1)/p}+\cdots+b_0 +b_{-1}e^{-N\pi/p}+\cdots,\label{proof10}
\end{align}
where
\begin{equation}\label{proof11}
 \b_0=1,\,\, \b_1=-1,
\end{equation}
and where $d=a-3$.

Now, let $z=x-iN$, $0\leq x \leq p$. Set
\begin{equation*}
\nu:=e^{\pi ix/p}e^{\pi N/p}.
\end{equation*}
We examine each of the terms comprising $f(x-iN)$.  First,
\begin{equation}\label{proof13}
\df{1}{e^{2\pi i (x-iN)}-1}=\df{1}{\nu^{2p}-1}=\nu^{-2p}+O(\mu^{-4p}), \qquad N\to\infty.
\end{equation}
Second, replacing $n$ by $p-n$ below, we find that
\begin{equation}\label{proof14}
G(x-iN,\chi)=\sum_{n=1}^{p-1}\chi(n)\nu^{2n}=\sum_{n=1}^{p-1}\chi(p-n)\nu^{2(p-n)}
=\nu^{2p}\sum_{n=1}^{p-1}\df{\chi(n)}{\nu^{2n}},
\end{equation}
since $\chi$ is even.
Third, by employing the same kind of calculations that produced \eqref{proof8} and \eqref{proof9}, we find that
\begin{equation}\label{proof15}
\sin(\pi a(x-iN)/p)=\df{\nu^a}{2i}(1-\nu^{-2a})
\end{equation}
and
\begin{equation}\label{proof16}
\df{1}{\sin(\pi(x-iN)/p)}=\df{2i}{\nu}\sum_{m=0}^{\infty}\nu^{-2m}.
\end{equation}
Bringing \eqref{proof13}--\eqref{proof16} together, we arrive at
\begin{align}
f(x-iN)=&\df{1}{G(\chi)}\nu^{2p}\sum_{n=1}^{p-1}\df{\chi(n)}{\nu^{2n}}\,\nu^a(1-\nu^{-2a})\nu^{-1}
\sum_{m=0}^{\infty}\nu^{-2m}\,\nu^{-2p}\left(1+O(\nu^{-2p})\right)\notag\\
=&\df{1}{G(\chi)}\sum_{n=1}^{p-1}\df{\chi(n)}{\nu^{2n}}\,\nu^{a-1}\sum_{r=0}^1\b_r\nu^{-2ar}
\sum_{m=0}^{\infty}\nu^{-2m}\left(1+O(\nu^{-2p})\right)\label{proofa7a}\\
=& c_d e^{\pi Nd/p}+c_{d-1}e^{\pi N(d-1)/p}+\cdots+c_0 +c_{-1}e^{-N\pi/p}+\cdots,\label{proof17}
\end{align}
where $\beta_r, r=0,1$, is defined in \eqref{proof11}, and, as before, $d=a-3$.

Recall that $C_{NT}$ and $C_{NB}$  denote those portions of $C_N$ in which $z=x+iN$ and $z=x-iN$, $0\leq x\leq N$, respectively. By combining like powers of $e^{\pi N/p}$, and beginning with $b_d+c_d$, we successively conclude that
$$\lim_{N\to\infty}\left(\int_{C_{NT}}b_jdx+\int_{C_{NB}}c_jdx\right)=0,\qquad d \geq j\geq 1.$$
Trivially, as $N\to\infty$,
\begin{equation*}
\lim_{N\to\infty}\left(\int_{C_{NT}}b_jdx+\int_{C_{NB}}c_jdx\right)=0,\qquad j\leq-1.
\end{equation*}
  Hence, in conclusion, there remains only one integral to investigate, namely,
\begin{align}\label{proof18}
\lim_{N\to\infty}\left(\int_{C_{NT}}b_0\, dz+\int_{C_{NB}}c_0\,dz\right)=&\int_p^0b_0\,dx+\int_0^p c_0\,dx
=-pb_0+pc_0=-2pb_0.
\end{align}
since, from \eqref{constant} and \eqref{proofa7a}, $b_0=-c_0$.

Finally, from \eqref{proof4} and \eqref{proof18},
\begin{equation}\label{proof19}
S(a,p, \chi)=-pb_0,
\end{equation}
where $b_0$ is the constant term in \eqref{constant}.  Note that the constant term arises only from terms when
\begin{equation}\label{proof20}
a-1-2m-2n-2ar=0.
\end{equation}

Examining \eqref{proof6}, \eqref{proof10}, \eqref{proof13}, and \eqref{proof17}, we see that we are ignoring secondary terms arising from the Laurent expansions of $1/(e^{2\pi iz}-1)$ on $C_{NT}$ and $C_{NB}$.  Thus, from \eqref{proof20}, we need to make the assumption that $a-3<2p$.
From \eqref{constant} and \eqref{constants}, we see that  $C_0=-b_0$, and so using \eqref{proof19},  we complete the proof of \eqref{evaluation}.
\end{proof}

As indicated above, the hypothesis $a-3<2p$ can be relaxed if further terms are employed from the Laurent series expansions  \eqref{proof6} and \eqref{proof13}.

In a similar manner, we can prove the following theorem.  Since the details of the proof are like those of Theorem \ref{theorem1}, we will forego many of them.

\begin{theorem}\label{theorem11} Let $\chi$ denote an even,  non-principal character modulo $p$, where $p$ is odd. Let $a>1$ denote an odd positive integer.  Assume that $a-3<2p$. Then
\begin{equation}\label{evaluation11}
S(a,p,\chi):=\sum_{0<n<p/2}\chi(n)\df{\cos(a\pi n/p)}{\cos(\pi n/p)}=C_0p,
\end{equation}
where $C_0$ is the constant defined by
\begin{equation}\label{constants11}
C_0:=\df{1}{G(\chi)}\sum_{\substack{m\geq0,n\geq1, r\geq0\\2m+2n+2ar=a-1}}\sum_{n=1}^{p-1}\df{\chi(n)}{\mu^{2n}}\,\mu^{a-1}\sum_{r=0}^1 \b_r \mu^{-2ar}
\sum_{m=0}^{\infty}(-1)^m\mu^{-2m}, \quad |\mu|<1,
\end{equation}
where $\b_0=\b_1=1$.
\end{theorem}

\begin{proof} Let
\begin{equation*}
f(z):=\df{G(z,\chi)}{G(\chi)}\df{\cos(a \pi z/p)}{(e^{2\pi iz}-1)\cos(\pi z/p)}.
\end{equation*}
Integrate over the same positively oriented indented rectangle $C_N$ as in the proof of Theorem \ref{theorem11}.  We see that $f(z)$ is analytic on $C_N$ and its interior except for simple poles at $z=1,2,\dots,p-1$.  The residue $R_n$, $1\leq n\leq p-1$, is easily seen to be
\begin{equation*}
R_n =\df{G(n,\chi)}{G(\chi)}\df{\cos(a \pi n/p)}{2\pi i\,\cos(\pi n/p)}, \quad 1\leq n\leq p-1.
\end{equation*}
Note that because $a$ is odd, $f(z)$ does not have a pole at $z=\tfrac12 p$.
Therefore, by the residue theorem and \eqref{factorization},
\begin{equation}\label{proof311}
\int_{C_N}f(z)dz=\sum_{n=1}^{p-1}\chi(n)\df{\cos(a\pi n/p)}{\cos(\pi n/p)}.
\end{equation}
Since $\chi$ is even, $R_n=R_{p-n}, 0<n<p/2$.
Hence, we can rewrite \eqref{proof311} in the form
\begin{equation*}
\int_{C_N}f(z)dz=2\sum_{0<n<p/2}\chi(n)\df{\cos(a\pi n/p)}{\cos(\pi n/p)}=2S(a,p,\chi),
\end{equation*}
by \eqref{evaluation11}.

We next directly calculate $\int_{C_N}f(z)dz$.  By periodicity, the integrals of $f(z)$ over the vertical sides of $C_N$ cancel.  Let $C_{NT} $ and $C_{NB}$ denote the upper and lower sides of $C_N$, respectively. Set
\begin{equation*}
\mu:=e^{-\pi ix/p}e^{\pi N/p}.
\end{equation*}
We examine each portion of $f(z)$ on $C_{NT}$, where we put $z=x+iN$.  First,
\begin{equation}\label{proof611}
\df{1}{e^{2\pi i (x+iN)}-1}=\df{1}{\mu^{-2p}-1}=-1+O(\mu^{-2p}), \qquad N\to\infty.
\end{equation}
Second, we readily see that
\begin{equation}\label{proof711}
G(x+iN,\chi)=\sum_{n=1}^{p-1}\df{\chi(n)}{\mu^{2n}}.
\end{equation}
Third, an elementary calculation gives
\begin{equation}\label{proof811}
\cos(\pi a(x+iN)/p)=\df{1}{2}\mu^a(1+\mu^{-2a}).
\end{equation}
Fourth,
\begin{equation}\label{proof911}
\df{1}{\cos(\pi(x+iN)/p)}=\df{2}{\mu(1+\mu^{-2})}=\df{2}{\mu}\sum_{m=0}^{\infty}(-1)^m\mu^{-2m}.
\end{equation}
Using \eqref{proof611}--\eqref{proof911}, we find that
\begin{align}
f(x+iN)=&\df{1}{G(\chi)}\sum_{n=1}^{p-1}\df{\chi(n)}{\mu^{2n}}\,\mu^a(1+\mu^{-2a})\mu^{-1}
\sum_{m=0}^{\infty}(-1)^m\mu^{-2m}\left(-1+O(\mu^{-2p})\right)\notag\\
=&\df{1}{G(\chi)}\sum_{n=1}^{p-1}\df{\chi(n)}{\mu^{2n}}\,\mu^{a-1}\sum_{r=0}^1\b_r\mu^{-2ar}
\sum_{m=0}^{\infty}(-1)^m\mu^{-2m}\left(-1+O(\mu^{-2p})\right)\notag\\
=& b_d e^{\pi Nd/p}+b_{d-1}e^{\pi N(d-1)/p}+\cdots+b_0 +b_{-1}e^{-N\pi/p}+\cdots,\label{proof1011}
\end{align}
where
\begin{equation*}
 \b_0=\b_1=1
\end{equation*}
and where $d=a-3$.

Arguing as we did in the previous proof, we find that
\begin{align}
f(x-iN)
=&\df{1}{G(\chi)}\sum_{n=1}^{p-1}\df{\chi(n)}{\mu^{2n}}\,\mu^{a-1}\sum_{r=0}^1\b_r\mu^{-2ar}
\sum_{m=0}^{\infty}(-1)^m\mu^{-2m}\left(1+O(\mu^{-2p})\right)\notag\\
=& c_d e^{\pi Nd/p}+c_{d-1}e^{\pi N(d-1)/p}+\cdots+c_0 +c_{-1}e^{-N\pi/p}+\cdots. \label{proof1111}
\end{align}

With the use of \eqref{proof1011} and \eqref{proof1111}, the remainder of the proof follows along the same lines as the proof for Theorem \ref{theorem1}.

\end{proof}

In \cite[p.~354, Theorem 5.1]{bbyz}, the following analogue of Theorem \ref{theorem1} is proved.  Suppose that $a,b$, and $k$ are  positive integers with $b$ odd and greater than 1.  Then
\begin{gather*}
\sum_{0<n<k/2}\df{\sin^a(2\pi bn/k)}{\sin^a(2\pi n/k)}=-\df12b^a
+k\sum_{\substack{m,n,r\geq 0\\2bn+2m+rk=ab-a}}(-1)^n\binom{a}{n}\binom{a-1+m}{m},
\end{gather*}
where in the case $r=0$, the terms are to be multiplied by $\frac12$.

\bigskip

\begin{center}\textbf{Examples}\end{center}

\bigskip

\begin{example}
Let $a=3$.  There is one term satisfying the conditions $m\geq 0, n\geq 1, r\geq 0, 2=2m+2n+2r$, namely, when $(m,n,r)=(0,1,0)$.  Thus, using \eqref{constants} and \eqref{constants11}, respectively, and also \eqref{gausssum}, we conclude that
\begin{equation*}
\sum_{0<n<p/2}\chi(n)\df{\sin(3\pi n/p)}{\sin(\pi n/p)}=\sqrt{p}
\end{equation*}
\end{example}
\noindent and
\begin{equation*}
\sum_{0<n<p/2}\chi(n)\df{\cos(3\pi n/p)}{\cos(\pi n/p)}=\sqrt{p}.
\end{equation*}

\begin{example}
Let $a=5$.  There are three terms $(m,n,r)$ satisfying the conditions $4=2m+2n+2r$, namely, for $(m,n,r)=(0,2,0), (1,1,0), (0,1,1)$.  Thus,
\begin{align*}
\sum_{0<n<p/2}\chi(n)\df{\sin(5\pi n/p)}{\sin(\pi n/p)}=\left\{\chi(2)+1-1\right\}\sqrt{p}\\
=\begin{cases}\sqrt{p}, \quad &\textup{if } p\equiv 1,7, \pmod8, \\
-\sqrt{p}, \quad &\textup{if } p\equiv 3,5 \pmod8.\end{cases}
\end{align*}
and
\begin{align*}
\sum_{0<n<p/2}\chi(n)\df{\cos(5\pi n/p)}{\cos(\pi n/p)}=\left\{\chi(2)-1+1\right\}\sqrt{p}\\
=\begin{cases}\sqrt{p}, \quad &\textup{if } p\equiv 1,7 \pmod8,\\
-\sqrt{p}, \quad &\textup{if } p\equiv 3,5 \pmod8.\end{cases}
\end{align*}
\end{example}

\section{Sums with Multiple $\sin$'s}

\begin{theorem}\label{theorem2} Let $\chi$ denote an even, non-principal character modulo the prime $p$. Let $a_1, a_2, \dots , a_k$ denote $k$ odd positive integers, with each $>1$. Assume that $a_1+a_2+\cdots+a_k-k-2<2p$.  Then
\begin{equation}\label{evaluation2}
S(a_1,a_2,\dots, a_k,p,\chi):=\sum_{0<n<p/2}\chi(n)
\df{\sin(a_1\pi n/p)\sin(a_2\pi n/p)\cdots\sin(a_k\pi n/p)}{\sin^k(\pi n/p)}=C_0p,
\end{equation}
where $C_0$ is the constant defined by
\begin{align}\label{constants2}
C_0:=&\df{1}{G(\chi)}\sum_{\substack{m\geq0,n\geq1,0\leq r\leq a_1+a_2+\cdots+a_k\\2m+2n+2r=a_1+a_2+\cdots+a_k-k}}
\sum_{n=1}^{p-1}\df{\chi(n)}{\mu^{2n}}\notag\\
&\times\sum_{r=0}^{a_1+a_2+\cdots+a_k}\beta_r\mu^{-2r}\mu^{a_1+a_2+\cdots+a_k-k}
\sum_{m=0}^{\infty}\binom{k+m-1}{k-1}\mu^{-2m}, \quad |\mu|<1,
\end{align}
where
\begin{equation*}
\sum_{r=0}^{a_1+a_2+\cdots+a_k}\beta_r\mu^{-2r}:=\prod_{j=1}^k(1-\mu^{-2a_j}).
\end{equation*}
\end{theorem}

\begin{proof}  We provide a sketch of the proof, since the details are similar to those in the proof of Theorem \ref{theorem1} in which $k=1$.  Let
\begin{equation*}
f(z):=\df{G(z,\chi)}{G(\chi)}\df{\sin(a_1 \pi z/p)\sin(a_2 \pi z/p)\cdots\sin(a_k\pi z/p)}
{(e^{2\pi iz}-1)\sin^k(\pi z/p)}.
\end{equation*}
Since $f(z)$ is analytic at the origin, its only poles on the interior of $C_N$ are at $z=1,2,\dots, p-1$.  The residue of $f$ at $z=n, 1\leq n\leq p-1,$ is given by
\begin{equation*}
R_n =\df{G(n,\chi)}{G(\chi)}\df{\sin(a_1 \pi n/p)\sin(a_2\pi n/p)\cdots\sin(a_k\pi n/p)}{2\pi i\,\sin^k(\pi n/p)}.
\end{equation*}
Thus, by the residue theorem and \eqref{factorization},
\begin{equation}\label{proof23}
\int_{C_N}f(z)dz=\sum_{n=1}^{p-1}\chi(n)\df{\sin(a\pi n/p)\sin(a_2\pi n/p)\cdots\sin(a_k\pi n/p)}{\sin^k(\pi n/p)}.
\end{equation}
Since $\chi$ is even and $a_j$ is odd, $1\leq j\leq k$,
\begin{gather*}
\chi(p-n)\df{\sin(a_1\pi(p-n)/p)\sin(a_2\pi(p- n)/p)\cdots\sin(a_k\pi (p-n)/p)}{\sin(\pi(p-n)/p)}\\=\chi(n)\df{\sin(a_1\pi n/p)\sin(a_2\pi n/p)\cdots\sin(a_k\pi n/p)}{\sin^k(\pi n/p)},\qquad 0<n<p/2.
\end{gather*}
Hence, we can rewrite \eqref{proof23} in the form
\begin{align}\label{proof24}
\int_{C_N}f(z)dz=&2\sum_{0<n<p/2}\chi(n)
\df{\sin(a_1\pi n/p)\sin(a_2\pi n/p)\cdots\sin(a_k\pi n/p)}{\sin^k(\pi n/p)}\notag\\
=&2S(a_1,a_2,\dots,a_k,p,\chi),
\end{align}
by \eqref{evaluation2}.

We now integrate $f(z)$ over $C_N$. First, consider $C_{NT}$.
The procedure is the same as that for Theorem \ref{theorem1}. Let
\begin{equation*}
\mu:=e^{-\pi ix/p}e^{\pi N/p}.
\end{equation*}
We examine each portion of $f(z)$ on $C_{NT}$, where we put $z=x+iN$.  First,
\begin{equation}\label{proof611a}
\df{1}{e^{2\pi i (x+iN)}-1}=\df{1}{\mu^{-2p}-1}=-1+O(\mu^{-2p}), \qquad N\to\infty.
\end{equation}
Also,
\begin{equation}\label{proof711a}
G(x+iN,\chi)=\sum_{n=1}^{p-1}\df{\chi(n)}{\mu^{2n}}.
\end{equation}
As in \eqref{proof8}, we have
\begin{equation}\label{proof25}
\sin(\pi a_j(x+iN)/p)=-\df{1}{2i}\mu^{a_j}(1-\mu^{-2a_j}), \quad 1\leq j \leq k.
\end{equation}
From \eqref{proof9},
\begin{equation}\label{proof26}
\df{1}{\sin(\pi(x+iN)/p)}=-\df{2i}{\mu(1-\mu^{-2})}.
\end{equation}
Lastly, by the binomial theorem,
\begin{equation}\label{binomial}
(1-\mu^{-2})^{-k} =\sum_{m=0}^{\infty}\binom{k+m-1}{k-1}\mu^{-2m}.
\end{equation}
Employing \eqref{proof611a}--\eqref{binomial}, we conclude that
\begin{equation}
f(x+iN)=\df{1}{G(\chi)}\sum_{n=1}^{p-1}\df{\chi(n)}{\mu^{2n}}\,\prod_{j=1}^k\mu^{a_j}(1-\mu^{-2a_j})
\mu^{-k}\sum_{m=0}^{\infty}\binom{k+m-1}{k-1}\mu^{-2m}
\left(-1+O(\mu^{-2p}\right)).\label{proof27}
\end{equation}

By the same procedure that we used in obtaining \eqref{proofa7a}, we can derive a formula for $f(x-iN)$ analogous to that in \eqref{proof27}.
Proceeding as before,
 we obtain the evaluation
\begin{equation}\label{CO}
\lim_{N\to\infty}\int_{C_N}f(z)dz=2pC_0,
\end{equation}
where the value of $C_0$ is given \eqref{constants2}.
Combining \eqref{proof24} with \eqref{CO},
we conclude our proof of Theorem \ref{theorem2}.
\end{proof}

 At the conclusion of the proof of Theorem \ref{theorem1}, we remarked that the hypothesis $a-3<2p$ could be relaxed.  Note  that in the proof of Theorem \ref{theorem2} to ensure that the error term does not affect the constant term, we need to make the assumption that $a_1+a_2+\cdots+a_k-k-2<2p$.    Thus, by refining the error term in \eqref{proof27}, the condition
$a_1+a_2+\cdots+a_k-k-2<2p$ can be relaxed.

Theorem \ref{theorem2} is a generalization of \cite[p.~346, Cor.~4.5]{bbyz}.  Namely, if $k$ is an odd positive integer and $\chi$ is an even character modulo $k$, then
\begin{gather*}
\sum_{0<n<k/2}\chi(n)\df{\sin(3\pi n/k)\sin(5\pi n/k)\sin(7\pi n/k)}{\sin^3(\pi n/k)}\\
=\sqrt{k}\sum_{m=0}^6\chi(m)\left\{\binom{8-m}{2}-\binom{5-m}{2}-\binom{3-m}{2}   \right\}.
\end{gather*}

\section{A Sum with Characters and Trigonometric Functions with Two Distinct Periods}

\begin{theorem}
Let $p$ and $q$ be relatively prime, positive integers.  Let $\chi_1$ and $\chi_2$ denote non-principal characters modulo $p$ and $q$, respectively.  Suppose that $a_1$ and $a_2$ are odd positive integers, and that $b_1$ and $b_2$ denote even positive integers.  Assume that
\begin{equation}\label{66a}
q(a_1b_1-a_1-3)+p(a_2b_2-a_2-3)<2pq.
\end{equation}
 Define
\begin{equation*}
S(a_1,a_2,b_1,b_2,p,q,\chi_1,\chi_2):= \sum_{\substack{0<n\leq pq-1\\n\neq kp,kq}}\chi_1(n)\chi_2(n)
\df{\sin^{a_1}(b_1\pi n/p)}{\sin^{a_1+1}(\pi n/p)}\df{\sin^{a_2}(b_2\pi n/q)}{\sin^{a_2+1}(\pi n/q)}.
\end{equation*}
Then,
\begin{align}
S(a_1,a_2,b_1,b_2,p,q,\chi_1,\chi_2)=&pq(-C_0+C_0^{\prime})
+4 \df{M(p,\chi_1)}{G(\chi_1)}\df{M(q,\chi_2)}{G(\chi_2)}b_1^{a_1}b_2^{a_2}\notag\\
&-2i\df{b_1^{a_1}M(p,\chi_1)}{G(\chi_1)}\chi_2(p)\sum_{k=1}^{q-1}\chi_2(k)\df{\sin^{a_2}(b_2\pi kp/q)}{\sin^{a_2+1}(\pi kp/q)}\notag\\
&-2i\df{b_2^{a_2}M(q,\chi_2)}{G(\chi_2)}\chi_1(q)\sum_{k=1}^{p-1}\chi_1(k)\df{\sin^{a_1}(b_1\pi kq/p)}{\sin^{a_1+1}(\pi kq/p)},\label{chieftheorem}
\end{align}
where $M(k,\chi)$ is defined in \eqref{M}, and where $C_0$ and $C_0^{\prime}$ are defined by \eqref{64} and \eqref{65}, respectively.
\end{theorem}

If $\chi_1$ and $\chi_2$ are odd, then the two sums on the right-hand side of \eqref{chieftheorem} have been evaluated by two of the present authors in \cite{mathann}.  It was assumed in \cite{mathann} that $\chi_1$ and $\chi_2$ are odd, so that the sums could be evaluated in terms of class numbers \eqref{classnumber}.  However, the evaluations can be expressed in terms of the sums \eqref{M}, and so the restriction that the characters be odd is unnecessary.

\begin{proof}
Define
\begin{equation}\label{F}
F(z):=\df{G(z,\chi_1)}{G(\chi_1)}\df{G(z,\chi_2)}{G(\chi_2)}
\df{\sin^{a_1}(b_1\pi z/p)}{\sin^{a_1+1}(\pi z/p)}\df{\sin^{a_2}(b_2\pi z/q)}{\sin^{a_2+1}(\pi z/q)}
\df{1}{e^{2\pi i z}-1}.
\end{equation}
Let $C_N, N>0$, denote the indented rectangle with horizontal sides passing through $\pm iN$ and vertical sides passing through $0$ and $pq$, with semicircular indentions around $z=0$ of radius $\epsilon<1$ in the left half-plane and around $z=pq$ of radius $\epsilon<1$ lying immediately to the left of the line $x=pq$. On the interior of $C_N$, $F(z)$ has a simple pole at $z=0$.  At $n=1,2, \dots , pq-1$, $F(z)$ also has poles. If $n\neq kp, kq$, for any integer $k$, then the poles are simple, since $(p,q)=1$.  Suppose that $n=kp, 1\leq k\leq q-1$. Then $G(n,\chi_1)=0$.  Hence, $kp$ is a simple pole. Similarly, if $n=kq, 1\leq k\leq p-1$, then $kq$ is also a simple pole.

Let $R_{\alpha}$ denote the residue of $F(z)$ at a pole $z=\alpha$.  First,
\begin{align}\label{0}
R_0=& \lim_{z\to 0}\df{G(z,\chi_1)}{G(\chi_1)(e^{2\pi iz}-1)}\df{G(z,\chi_2)}{G(\chi_2)\sin(\pi z/q)}
\df{z}{\sin(\pi z/p)}
 \df{\sin^{a_1}(b_1\pi z/p)}{\sin^{a_1}(\pi z/p)}\df{\sin^{a_2}(b_2\pi z/q)}{\sin^{a_2}(\pi z/q)}\notag\\
=&\df{\sum_{j=1}^{p-1}\chi_1(j)(2\pi ij/p)}{2\pi iG(\chi_1)}
\df{\sum_{j=1}^{q-1}\chi_2(j)(2\pi ij/q)}{(\pi/q)G(\chi_2)}\df{1}{(\pi/p)}
 \df{(b_1\pi/p)^{a_1}}{(\pi/p)^{a_1}} \df{(b_2\pi/q)^{a_2}}{(\pi/q)^{a_2}}\notag\\
 =&-\df{4}{2\pi i} \df{M(p,\chi_1)}{G(\chi_1)}\df{M(q,\chi_2)}{G(\chi_2)}b_1^{a_1}b_2^{a_2},
\end{align}
where we used \eqref{M}.

Second, if $n\neq kp,kq$, $0<n\leq pq-1$, then by \eqref{factorization},
\begin{equation}\label{n}
R_n=\df{\chi_1(n)\chi_2(n)}{2\pi i}
\df{\sin^{a_1}(b_1\pi n/p)}{\sin^{a_1+1}(\pi n/p)}\df{\sin^{a_2}(b_2\pi n/q)}{\sin^{a_2+1}(\pi n/q)}.
\end{equation}

Third, let $n=kp$, $1\leq k\leq q-1$. Then
\begin{align}\label{00}
R_{kp}=& \lim_{z\to kp}\df{G(z,\chi_1)}{G(\chi_1)\sin(\pi z/p)}\df{G(z,\chi_2)}{G(\chi_2)}\df{z-kp}{e^{2\pi iz}-1}
\df{\sin^{a_1}(b_1\pi z/p)}{\sin^{a_1}(\pi z/p)}\df{\sin^{a_2}(b_2\pi z/q)}{\sin^{a_2+1}(\pi z/q)}\notag\\
\end{align}
Now,
\begin{align}\label{1}
\lim_{z\to kp}\df{G(z,\chi_1)}{G(\chi_1)\sin(\pi z/p)}=&\lim_{z\to kp}
\df{\sum_{j=1}^{p-1}\chi_1(j)(2\pi ij/p)e^{2\pi ijz/p}}{G(\chi_1)(\pi/p)\cos(\pi z/p)},\notag\\
=&2i(-1)^k\df{M(p,\chi_1)}{G(\chi_1)},
\end{align}
by \eqref{M}.  Next, by \eqref{factorization},
\begin{equation}\label{2}
\lim_{z\to kp}\df{G(z,\chi_2)}{G(\chi_2)}\df{\sin^{a_2}(b_2\pi z/q)}{\sin^{a_2+1}(\pi z/q)}=
\chi_2(p)\chi_2(k)\df{\sin^{a_2}(b_2\pi kp/q)}{\sin^{a_2+1}(\pi kp/q)}.
\end{equation}
Next,
\begin{equation}\label{3}
\lim_{z\to kp}\df{\sin^{a_1}(b_1\pi z/p)}{\sin^{a_1}(\pi z/p)}=(-1)^kb_1^{a_1},
\end{equation}
since $b_1$ is even.  Lastly,
\begin{equation}\label{4}
\lim_{z\to kp}\df{z-kp}{e^{2\pi iz}-1}=\df{1}{2\pi i}.
\end{equation}
Gathering \eqref{1}--\eqref{4} and putting them in \eqref{00}, we conclude that
\begin{equation}\label{5}
R_{kp}=\df{b_1^{a_1}M(p,\chi_1)}{\pi \,G(\chi_1)}\chi_2(p)\chi_2(k)\df{\sin^{a_2}(b_2\pi kp/q)}{\sin^{a_2+1}(\pi kp/q)}, \qquad 1\leq k \leq q-1.
\end{equation}
By analogy,
\begin{equation}\label{6}
R_{kq}=\df{b_2^{a_2}M(q,\chi_2)}{\pi \,G(\chi_2)}\chi_1(q)\chi_1(k)\df{\sin^{a_1}(b_1\pi kq/p)}{\sin^{a_1+1}(\pi kq/p)}, \qquad 1\leq k \leq p-1.
\end{equation}

Finally, using \eqref{0}, \eqref{n}, \eqref{5}, and \eqref{6}, and the residue theorem, we conclude that
\begin{align}
\int_{C_N}F(z)dz=&-4 \df{M(p,\chi_1)}{G(\chi_1)}\df{M(q,\chi_2)}{G(\chi_2)}b_1^{a_1}b_2^{a_2}+\sum_{\substack{0<n<pq-1\\n\neq kp,kq}}
\chi_1(n)\chi_2(n)
\df{\sin^{a_1}(b_1\pi n/p)}{\sin^{a_1+1}(\pi n/p)}\df{\sin^{a_2}(b_2\pi n/q)}{\sin^{a_2+1}(\pi n/q)}\notag\\
&+2i\df{b_1^{a_1}M(p,\chi_1)}{G(\chi_1)}\chi_2(p)\sum_{k=1}^{q-1}\chi_2(k)\df{\sin^{a_2}(b_2\pi kp/q)}{\sin^{a_2+1}(\pi kp/q)}\notag\\
&+2i\df{b_2^{a_2}M(q,\chi_2)}{G(\chi_1)}\chi_1(q)\sum_{k=1}^{p-1}\chi_1(k)\df{\sin^{a_1}(b_1\pi kq/p)}{\sin^{a_1+1}(\pi kq/p)}\notag\\
=:&I(a_1,a_2,b_1,b_2,p,q,\chi_1,\chi_2).\label{residuetheorem}
\end{align}

We now calculate the integral on the left-hand side of \eqref{residuetheorem} directly. First observe from the definition \eqref{F} that because $a_1$ and $a_2$ are odd and $b_1$ and $b_2$ are even, the integrals over the two vertical sides of $C_N$ cancel.

 We  examine the integral over the top side, $C_{NT}$, of $C_N$.  Let $z=x+iN$, $0\leq x\leq pq$,  and set
\begin{equation*}
\mu=e^{-\pi i x/(pq)}e^{\pi N/(pq)} .
\end{equation*}
 We express the functions comprising $F(z)$ in terms of $\mu$.
  First,
  \begin{equation}\label{15aa}
  G(z,\chi_1)=\sum_{n=0}^{p-1}\chi_1(n)\mu^{-2q n}\quad\text{and}\quad G(z,\chi_2)=\sum_{n=0}^{q-1}\chi_2(n)\mu^{-2p n}.
\end{equation}
Second, since $a_1$ is odd,
\begin{align}
\sin^{a_1}(b_1\pi z/p)&=\left(\df{\mu^{-b_1q}-\mu^{b_1q}}{2i}\right)^{a_1}
=-\df{1}{(2i)^{a_1}}\mu^{a_1b_1q}\sum_{m=0}^{a_1}(-1)^m\binom{a_1}{m}\mu^{-2b_1qm}.\label{15}
\end{align}
Analogously,
\begin{align}
\sin^{a_2}(b_2\pi z/q)&=\left(\df{\mu^{-b_2p}-\mu^{b_2p}}{2i}\right)^{a_2}
 \notag\\
&=-\df{1}{(2i)^{a_2}}\mu^{a_2b_2p}\sum_{m=0}^{a_2}(-1)^m\binom{a_2}{m}\mu^{-2b_2pm}.\label{15aaa}
\end{align}
Third, since $a_1$ is odd,
\begin{align}
\sin^{-a_1-1}(\pi z/p)=&\left(\df{\mu^{-q}-\mu^q}{2i}\right)^{-a_1-1}\notag\\
=&(2i)^{a_1+1}\mu^{-(a_1+1)q}(1-\mu^{-2q})^{-a_1-1}\notag\\
=&(2i)^{a_1+1}\mu^{-(a_1+1)q}\sum_{r=0}^{\infty}(-1)^r\binom{-(a_1+1)}{r}\mu^{-2qr}\notag\\
=&(2i)^{a_1+1}\mu^{-(a_1+1)q}\sum_{r=0}^{\infty}\binom{(a_1+1)-1+ r}{r}\mu^{-2qr},\label{16}
\end{align}
and, analogously,
\begin{align}
\sin^{-a_2-1}(\pi z/q)=&\left(\df{\mu^{-p}-\mu^p}{2i}\right)^{-a_2-1}\notag\\
=&(2i)^{a_2+1}\mu^{-(a_2+1)p}\sum_{r=0}^{\infty}\binom{(a_2+1)-1+ r}{r}\mu^{-2pr}.\label{16aa}
\end{align}
Fourth,
\begin{align}\label{18}
\df{1}{e^{2\pi iz}-1} = -1+O(\mu^{-2pq}),
\end{align}
as $\mu\to\infty$.
 Using \eqref{15aa}--\eqref{18}, we deduce that
\begin{align}\label{19}
F(z)=&F(x+iN)=\left\{\df{2i}{G(\chi_1)}\mu^{q(a_1b_1-a_1-1)}
\sum_{n_1=1}^{p-1}\df{\chi_1(n_1)}{\mu^{2qn_1}}\right.
\sum_{m_1=0}^{a_1}(-1)^{m_1}\binom{a_1}{m_1}\mu^{-2b_1qm_1}\notag\\
&\times\left.\sum_{r_1=0}^{\infty}\binom{a_1+r_1}{r_1}\mu^{-2qr_1}\right\}\notag\\
&\times\left\{\df{2i}{G(\chi_2)}\mu^{p(a_2b_2-a_2-1)}\sum_{n_2=1}^{q-1}\df{\chi_2(n_2)}{\mu^{2pn_2}}
\sum_{m_2=0}^{a_2}(-1)^{m_2}\binom{a_2}{m_2}\mu^{-2pb_2m_2}\right.\notag\\
&\times\left.\sum_{r_2=0}^{\infty}\binom{a_2+r_2}{r_2}\mu^{-2pr_2}\right\}(-1+O(\mu^{-2pq})),
\end{align}
as $\mu$ tends to infinity.
For brevity, we write
\begin{equation*}
F(z)=F(x+iN)=A_D(x)e^{D\pi N/(pq)}+A_{D-1}(x)e^{(D-1)\pi N/(pq)}+\cdots+A_0(x)+O(e^{-\pi N/(pq)}),
\end{equation*}
where $A_D(x),A_{D-1}(x),\dots, A_1(x), A_0(x)$ are functions of $x$, and
\begin{equation}\label{59}
D=q(a_1b_1-a_1-1)-2p+p(a_2b_2-a_2-1)-2q.
\end{equation}
Thus,
\begin{equation}\label{60}
\int_{C_{NT}}F(z)dz =B_De^{D\pi N/(pq)}+B_{D-1}e^{(D-1)\pi N/(pq)}+\cdots+B_0+O(e^{-\pi N/(pq)}),
\end{equation}
where $B_D, B_{D-1}, \dots, B_1, B_0$ are constants.

The calculation of the integral over $C_{NB}$, the lower edge of $C_N$, follows along the same lines as above.  Set $z=x-iN$ and define
\begin{equation*}
\mu=e^{\pi i x/(pq)}e^{\pi N/(pq)} .
\end{equation*}
Note that \eqref{18} needs to be replaced by
\begin{equation*}
\df{1}{e^{2\pi iz}-1}=\mu^{-2pq}+O(\mu^{-4pq}), \qquad |\mu|\to\infty.
\end{equation*}
Then,
\begin{align}\label{19a}
F(z)=&F(x-iN)=\left\{\df{2i}{G(\chi_1)}\mu^{q(a_1b_1-a_1-1)}
\sum_{n_1=1}^{p-1}\df{\chi_1(n_1)}{\mu^{-2qn_1}}\right.
\sum_{m_1=0}^{a_1}(-1)^{m_1}\binom{a_1}{m_1}\mu^{-2b_1qm_1}\notag\\
&\times\left.\sum_{r_1=0}^{\infty}\binom{a_1+r_1}{r_1}\mu^{-2qr_1}\right\}\notag\\
&\times\left\{\df{2i}{G(\chi_2)}\mu^{p(a_2b_2-a_2-1)}\sum_{n_2=1}^{q-1}\df{\chi_2(n_2)}{\mu^{-2pn_2}}
\sum_{m_2=0}^{a_2}(-1)^{m_2}\binom{a_2}{m_2}\mu^{-2pb_2m_2}\right.\notag\\
&\times\left.\sum_{r_2=0}^{\infty}\binom{a_2+r_2}{r_2}\mu^{-2pr_2}\right\}(\mu^{-2pq}+O(\mu^{-4pq})),
\end{align}
as $\mu$ or $N$ tends to infinity.
Again, for brevity, write
\begin{equation*}
F(z)=F(x-iN)=A_H^{\prime}(x)e^{H\pi N/(pq)}+A_{H-1}^{\prime}(x)e^{(H-1)\pi N/(pq)}+\cdots+A_0^{\prime}(x)+O(e^{-\pi N/(pq)}),
\end{equation*}
where $A_H^{\prime}(x), A_{H-1}^{\prime}(x),\dots, A_0^{\prime}(x)$ are functions of $x$, and
\begin{align}\label{59a}
H=&q(a_1b_1-a_1-1)+2q(p-1)+p(a_2b_2-a_2-1)+2p(q-1)-2pq\notag\\
=&q(a_1b_1-a_1-1)-2p+p(a_2b_2-a_2-1)-2q+2pq.
\end{align}
From \eqref{59} and \eqref{59a}, we observe that $H=D+2pq$.
Thus,
\begin{equation}\label{60a}
\int_{C_{NB}}F(z)dz =B^{\prime}_He^{H\pi N/(pq)}+B_{H-1}^{\prime}e^{(H-1)\pi N/(pq)}+\cdots+B_0^{\prime}
+O(e^{-\pi N/(pq)}),
\end{equation}
where $B_H^{\prime}, B_{H-1}^{\prime},\dots, B_0^{\prime}$ are constants.
Define $B_H=B_{H-1}=\cdots= B_{D+1}=0$.

In summary, by \eqref{60} and \eqref{60a}, we have shown that
\begin{align}\label{61}
\int_{C_N}F(z)dz=&(B_H+B_H^{\prime})e^{H\pi N/(pq)}+(B_{H-1}+B_{H-1}^{\prime})e^{(H-1)\pi N/(pq)}\notag\\
&+\cdots+(B_0+B_0^{\prime})+O(e^{-\pi N/(pq)}),
\end{align}
as $N\to\infty$.
Combining \eqref{61} and \eqref{residuetheorem} and recalling the definition \eqref{residuetheorem},  we deduce that
\begin{gather}
I(a_1,a_2,b_1,b_2,p,q,\chi_1,\chi_2)\notag\\=(B_H+B_H^{\prime})e^{H\pi N/(pq)}
+(B_{H-1}+B_{H-1}^{\prime})e^{(H-1)\pi N/(pq)}
+\cdots+(B_0+B_0^{\prime})+O(e^{-\pi N/(pq)}).\label{62}
\end{gather}
Since the right-hand side of \eqref{62} is independent of $N$, in succession we conclude that
\begin{equation}\label{63}
B_H+B_H^{\prime}=B_{H-1}+B_{H-1}^{\prime}=\cdots=B_1+B_1^{\prime}=0.
\end{equation}
Hence, from \eqref{62} and \eqref{63},
\begin{equation*}
I(a_1,a_2,b_1,b_2,p,q,\chi_1,\chi_2)=B_0+B_0^{\prime}.
\end{equation*}

Returning to \eqref{19} and \eqref{60}, we observe that $B_0=-pqC_0$, where
\begin{align}\label{64}
C_0=&\df{4}{G(\chi_1)G(\chi_2)}\sum_{\substack{n_1,m_1,r_1,n_2,m_2,r_2\geq0\\
q(a_1b_1-a_1-1)-2qn_1-2qb_1m_1-2qr_1\\+p(a_2b_2-a_2-1)-2pn_2-2pb_2m_2-2pr_2=0}}\notag\\
&\left\{\mu^{q(a_1b_1-a_1-1)}
\sum_{n_1=1}^{p-1}\df{\chi_1(n_1)}{\mu^{2qn_1}}\right.
\sum_{m_1=0}^{a_1}(-1)^{m_1}\binom{a_1}{m_1}\mu^{-2b_1qm_1}\notag\\
&\times\left.\sum_{r_1=0}^{\infty}\binom{a_1+r_1}{r_1}\mu^{-2qr_1}\right\}\notag\\
&\times\left\{\mu^{p(a_2b_2-a_2-1)}\sum_{n_2=1}^{q-1}\df{\chi_2(n_2)}{\mu^{2pn_2}}
\sum_{m_2=0}^{a_2}(-1)^{m_2}\binom{a_2}{m_2}\mu^{-2pb_2m_2}\right.\notag\\
&\times\left.\sum_{r_2=0}^{\infty}\binom{a_2+r_2}{r_2}\mu^{-2pr_2}\right\}.
\end{align}
Returning to \eqref{19a} and \eqref{60a}, we furthermore observe that $B_0^{\prime}=pqC_0^{\prime}$, where
\begin{align}\label{65}
C_0^{\prime}=&-\df{4}{G(\chi_1)G(\chi_2)}\sum_{\substack{n_1,m_1,r_1,n_2,m_2,r_2\geq0\\
q(a_1b_1-a_1-1)+2qn_1-2 qb_1m_1-2qr_1\\+p(a_2b_2-a_2-1)+2pn_2-2pb_2m_2-2pr_2-2pq=0}}\notag\\
&\left\{\mu^{q(a_1b_1-a_1-1)}
\sum_{n_1=1}^{p-1}\df{\chi_1(n_1)}{\mu^{-2qn_1}}\right.
\sum_{m_1=0}^{a_1}(-1)^{m_1}\binom{a_1}{m_1}\mu^{-2b_1qm_1}\notag\\
&\times\left.\sum_{r_1=0}^{\infty}\binom{a_1+r_1}{r_1}\mu^{-2qr_1}\right\}\notag\\
&\times\left\{\mu^{p(a_2b_2-a_2-1)}\sum_{n_2=1}^{q-1}\df{\chi_2(n_2)}{\mu^{-2pn_2}}
\sum_{m_2=0}^{a_2}(-1)^{m_2}\binom{a_2}{m_2}\mu^{-2pb_2m_2}\right.\notag\\
&\times\left.\sum_{r_2=0}^{\infty}\binom{a_2+r_2}{r_2}\mu^{-2pr_2}\right\}\mu^{-2pq}.
\end{align}
In summary,
\begin{equation}\label{atlast}
I(a_1,a_2,b_1,b_2,p,q,\chi_1,\chi_2)=pq(-C_0+C_0^{\prime}),
\end{equation}
where $C_0$ and $C_0^{\prime}$ are given by \eqref{64} and \eqref{65}, respectively.

Lastly, in our calculations of both $C_0$ and $C_0^{\prime}$, we ignored the big-O terms arising from expanding
\begin{equation}\label{67}
\df{1}{e^{2\pi iz}-1}
\end{equation}
in powers of $\mu$.  In each case, to do this, we require that
\begin{equation}\label{66}
q(a_1b_1-a_1-3)+p(a_2b_2-a_2-3)<2pq.
\end{equation}

If we combine \eqref{atlast} with \eqref{residuetheorem}, we arrive at \eqref{chieftheorem}.

\end{proof}

    In our proof, we used only the first term when expanding \eqref{67} in a geometric series.  If we employed additional terms, this would enable us to weaken the hypothesis \eqref{66}, \eqref{66a} and prove a stronger theorem.

One could also prove a theorem involving a sum of products of $n$ sines.  The `evaluation' would be in terms of sums with the quotients of $1,2, \dots, n-1$ sines in the summands.  One of the technical difficulties in the proof would be the calculation of the residue at $z=0$.   Also, there would be poles at integers with different numbers of prime factors in them, corresponding to sums with different numbers of sine products in their summands.

\end{document}